\documentclass[10pt]{amsart}
\usepackage[dvips]{epsfig}
\usepackage{amsmath,amssymb,bbm}
\usepackage[all, knot]{xy}
\SelectTips{cm}{}
\xyoption{arc}
\xyoption{2cell}
\UseAllTwocells

\newcommand \driicell [1] {\drtwocell \omit {#1}}

\newcommand \drriicell [1] {\drrtwocell \omit {#1}}
\newcommand \uriicell [1] {\urtwocell \omit {#1}}

\newtheorem{theorem}{Theorem}
\newtheorem{lemma}{Lemma}
\newtheorem{varexample}{Example}
\newenvironment{example}{\begin{varexample}\em}{\em\end{varexample}}
\newtheorem{definition}{Definition}
\newtheorem{varremark}{Remark}
\newenvironment{remark}{\begin{varremark}\em}{\em\end{varremark}}

\newcommand{\ra}{\mathop{\rightarrow}}
\newcommand{\ralim}{\mathop{\ra}\limits}
\newcommand{\la}{\mathop{\leftarrow}}
\newcommand{\lalim}{\mathop{\la}\limits}

\newcommand{\longrightarrowlim}{\mathop{\longrightarrow}\limits}

\newcommand{\opname}[1]{\operatorname{#1}}

\newcommand{\catname}[1]{\boldsymbol{\opname{{#1}}}}
\newcommand{\inprod}[1]{\left\langle{#1}\right\rangle}

\newcommand{\C}{\catname{C}}

\newcommand{\X}{\catname{X}}
\newcommand{\Y}{\catname{Y}}
\newcommand{\Cop}{\catname{C}^{\opname{op}}}

\newcommand{\Cat}{\catname{Cat}}
\newcommand{\Set}{\catname{Set}}
\newcommand{\Bicat}{\catname{Bicat}}

\newcommand{\Hor}{\catname{Hor}}
\newcommand{\Ver}{\catname{Ver}}
\newcommand{\Squ}{\catname{Squ}}

\newcommand{\iiCob}{\catname{2Cob}}
\newcommand{\nCobi}{\catname{nCob}}
\newcommand{\nCob}{\catname{nCob}_2}

\newcommand{\Cspan}{\opname{Span}(\catname{C})}
\newcommand{\Cosp}{\opname{Cosp}}
\newcommand{\CCosp}{\opname{Cosp}(\catname{C})}
\newcommand{\iiCCosp}{\opname{2Cosp}(\catname{C})}

\newcommand{\Ob}{\catname{Obj}}
\newcommand{\M}{\catname{Mor}}
\newcommand{\B}{\catname{2Mor}}

\newcommand{\V}{\catname{Vect}}
\newcommand{\iiV}{\catname{2Vect}}
\newcommand{\Hilb}{\catname{Hilb}}
\newcommand{\iiH}{\catname{2Hilb}}

\newcommand{\CG}{\mathbbm{C}[G]}
\newcommand{\ZCG}{Z(\mathbbm{C}[G])}
\newcommand{\VG}{\catname{Vect}[G]}
\newcommand{\ZVG}{Z(\catname{Vect}[G])}

\newcommand{\db}{double bicategory}
\newcommand{\dbs}{double bicategories}
\newcommand{\vdb}{Verity double bicategory}
\newcommand{\vdbs}{Verity double bicategories}

\newcommand{\Obj}{\opname{Obj}}
\newcommand{\Mor}{\opname{Mor}}
\newcommand{\Path}{\opname{Path}}

\newcommand{\id}{\opname{id}}
\newcommand{\mathd}{\mathrm{d}}

\newcommand{\br}[1]{\langle {#1} \rangle}

\newcommand{\fc}[1]{[ \Pi_1({#1}) , G ]}
\newcommand{\Z}[1]{\bigl{[} \fc{#1} , \V \bigr{]}}

\begin{document}


\title{Extended TQFT's and Quantum Gravity}
\author{Jeffrey Morton}
\address{Mathematics Department, University of Western Ontario}
\email{\tt{jmorton9@uwo.ca}}

\begin{abstract}
\addcontentsline{toc}{chapter}{Abstract}
This paper gives a definition of an extended topological quantum
field theory (TQFT) as a weak 2-functor $Z: \nCob \ra \iiV$, by
analogy with the description of a TQFT as a functor $Z: \nCobi \ra
\V$.  We also show how to obtain such a theory from any finite group
$G$.  This theory is related to a topological gauge theory, the
Dijkgraaf-Witten model.  To give this definition rigorously, we first
define a bicategory of cobordisms between cobordisms.  We also give
some explicit description of a higher-categorical version of $\V$,
denoted $\iiV$, a bicategory of \textit{2-vector spaces}.  Along the
way, we prove several results showing how to construct 2-vector spaces
of \textit{$\V$-valued presheaves} on certain kinds of groupoids.  In
particular, we use the case when these are groupoids whose objects are
connections, and whose morphisms are gauge transformations, on the
manifolds on which the extended TQFT is to be defined.  On cobordisms
between these manifolds, we show how a construction of ``pullback and
pushforward'' of presheaves gives both the morphisms and 2-morphisms
in $\iiV$ for the extended TQFT, and that these satisfy the axioms for
a weak 2-functor.  Finally, we discuss the motivation for this
research in terms of Quantum Gravity.  If the results can be extended
from a finite group $G$ to a Lie group, then for some choices of $G$
this theory will recover an existing theory of Euclidean quantum gravity
in 3 dimensions.  We suggest extensions of these ideas which may be
useful to further this connection and apply it in higher dimensions.
\end{abstract}
\maketitle             
\pagebreak
\tableofcontents
\pagebreak

\section{Introduction}

In this paper, I will describe a connection between the ideas of
extended topological quantum field theory and topological gauge
theory.  This is motivated by consideration of a possible application
to quantum gravity, and in particular in 3 dimensions--a situation
which is simpler than the more realistic 4D case but has many of the
essential features.  Here, we consider this example as
related to one interesting case of a general formulation of
``Extended'' TQFT's.  This is described in terms of higher category
theory.

The idea that category theory could play a role in clarifying problems
in quantum gravity seems to have been first expressed by Louis Crane
\cite{crane}, who coined the term ``categorification'' .
Categorification is a process of replacing set-based concepts by
category-based concepts.  Categories are structures which have not
only elements (that is, \textit{objects}), but also arrows, or
\textit{morphisms} between objects as logically primitive concepts.
In many examples of categories, the morphisms are functions or
relations between the objects, though this is not always the case.
Categorification therefore is the reverse of a process of
\textit{decategorification} which involves discarding the structure
encoded in morphisms.  A standard example is the semiring of natural
numbers $\mathbbm{N}$, which can be seen as a decategorification of
the category of finite sets with set functions as arrows, since each
natural number can be thought of as an isomorphism class of finite
sets.  The sum and product on $\mathbbm{N}$ correspond to the
categorical operations of coproduct (disjoint union) and product
(cartesian product), which have purely arrow-based descriptions.  For
some further background on the concept of categorification, see work
by Crane and Yetter \cite{cycategorification}, or Baez and Dolan
\cite{categorification}.

So what we study here are categorified topological quantum field
theories (TQFT's).  The program of applying categorical notions to
field theories was apparently first described by Dan Freed
\cite{freed}, who referred to them as ``higher algebraic'' structures.
The motivation for doing this is that this framework appears to allow
us to find a new way of obtaining a known theory of quantum gravity in
3 dimensions---the Ponzano-Regge model---as a special case.  Moreover
what we recover is not just to the vacuum version of this 3D quantum
gravity---what one could expect from an ordinary TQFT---but to a form
in which spacetime contains matter.

To categorify the notion of a TQFT, we use the fact that a TQFT can be
described, in the language of category theory, as a \textit{functor}
from a category of \textit{cobordisms}---which is topological in
character--- into the category of Hilbert spaces.  To ``categorify''
this means to construct an analogous theory in the language of higher
categories---in particular, 2-categories.  One of the obstacles to
doing this is that one needs to have a suitable 2-category analogous
to the category of cobordisms, to represent structures such as the one
in Figure \ref{fig:cobcorners}.

A cobordism from a manifold $S$ to another manifold $S'$ is a manifold
with boundary $M$ such that $\partial M$ is the disjoint union of $S$
and $S'$, which we think of as an arrow $M : S \rightarrow S'$.  One
can define composition of cobordisms, by gluing along components of
the boundary, leading to the definition of a category $\nCobi$ of
$n$-dimensional cobordisms between $(n-1)$-dimensional manifolds.

\begin{figure}[h]
\begin{center}
\includegraphics{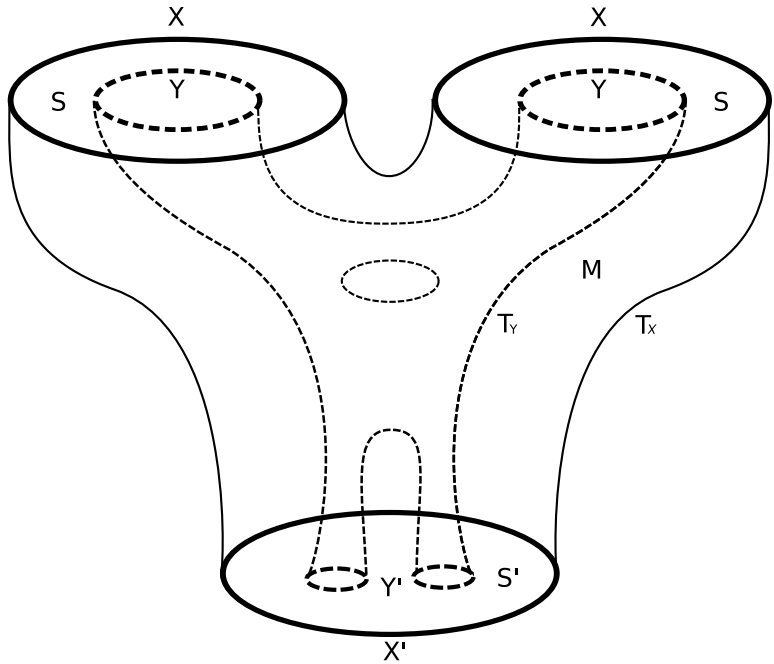}
\end{center}
\caption{\label{fig:cobcorners}A Cobordism With Corners}
\end{figure}

In Figure \ref{fig:cobcorners} we see a 3-manifold with corners which
illustrates these points and provides some motivating intuition.  This
can be seen a cobordism from the pair of annuli at the top to the
two-punctured disc at the bottom.  These in turn can be thought of,
respectively, as cobordisms from one pair of circles to another, and
from one circle to two circles.  The large cobordism has other
boundary components: the outside boundary is itself a cobordism from
two circles to one circle; the inside boundary (in dotted lines) is a
cobordism from one pair of circles to another pair.  We could
``compose'' this with another such cobordism with corners by gluing
along any of the four boundary components: top or bottom, inside or
outside.  This involves attaching another such cobordism along
corresponding boundary components by a diffeomorphism.  These
components are themselves manifolds with boundary, and ``gluing'' is
accomplished by specifying a diffeomorphism between them, fixing their
own boundaries.  Furthermore, as the Figure suggests, we can do such a
composition in either a ``vertical'' direction, gluing at $S$ or $S'$,
or a ``horizontal'' direction, gluing at $T_X$ or $T_Y$.

We want to define an ``extended TQFT'', which assigns higher
algebraic data to the manifolds, cobordisms, and cobordisms with
corners in this setting.  One necessary preliminary for the example we
are interested in is a description of topological quantum field
theories in the usual sense.  This is reviewed in Chapter
\ref{chap:TQFT}, beginning in Section \ref{sec:ncob}.  Atiyah's
axiomatic description of TQFTs \cite{atiyah}, reviewed in Section
\ref{sec:tqftfunctor}), can be interpreted as defining TQFT's as
functors from a category of \textit{cobordisms} into $\V$:
\begin{equation}\label{eq:introtqft}
Z : \nCobi \rightarrow \V
\end{equation}
Where $\nCobi$ has $(n-1)$-dimensional manifolds its objects and
$n$-dimensional cobordisms as its morphisms.  A TQFT assigns a space
of states to each manifold, and a linear transformation between states
to cobordisms.

Section \ref{sec:fhk} discusses a construction due to Fukuma, Hosono,
and Kawai \cite{FHK} for constructing a TQFT explicitly in dimension
$n=2$ starting from any finite group $G$.  The FHK construction is an
example of how this quantum theory intimately involves a relation
between smooth and discrete geometric structures.  Specifically, this
topological theory can be thought of as coming from structures built
on manifolds and cobordisms via a triangulation---a decomposition of
the manifold into simplices.  It turns out that there is a close
connection between the ideas of a theory having ``no local degrees of
freedom'' in the discrete and continuum setting.  In the continuum
setting, this means that the theory is \textit{topological}---the
vector spaces and linear operators it assigns depending only on the
isomorphism class of the manifold or cobordism.  In the discrete
setting of a triangulated manifold, it means that the theory is
\textit{triangulation independent}

An important feature of a TQFT constructed this way is that it assigns
to a closed, connected 1-manifold (i.e. a circle) just some element of
the centre of the group algebra of $G$, denoted $\ZCG$.  A
standard interpretation of such a space in quantum theory would hold
that this is a quantization of a classical space of states.  The
classical space would then simply be $\CG$, so that quantum states
are (complex) linear combinations of classical states.  An assignment
of a group element to a circle, or loop, can be interpreted as a
\textit{connection} on the circle.  Then $\CG$ consists of
complex-valued linear combinations (``superpositions'') of such
connections.  The centre, $\ZCG$, consists of such superpositions
which commute with any element of $G$ (and hence of $\CG$).  These are
thus invariant under conjugation by any element of $G$.  Such a
conjugation is a ``gauge transformation'' of a connection - so these
elements are \textit{gauge invariant superpositions} of connections.

These interpretations turn out to be useful when we aim to produce
\textit{extended TQFT's}.  This notion was described by Ruth Lawrence
\cite{lawrence}.  These are theories similar to TQFT's, for which the
theory is defined not on cobordisms, but on manifolds with corners.
One setting where this arises is if we consider the possibility of
manifolds \textit{with boundary} connected by a cobordism.  In
particular, we are interested in the case where $S: X \rightarrow Y$
and $S' : X' \rightarrow Y'$ are already themselves cobordisms.  These
cobordisms between cobordisms, then, are manifolds with corners.  Here
we shall present a formalism for describing the ways such cobordisms
can be glued together.  Louis Crane has written a number of papers on
this issue, including one with David Yetter \cite{CY} which gives a
\textit{bicategory} of such cobordisms.  We want to define a structure 
$\nCob$, whose objects are $(n-2)$-manifolds, whose morphisms are
$(n-1)$-cobordisms, and whose 2-morphisms are $n$-cobordisms with
corners.  Just as a TQFT assigns a space of states to a manifold and a
linear map to a cobordisms, an extended TQFT will assign some such
algebraic data to $(n-2)$-manifolds, $(n-1)$-manifolds with boundary,
and $n$-dimensional manifolds with corners.  This data should have an
interpretation similar to that for a TQFT.

To clarify how to do this, we need to consider more carefully what
kind of structure $\nCob$ must be.  So we consider some background on
higher category theory.  This field of study is still developing, but
there are good introductions by Leinster \cite{leinster} and by Cheng
and Lauda \cite{chenglauda}.  The essential idea of higher category
theory is that as well as objects (represented in diagrams as
zero-dimensional), and morphisms (or arrows) connecting them (which
are one-dimensional), there also should be morphisms represented
by ``cells'' of two, three, or even higher dimensions, connecting
lower-dimensional morphisms.  For our purposes here, we only need to
consider higher categories with morphisms represented by at most
2-dimensional cells.  Chapter \ref{chap:bianddbl} discusses
bicategories and double categories, which we will generalize later,
and briefly describes some standard examples of these from homotopy
theory.

Whereas a category has objects and morphisms between objects, a
bicategory will have an extra layer of structure: objects, morphisms
between objects, and 2-morphisms between morphisms:
\begin{equation}\label{xy:2morphism}
  \xymatrix{
    x \ar@/^1pc/[r]^{f}="0" \ar@/_1pc/[r]_{g}="1" & y \\
      \ar@{=>}"0"+<0ex,-2ex> ;"1"+<0ex,+2ex>^{\alpha}
  }
\end{equation} The ``strict'' form of a bicategory is a 2-category,
which are reviewed by Kelly and Street \cite{kellystreet}, but we are
really interested in the weak forms---here, all the axioms which must
be satisfied by a category hold only ``up to'' certain
higher-dimensional morphisms.  That is, what had been equations are
replaced by specified 2-isomorphisms, which then must themselves
satisfy certain equations called coherence conditions.  Such coherence
conditions have been a persistent theme of category theory since its
inception by MacLane and Eilenberg (see, for instance,
\cite{maclane}), and are important features of higher categorical
structures.

Double categories, introduced by Ehresmann \cite{ehresmann1}
\cite{ehresmann2}, may be seen as ``internal'' categories in $\Cat$.
That is, a double category is a structure with a category of objects
and a category of morphisms.  Less abstractly, it has objects,
horizontal and vertical morphisms which can be represented
diagrammatically as edges, and squares.  These can be composed in
geometrically obvious ways to give diagrams analogous to those in
ordinary category theory.  Our example of cobordisms with corners
appears to be an example of a double category: the objects are the
manifolds, the morphisms are the cobordisms, and the 2-cells are the
cobordisms with corners.  In fact, as we shall see, this is too strict
for our needs.

We note here that there are several relations between TQFT's and
extended TQFT's on the one hand, and higher categories on the other.
The categorical features of standard TQFT's are described in some
detail by Bruce Bartlett \cite{bartlett}.  Crane and Yetter \cite{CY}
describe the algebraic structure of TQFT's and extended TQFT's,
showing how certain algebraic and higher-algebraic structures are
implied in the definition of a TQFT.  Examples include the well known
equivalence between 2D TQFT's and Frobenius algebras; connections
between 3D TQFT's and either suitable braided monoidal categories, or
Hopf algebras; and the appearance of ``Hopf categories'' in 4D TQFT's.
These illustrate the move to higher-categorical structures in
higher-dimensional field theories.  Baez and Dolan \cite{hdatqft}
summarize the connection between TQFT's and higher category theory, in
the form of the \textit{Extended TQFT Hypothesis}, suggesting that all
extended TQFT's can be viewed as representations of a certain kind of
``free $n$-category''.

The kind of $n$-category we are interested in in this paper is a
common generalization of a double category and a bicategory.  Double
categories are too strict to be really natural for our purpose,
however---composition in a double category must be strictly
associative, and in order to achieve this, one only considers
equivalence classes of cobordisms, not cobordisms themselves, as
morphisms.  So we consider a weakening of this structure, in the sense
that axioms for a double category giving equations (such as
associativity) will be true only up to specified 2-morphisms.  This
allows us to take morphisms to be just cobordisms themselves, and the
diffeomorphisms between them as 2-morphisms.  This is analogous to the
way in which the idea of a bicategory is a weakening of the idea of a
category.

Bicategories, however, are not really what we want either, since we
want to describe systems with changing boundary conditions, and the
most natural way to do this is by thinking of both initial and final
states, and these changing conditions, as part of the boundary.  We
call the structure which accomplishes this a {\vdb}, referring to
Dominic Verity, who introduced them and called them simply {\dbs}.  On
the other hand, we show in Theorem \ref{thm:equiv} that {\vdbs}
satisfying certain conditions give rise to bicategories.  In fact
$\nCob$ is an example of this.  The structure we use to describe such
compositions is the one we call a \textit{\vdb}.  These we describe in
Chapter \ref{chap:doublebicat}.  In these, the composition
laws of double categories are weakened.  That is, the associativity of
composition, and unit laws, of the horizontal and vertical categories
apply only up to specified higher morphisms.

In Section \ref{sec:equiv} we prove that a {\vdb} satisfying certain
conditions gives a bicategory.  To finish Section
\ref{chap:doublebicat}, we describe a general class of examples of
{\vdbs}, analogous to the result that $\Cspan$ is a bicategory.  A
``span'' is a diagram of the form $A \la C \ra B$.  Given two spans $A
\la C \ra B$ and $A \la C' \ra B$, a span map is a morphism $f: C \ra
C'$ such that the diagram:
\begin{equation}
\xymatrix{
 &  C' \ar[dl] \ar[dr] & \\
A & C \ar[u]^{f} \ar[r] \ar[l] & B
}
\end{equation}
commutes.  A cospan is defined in the same way, but with the arrows
reversed.  It is a classical result of B\'enabou \cite{benabou} that
for any category $C$ which has all limits, there is a bicategory
$\Cspan$ whose objects are objects of $C$, whose morphisms are spans
in $C$, and whose 2-morphisms are span maps.  The composition of
morphisms is by pullback - a universal construction.  In Section
\ref{sec:dblspan}, a similar concept in 2 dimensions is introduced,
namely ``double spans'' and ``double cospans''.  These give a broad
class of examples of {\vdbs}, and in particular, we can use them to
derive the fact that there is a double bicategory of cobordisms with
corners.

To prove this fact, Theorem \ref{thm:maintheorem}, requires some
technical lemmas, which are put off until Appendix
\ref{sec:internalbicat}.  These extend some results about bicategories
and double categories, namely that a double category can be seen as an
\textit{internal} category in $\Cat$, and that \textit{spans} in a
category $\catname{C}$ with pullbacks constitute the morphisms of a
bicategory, $\Cspan$.  We show a way to describe {\dbs}, internal
bicategories in $\Bicat$, and that {\vdbs} are simply examples of
these which satisfy certain special conditions.  We also show that
\textit{double spans} most naturally form an example of a {\db}, but
that they can be reduced by taking isomorphism classes in order to
obtain a {\vdb}.

We describe more specifically the geometric framework for cobordisms
with corners in Chapter \ref{chap:cobcorn}.  Gerd Laures \cite{laures}
discusses the general theory of cobordisms of manifolds with corners.
In the terminology used there, introduced by J\"anich \cite{jan68},
what we primarily discuss in this work are $\br{2}$-manifolds.  This
describes the relation of ``faces'' of the manifold, but in particular
in this case it is related to the fact that the \textit{codimension}
of the manifold is 2.  That is, the manifold $M$ (whose dimension is
$\opname{dim}(M) = n$) will have a boundary $\partial M$, which will
in turn be composed of faces which are manifolds with boundary, of
dimension $(n-1)$.  However, the boundaries of these faces will be
closed manifolds: they are manifolds of dimension $(n-2)$.  This
separates into \textit{faces}.  For us, the faces decompose into
components, and the codimension-2 faces are the source and the target
in both horizontal and vertical directions.  We call the resulting
structure $\nCob$, and in Section \ref{sec:bicatcwc} we prove the main
result about $\nCob$, Theorem \ref{thm:maintheorem}, that this indeed
forms a {\vdb}.

In Chapter \ref{chap:2hilb} we turn to the next essential element of
an extended TQFT, the 2-category $\iiV$ of 2-vector spaces.  This is a
categorified analog of the category $\V$ of vector spaces.  There are
several alternative notions of what $\iiV$ should be---this is a
common feature of categorification, since the same structure may have
arisen by discarding structure in more than one way.  The view adopted
here is that a 2-vector space is a certain kind of
$\mathbbm{C}$-linear additive category.  The properties of being
\textit{$\mathbbm{C}$-linear} and \textit{additive} give analogs of
the linear structure of a vector space at both the object and morphism
levels.  $\mathbbm{C}$-linearity means that the set of morphisms are
complex vector spaces.  We should remark that these properties mean
that 2-vector spaces are closely related to abelian categories
(introduced by Freyd \cite{freyd}, and studied extensively as the general setting for
homological algebra) have a structure on objects which is similar to
addition for vectors.  In particular, we are interested in the
analog of ``finite dimensional'' vector spaces, so 2-vector spaces
also need to be \textit{finitely semi-simple}, so every object is a
finite sum of simple ones. 

Section \ref{sec:KV} describes Kapranov-Voevodsky (KV) 2-vector
spaces---the kind described above.  Each of these is equivalent to the
category $\V^n$ for some $n$ (a folklore theorem whose proof has been
difficult to find, so is presented here, along with some others).
Thus, KV vector spaces give a higher analog of complex vector spaces,
which are all equivalent to some $\mathbbm{C}^n$.  In fact, categories
with both $\mathbbm{C}$-linearity and additiveness naturally have a
kind of ``scalar'' multiplication by vector spaces.  So in the
categorified setting, the category $\V$ itself plays the role of
$\mathbbm{C}$ for complex vector spaces.  So Yetter's \cite{yet}
alternative definition of a 2-vector space as a $\V$-module turns out
to be equivalent to a KV vector space in the case where it is finitely
semisimple.

We describe the morphisms between KV 2-vector
spaces---\textit{2-linear maps}.  A 2-linear map $T: \V^n \ra \V^m$
can be represented as matrices of vector spaces:
\begin{equation}
\begin{pmatrix}
T_{1,1} & \dots & T_{1,n} \\
\vdots & & \vdots \\
T_{l,1} & \dots & T_{l,k} \\
\end{pmatrix}
\begin{pmatrix}
  V_1 \\ 
  \vdots \\
  V_k
\end{pmatrix}
\end{equation}
which act on 2-vectors by matrix multiplication, using the tensor
product $\otimes$ in the role of multiplication, and the direct sum
$\oplus$ in the role of addition.  All 2-morphisms between two such
2-linear maps can be represented as matrices of linear transformations
which act componentwise.  Proofs of these widely-known bits of
folklore are, again, difficult to find, so are presented here.

We also show that the concept of an \textit{adjoint} functor can be
described in terms of matrix representations of 2-linear maps in much
the same way that the description of the adjoint of a linear map
relates to its matrix representation.  So the two notions of
``adjoint'' turn out to be closely connected in 2-vector spaces.

A special example of a 2-vector spaces---a \textit{group
2-algebra}---is described.  This turns out to be the starting point to
describe what 2-vector space a 3-dimensional extended TQFT assigns to
a circle.

This example leads to discussion, in Section \ref{sec:kv2vsgrpd}, of
how to build 2-vector spaces from groupoids.  We introduce the concept
of ``$\V$-presheaves'' on $\X$.  These are just functors from
$\X^{op}$ to $\V$ (or equivalently, since $\X$ is a groupoid, just
from $\X$ to $\V$).  The totality of these functors forms a category,
which we call $[\X,\V]$, whose objects are functors from $\X$ to $\V$,
and whose morphisms are natural transformations between functors.  One
important result, Lemma \ref{lemma:fgfckv}, says that for any finite
groupoid $\X$ (or one which is ``essentially'' finite, in a precise
sense) the category $[\X,\V]$ is a KV 2-vector space.

Studying these $\V$-presheaves on groupoids is of interest, partly
because it opens up the possibility of a categorified version of
quantizing a system by taking the space of $L^2$ functions on its
classical configuration space.  This is a Hilbert space of
complex-valued functions on that space---so considering a 2-vector
space of $\V$-valued functions is a categorified analog.

On the other hand, $\Set$-valued presheaves on certain kinds of
categories are generic examples of toposes, about which much is known
(see, for example, Johnstone \cite{elephant1}, \cite{elephant2}).
Some results about these can be shown for $\V$-valued presheaves also,
although there are significant differences resulting from the fact
that $\V$ is an additive category, whereas $\Set$ is Cartesian.

A theorem for $\V$-valued presheaves which resembles one for
$\Set$ is that functors between groupoids give rise to ``pullback''
and ``pushforward'' 2-linear maps between these 2-vector spaces of
presheaves.  From a functor
\[
f : \X \ra \Y
\]
we get the ``pullback''
\[
f^{\ast}: [\Y,\V] \ra [\X,\V]
\]
and the ``pushforward''
\[
f_{\ast} : [\X,\V] \ra [\Y,\V]
\]
The pullback is easy to describe: a functor $F$ on $\Y$ gives a
functor $F \circ f$ on $\X$ by composition with $f$.  But the
pushforward depends on the structure of $\V$: as described in
Definition \ref{def:pushcolimit}, given a presheaf $V \in [\X,\V]$,
the pushforward $f_{\ast}V $ gives a presheaf in $[\Y,\V]$ which
gives, at any object $y$ in $\Y$, the colimit of a certain diagram.
This holds more famously for ordinary---that is,
$\Set$-valued---presheaves (see, e.g. \cite{macmoer}), where such a
collimit is simply the union of all the sets in the essential preimage
of the object $y$, modulo any relations imposed by the morphisms in
this essential preimage.  The intuition is that one ``adds up'' the
contributions from an entire preimage---but since there are
isomorphisms, this must be modified.  Similarly, for $\V$-valued
presheaves, the colimit is a coproduct (i.e. direct sum) of vector
spaces modulo similar relations.  Clearly, the ability to construct
this pushforward depends critically on the ability to take finite
colimits in $\V$.

Both the pullback and pushforward maps carry presheaves on one
groupoid to presheaves on another.  For a given $f$, the two 2-linear
maps $f^{\ast}$ and $f_{\ast}$ form an \textit{ambidextrous
adjunction}.  That is, $f_{\ast}$ is both a left and a right adjoint
to $f$, meaning in particular that for any presheaves $V \in [\X,\V]$
and $W \in [\Y,\V]$, we have both $\hom(V,f^{\ast}W) \cong
\hom(f_{\ast}V,W)$ and $\hom(f^{\ast}W,V) \cong \hom(W,f_{\ast}V)$.
We then say that they are adjoint 2-linear maps---this is an example
of the relationship between adjointness of functors and adjointness of
linear maps.

This pair of adjoint maps, the pullback and pushforward, turns out to
be essential to the constructions used to develop the extended TQFT's
we are interested in.  The reason is related to the fact that we
described the cobordisms on which they are defined in terms of
\textit{cospans}, as we will see shortly.

In Section \ref{sec:CY}, we fill out some of the details of what a
2-Hilbert space should be, including a definition of the inner
product, and an extension to infinite dimension.  Not all of this will
be used for our main theorem, but it is helpful to put the rest in
perspective, and will be referred to in Chapter \ref{chap:QG} when we
discuss proposed extensions of our main results to quantum gravity.

In Chapter \ref{chap:connexttqft} we discuss how to construct an
extended TQFT based on a double bicategory of cobordisms with corners,
by means of the interpretation of a TQFT in terms of a
\textit{connection} on the manifolds involved.  This is related to the
\textit{Dijkgraaf-Witten models}, which are topological gauge
theories.  Our aim is to give a construction of an extended TQFT $Z_G$
as a \textit{weak} 2-functor, starting from any finite gauge group $G$
(in a way which suggests how to extend the theory to an infinite gauge
group).

Section \ref{sec:ZGonMan} describes how to get a KV 2-vectorspace from
a manifold.  Given a manifold $B$, one first takes the fundamental
groupoid $\Pi_1(B)$, whose objects are the points in $B$ and whose
morphisms are homotopy classes of paths in $B$.  Then a connection on
the cobordism (or one of the components of the boundary) is a functor
$A : \Pi_1(B) \ra G$ where the gauge group $G$ is thought of as a
category (in fact a groupoid) with one object.

These functors correspond to \textit{flat $G$-bundles}---that is, each
such functor from $\Pi_1(M)$ to $G$ corresponds to a flat connection
on some principal $G$-bundle over $M$.  Some such functor corresponds
to any such connection on \textit{any} $G$-bundle.  For convenience,
we just call them ``connections''.  Gauge transformations between
connections correspond exactly to the natural transformations between
the functors into $G$.  So the connections and gauge transformations
are naturally organized into a functor category
$\opname{hom}(Pi_1(B),G)$, or just $\fc{B}$ for short.  This category
is a groupoid, and since manifolds have finitely generated fundamental
groups, it is a finite groupoid.  This now plays the role of the
``configuration space'' of the theory.

We then want to \textit{quantize} this configuration space $\fc{B}$.
In ordinary quantum mechanics, quantization might involve taking the
space of $L^2$ functions from a configuration space into
$\mathbbm{C}$.  In the categorified setting, we take the category
functors into $\V$---what we have called $\V$-presheaves---and get a
2-vector space.  We will be considering only the case $G$ is finite,
and as remarked, $\Pi_1(B)$ finitely generated.  So then $\fc{B}$ is
an essentially finite groupoid, and $Z_G(B) = \Z{B}$ will be a KV
2-vector space.

Next one wants to find 2-linear maps from cobordisms.  But a cobordism
$S: B \ra B'$ can be interpreted as a special cospan
\begin{equation}
\xymatrix{
  & S & \\
 B \ar[ur]^{i} & & \ar[ul]_{i'} B'
}
\end{equation}
with two inclusion maps.  Since the operation $\fc{-}$ is a
contravariant functor, applying it results in a span of the resulting
groupoids, where the inclusions are replaced with restriction maps:
\begin{equation}
\xymatrix{
  & \fc{S} \ar[dl]_{p} \ar[dr]_{p'} & \\
 \fc{B}  & &  \fc{B'}
}
\end{equation}

This is a span, which we can think of as giving restrictions from a
groupoid of ``histories'' in the middle to groupoids of
``configurations'' at the ends, via the projection maps $p$ and $p'$.
These are source and target maps, when we think of the original span
as a cobordism in $\nCobi$.  This groupoid represents configurations
of some system whose individual states are flat $G$-bundles.  Thinking
of spaces in terms of their path groupoids forces us to categorify the
gauge group.  The DW model fits this framework if we think of $G$ as a
one-object groupoid (though one might generalize to replace the gauge
group $G$ in various ways, such as a 2-group, as discussed by Martins
and Porter \cite{martinsporter}) and get a different theory.

After taking $\V$-presheaves, we are back to a cospan (again because
the functor $[-,\V]$ is contravariant).  It is:
\begin{equation}
\xymatrix{
  & \Z{S}   & \\
 \Z{B} \ar[ur]_{p^{\ast}} & &  \Z{B'} \ar[ul]_{(p')^{\ast}}
}
\end{equation}
where the most evident choices for 2-linear maps between these KV
2-vector spaces are the pullbacks along the restriction maps.  The
functor $\Z{-}$ which gives 2-vector spaces for manifolds, and indeed
topological spaces (as long as the fundamental group is finitely
generated).  We want to use it to yield some 2-functor $Z_G : \nCob \ra
\iiV$.  Objects in $\nCob$ are objects in a category of manifolds with corners, but we then
would like to get a 2-linear map from a cobordism.  However, this is
given as a cospan, so we have two pullback maps in the above diagram,
both of which have the adjoints discussed above.  Since $S$ is a
cobordism with source $B$ and target $B'$, we can take the adjoint
$(p')^{\ast}$ of the right-hand map, $(p')_{\ast}$, to get a 2-linear map:
\begin{equation}
(p')_{\ast} \circ (p)^{\ast} : Z_G(B) \ra Z_G(B')
\end{equation}
This will be $Z_G(S)$.  We refer to this as a ``pull-push'' process.
It consitsts of two stages.  The first stage is a ``pull'', which
gives a $\V$-presheaf $p^{\ast} F$ on the groupoid of connections on
the cobordism $S$ from $F$ on the manifold $B$.  This is done by
assigning to each connection $A$ on $S$ the vector space $p^{\ast}F(A)
= F \circ p (A)$ assigned by $F$ to the restriction $p(A) = A|_{B}$ of
$A$ to $B$ (and acts on gauge transformations in a compatible way).

The second stage is a ``push'', which gives a $\V$-presheaf on $B'$
from this $p^{\ast} F$ on the groupoid of connections on $S$.  This
assigns to each connection $A'$ on $B'$ a vector space
$(p')_{\ast}\circ p^{\ast}(F)= \opname{colim} p^{\ast}F(A)$, which is a
colimit over all the connections $A$ on $S$ which restrict to $A'$.
The colimit should be thought of as a direct sum over the equivalence
classes of such components.  The terms of the sum are, not the vector
spaces assigned by $p^{\ast}F$, but quotients of these which arise
from the fact that some connections may have nontrivial automorphisms.

The ``pull-push'' process is related to the idea of a ``sum over
histories''.  Recall that we can think of the 2-vector space of
$\V$-presheaves $Z_G(B)$ as a categorified equivalent of the Hilbert
space $L^2(X)$ we get when quantizing a classical system with
configuration space $X$.  So a component in the matrix representation
of the 2-linear transformation $Z_G(S)$ is indexed by configurations
(i.e. connections) on the initial and final spaces.  This component
vector space can be interpreted as a categorified ``amplitude'' to get
from the initial configuration to the final configuration.

A similar procedure, discussed in Section \ref{sec:ZGonCobCob}, is
used to get a 2-morphism from a cobordism between cobordisms.  That
is, given a cobordism with corners, $M : S_1 \ra S_2$, between two
cobordisms $S_1, S_2 : B \ra B'$, we have :
\begin{equation}
\xymatrix{
 & S_1 \ar[d]_{i} & \\
B \ar[ur]^{i_1} \ar[dr]_{i_2} & M & B' \ar[ul]_{i'_1} \ar[dl]^{i'_2} \\
 & S_2 \ar[u]^{i'} & \\
}
\end{equation}
To construct a natural transformation $Z_G(M) : Z_G(S_1) \ra
Z_G(S_2)$, a very similar process of ``pull-push'' The difference is
that instead of pulling and pushing $\V$-presheaves---that is,
2-vectors---one is pulling and pushing vectors.  These vectors can be
interpreted as $\mathbbm{C}$-valued functions on a basis of the vector
space which forms a component of the 2-linear map $Z_G(S_1)$ or
$Z_G(S_2)$.  Such a basis consists of equivalence classes of
connections on $S_1$ and $S_2$ respectively.  Choosing a particular
component (that is, fixing equivalence classes connections $A$ and
$A'$ on $B$ and $B'$), one then builds a linear transformation
\begin{equation}
Z_G(M)_{[A],[A']} : Z_G(S_1)_{[A],[A']} \ra Z_G(S_2)_{[A],[A']}
\end{equation}
by a ``pull-push''.  The ``pull'' phase of this process simply pulls
$\mathbbm{C}$-valued functors along the restriction map taking
connections on $M$ to connections on $S_1$.  The ``push'' phase here,
as at the previous level, assigns to a connection $A_2$ on $S_2$ a sum
over all connections on $M$ restricting to $A_2$.  And again, the sum
is not just of these components, but of a ``quotient'' which arises
from the automorphism group of each such connection on $M$.  This
quotient is related to the concept of ``groupoid cardinality'', and
this is discussed in Section \ref{sec:ZGonCobCob}.

So we have described a construction of an assignment $Z_G$ which gives
a KV 2-vector space for any manifold, a 2-linear map for any cobordism
of manifolds, and a natural transformation of 2-linear maps for any
cobordism between cobordisms. The main theorem here, which forms the
focus of Section \ref{sec:maintheorem}, is that this $Z_G$ indeed
forms a weak 2-functor from $\nCob$ to $\iiV$.  Along the way we will
have proved most of the properties needed, and it remains to verify
some technical conditions about the 2-morphisms which accomplish the
weak preservation of composites and units.

Finally, Chapter \ref{chap:QG} describes some of the motivation for
this work coming from quantum gravity, and particularly 3-dimensional
quantum gravity.  To really apply these results to that subject, one
would need to extend them.  Most immediately, one would need to show
that a construction like the one described will still give a weak
2-functor even when $G$ is not a finite group, but a Lie
group---or at least a compact one.

To do this would presumably require the use of the
infinite-dimensional variant of KV 2-vector spaces which Crane and
Yetter \cite{CY} call measurable categories.  This, and some of the
categorified equivalent of the structure of Hilbert spaces is
discussed in Section \ref{sec:CY}, and in Section \ref{sec:liegroups}
we address how it might be used to generalize the results above.  In
particular, we may not have infinite colimits available to perform the
``push'' part of our ``pull-push'' construction.  This means there
would have to be some other way to apply the idea of a ``sum over
histories'' in the categorified setting.  Our proposal is that this
should be related to the ``direct integral'' in the Crane-Yetter
measurable categories mentioned above.

Section \ref{sec:QG3D}, considers the special case when $G=SU(2)$,
which is the relevant gauge group for 3D quantum gravity.  The
particular case of interest is a 3D extended TQFT, where manifolds are
1-dimensional, joined by 2D cobordisms, which are in turn joined by 3D
cobordisms with corners.  We discuss how one might interpret the
theory as quantum gravity coupled to matter.  The basic idea is that
the manifolds represent boundaries of regions in space.  A circle
describes the boundary left when a point (up to homotopy) is removed
from 2-dimensional space.  The 2D cobordisms in our double bicategory
can then represent the ambient space it is removed from.
Alternatively the cobordisms can describe the ``world-line'' of such a
point particle''.  These two possibilities represent the
``horizontal'' and ``vertical'' directions within the {\vdb} of
cobordisms.  The cobordisms of cobordisms then represent the whole
``spacetime'', in a general sense, in which this situation is set.

The cobordism with corners in Figure \ref{fig:cobcorners} would then be
interpreted (reading top-to-bottom) as depicting a space in which two
regions bounded by the outside circles merge together into a single
region over time.  Inside each region at the beginning there is a
single puncture.  After the regions merge, the two punctures---now in
the same region---merge and split apart twice.  At the ``end''
(i.e. the bottom of the picture), there is a single region containing
two punctures.  The physical intuition is that a ``puncture'', or
equivalently the circular boundary around it, describes a point
particle.  The 2-vector space of states which the extended TQFT
assigns to the circle is then the 2-vector space of states for a
particle.

This 2-vector space consists of $\V$-presheaves on $\fc{S^1}$.
Example \ref{ex:ZonS1} shows for \textit{finite} groups $G$ that this
is generated by a finite set of objects, each of which corresponds to
a pair $([g],\rho)$, where $[g]$ is a conjugacy class in $G$, and
$\rho$ is a linear representation of $G$.  There is an obstacle to an
analogous fact in infinite dimensional 2-vector space, since these may
not have a basis of simple objects.  This fact is precisely analogous
to the fact that an infinite dimensional Hilbert space need not have a
countable basis, since it follows from the fact that not every object
will be \textit{finitely} generated from some set of simple objects -
and we do not have infinite sums available in $\V$.  However, even in an
infinite dimensional 2-vector space, it does make sense to speak of
simple objects, and we expect these to be of the form described.

So then for $G=SU(2)$, we have the simple $\V$-presheaves classified
by a conjugacy class in $SU(2)$, which is just an ``angle'' in
$[0,4\pi)$---since $SU(2)$ doubly covers $SO(3,1)$---and a
representation of $U(1)$, the stabilizer subgroup of a point under the
adjoint action of $SU(2)$ on itself.  Such representations give
integer ``spins''.  These are the same data which label particles in
3D quantum gravity - the ``angle'' is a mass, which has a maximum
value in 3D gravity, since mass causes a ``conical defect'' in the
geometry of space, which has a maximum possible angle.  The ``spin''
is related to angular momentum.

So this theory allows us to describe a space filled with world-lines
of ``particles'' labelled by (bounded) mass and spin.  This is exactly
the setup of the Ponzano-Regge model of 3D quantum gravity.  Our
expectation is that this model can be recovered from an extended TQFT
based on $SU(2)$.  This is related to a program, on which more details
can be found in a paper of Lee Smolin \cite{smolin}, which seeks to
study 3D quantum gravity by means of its relation to a 3D TQFT
associated to $SU(2)$ Chern-Simons theory.

Finally, in Section \ref{sec:QGwards}, we briefly suggest a possilble
direction to look for links between the theory given here, and
spin-foam models for BF theory, based on a categorification of the FHK
state sum approach to defining an ordinary TQFT.  We also suggest two
more directions in which one might generalize the theory described in
this paper in the same style as the passage from finite groups to
infinite Lie groups.  Two others are to pass from groups to
categorical groups, and to pass from groups to quantum groups.

We can think of a group as a kind of category with one object and all
morphisms invertible.  A categorical group will have a group of
objects and a group of morphisms, satisfying certain conditions.
Replacing our gauge group $G$ with a categorical group gives a theory
based not on \textit{connections}, but on \textit{2-connections}.
There is extensive work on this topic, but a good overview is the
discussion by Baez and Schreiber \cite{HGA} (see also the definition
of 2-bundles by Bartels \cite{bartels}).  The extension of the
Dijkgraaf-Witten model to categorical groups is discussed in a
somewhat different framework by Martins and Porter
\cite{martinsporter}.  An extension of these ideas to quantum groups
is less well studied, but the hope is to recover the connection
between $q$-deformed $SU(2)$ and the Turaev-Viro model, just as using
$SU(2)$ as gauge group recovers the Ponzano-Regge model, for quantum
gravity.

In all these directions, and possibly more, the expression of an
extended TQFT in functorial terms seems to provide a window on a
variety of potentially useful applications and generalizations.

\section{Topological Quantum Field Theories}\label{chap:TQFT}

\subsection{The Category $\nCobi$}\label{sec:ncob} 
In this section, we review the structure of the symmetric monoidal
category $\iiCob$ which we generalize in this paper.  Cobordism theory
goes back to the work of Ren\'e Thom \cite{thom}, who showed that it
is closely related to homotopy theory.  In particular, Thom showed
that \textit{cobordism groups}, whose elements are \textit{cobordism
classes} of certain spaces, can be computed as homotopy groups in a
certain complex.  However, this goes beyond what we wish to examine
here: a good introductory discussion suitable for our needs is found,
e.g. in Hirsch \cite{hirsch}.  There is substantial research on many
questions in, and applications of, cobordism theory: a brief survey of
some has been given by Michael Atiyah \cite{aticob}.  Some further
examples related to our motivation here include Khovanov homology
\cite{khovanov} (also discussed in \cite{barnatan} and
\cite{kauffman}), and Turaev's recent work on cobordism of knots on
surfaces \cite{turaev}.

Two manifolds $S_1, S_2$ are \textit{cobordant} if there is a compact
manifold with boundary, $M$, such that $\partial M$ is isomorphic to
the disjoint union of $S_1$ and $S_2$.  This $M$ is called a
\textit{cobordism} between $S_1$ and $S_2$.  We note that there is
some similarity between this concept and that of homotopy of paths,
except that such homotopies are understood as embedded in an ambient
space.  We will return to this in Section \ref{sec:homotopy}. Our aim
here is to describe a generalization of categories of cobordisms.  To
begin with, we recall some of the structure of $\nCobi$, and
particularly $\iiCob$, to recall why this is of interest.

\begin{definition}$\iiCob$ is the category with:
\begin{itemize}
\item Objects: one-dimensional compact oriented manifolds
\item Morphisms: diffeomorphism classes of two-dimensional compact oriented cobordisms between such manifolds.
\end{itemize}
\end{definition}

That is, the objects are collections of circles, and the morphisms are
(diffeomorphism classes of) manifolds with boundary, whose boundaries
are broken into two parts, which we consider their source and target.
We think of the cobordism as ``joining'' two manifolds, rather as a
relation joins two sets, in the category of sets and relations (this
analogy will be made more precise when we discuss spans and cospans).
More generally, $\nCobi$ is the category whose objects are (compact,
oriented) $(n-1)$-dimensional manifolds, and whose morphisms are
diffeomorphism classes of compact oriented $n$-dimensional cobordisms.

It has been known for some time that $\iiCob$ can be seen as the free
symmetric monoidal category on a commutative Frobenius object.  (This
is shown in the good development by Joachim Kock \cite{kock}.)  This
is a categorical formulation of the fact, shown by Abrams
\cite{abrams}, that $\iiCob$ is generated from four generators, called
the \textbf{unit}, \textbf{counit}, \textbf{multiplication},
\textbf{comultiplication}, subject to some relations.  The generating
cobordisms are the following: taking the empty set to the circle (the
unit); taking two circles to one circle (the multiplication); adjoints
of each of these (counit and comultiplication respectively).

\begin{figure}[h]
\begin{center}
\includegraphics{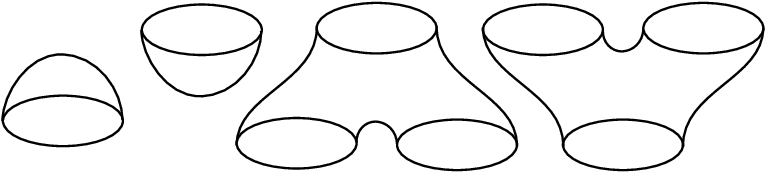}
\end{center}
\caption{\label{fig:cobgens}Generators of $\iiCob$}
\end{figure}

The ``commutative Frobenius object'' here is the circle, equipped with
these morphisms, as illustrated in Figure \ref{fig:cobgens}.  The
relations which these are subject to include associativity,
coassociativity, and relations for the unit and counit.  The most
interesting is the Frobenius relation, illustrated in Figure
\ref{fig:frobrel}.

\begin{figure}[h]
\begin{center}
\includegraphics{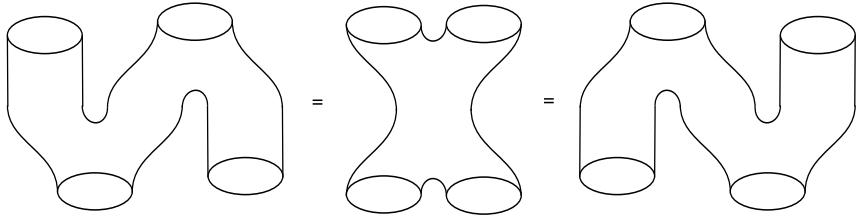}
\end{center}
\caption{\label{fig:frobrel}The Frobenius Relation}
\end{figure}

Diffeomorphism classes of cobordisms automatically satisfy these
relations, since they identify composites of cobordisms which are, in
fact, diffeomorphic.

Moreover, as a \textit{monoidal} category, $\iiCob$ must have a tensor
product operation.  For objects, this is just the disjoint union:
given objects $\mathbf{m}, \mathbf{n} \in \iiCob$, consisting of
collections of $m$ and $n$ circles respectively, the object
$\mathbf{m} \otimes \mathbf{n}$ is the disjoint union of $\mathbf{m}$
and $\mathbf{n}$: a collection of $m+n$ circles.  The tensor product
of two cobordisms $\mathbf{C_1}:\mathbf{m_1}\ra\mathbf{n_1}$
and $\mathbf{C_2}:\mathbf{m_2}\ra\mathbf{n_2}$ is likewise the
disjoint union of the two cobordisms, giving $\mathbf{C_1} \otimes
\mathbf{C_2} : \mathbf{m_1} \otimes \mathbf{m_2} \ra
\mathbf{n_1} \otimes \mathbf{n_2}$.

This monoidal operation has a \textit{symmetry}, so in particular
$\iiCob$ also includes the \textbf{switch} cobordism, exchanging the
order of two circles by two cylinders (this gives the symmetry for the
monoidal operation).  These are required to exist by the assumption
that $\iiCob$ is a free symmetric monoidal category.  They are
illustrated in Figure \ref{fig:symmongens} (along with the identity,
which is, of course, also required).

\begin{figure}[h]
\begin{center}
\includegraphics{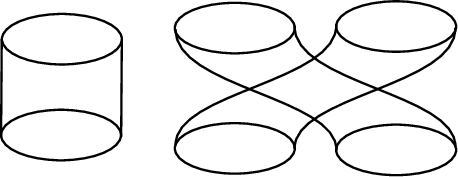}
\end{center}
\caption{\label{fig:symmongens}Morphisms Required for $\iiCob$ to be a Symmetric Monoidal Category}
\end{figure}

Two proofs can be given for the fact than $\iiCob$ is generated by
these cobordisms.  Each proof relies on some special conditions satisfied by
2D cobordisms.  The first is that 2-dimensional manifolds with
boundary can be completely classified up to diffeomorphism class by
genus and number of punctures.  The second is that we can use the
results of Morse theory to decompose any such surface, equipped with a
smooth Morse function into $[0,1]$, into a composite of pieces, in the
sense of composition of morphisms in $\iiCob$.  In each piece, there
is just one ``topology change'' (a value in $[0,1]$ where the preimage
changes topology).  We will return to this point when we discuss the
question of how to present $\nCob$ in terms of generators.

So far, we have described the presentation of $\iiCob$ in terms of
generators and relations, but not yet how the composition operation
for morphisms works.  The main idea is that we compose cobordisms by
identifying their boundaries.  However, since the morphisms in
$\iiCob$ are \textit{diffeomorphism} classes of manifolds with boundary, some extra
considerations are needed to ensure that the composite is equipped
with a differentiable structure.

In particular, the \textit{collaring theorem} means that any manifold
with boundary, $M$ can be equipped with a ``collar'': an injection
$\phi : \partial M \times [0,1] \ra M$ such that $\phi(x,0)=x,
\forall x \in \partial M$.  The idea is that, while we can compose
\textit{topological} cobordisms along their boundaries, we should
compose \textit{smooth} cobordisms $M_1$ and $M_2$ along collars.
This ensures that every point---including points on the boundary of
$M_i$---will have a neighborhood with a smooth coordinate chart.
Section \ref{sec:collarman} describes this in detail for a more
general setting.

The category $\iiCob$ is particularly interesting in the study of
topological quantum field theories (TQFT's), as formalized by Michael
Atiyah \cite{atiyah}.  Atiyah's axiomatic formulation of a TQFT
amounts to saying that it is a symmetric monoidal functor $F : \iiCob
\ra \catname{Vect}$.  The presentation of $\iiCob$ means that
this immediately defines an algebraic structure with a unit, counit,
multiplication, comultiplication, and identity, which satisfy the same
relations as the corresponding cobordisms. This, together with the
fact that $F$ preserves the symmetric monoidal structure of $\iiCob$
means that this structure satisfies the axioms of a commutative
Frobenius algebra.  A similar presentation has not been found for
$\nCobi$ for general $n$.

One may wish to describe an ``extended topological quantum field
theory'' in the same format.  These are topological field theories
which are defined not just on manifolds with boundary, but also on
manifolds with corners.  This idea is described by Ruth Lawrence in
\cite{lawrence}.  In particular, what we are interested in here is
that, instead of using a category of cobordisms between manifolds, we
would want to use some structure of cobordisms between
\textit{cobordisms between} manifolds, which we tentatively call
$\nCob$.  However, to do this, we must use a structure with more
elaborate than a mere category.

Later, we will describe such a structure---a \textit{\vdb}, and show
how the putative $\nCob$ is an example, and indeed a special case of a
wider class of examples.

\subsection{TQFT's as Functors}\label{sec:tqftfunctor}

Atiyah's formulation of the axioms for a TQFT can be summarized as follows:
\begin{definition}A Topological Quantum Field Theory is a (\textit{symmetric}) monoidal functor
\begin{equation}
Z: \iiCob \ra \V
\end{equation}
where $\iiCob$ is as described in Section \ref{sec:ncob}, and $\V$ is
the category whose objects are vector spaces and whose arrows are
linear transformations.
\end{definition}

We note that $\V$ is naturally made into a \textit{monoidal} category
with the tensor product $\otimes$, where $V_1 \otimes V_2$ is
generated by objects of the form $v_1 \otimes v_2$, modulo relations
imposing bilinearity.  Moreover, $\iiCob$ is a monoidal category as
well, whose monoidal product on objects and morphisms is just the
disjoint union of manifolds and cobordisms, respectively.

In fact, a quantum field theory should give a \textit{Hilbert} space
of states.  However, $\Hilb$, the category of Hilbert spaces and
bounded linear maps, is a subcategory of $\V$, so the above is still
true.

What, however, does this definition mean?

A TQFT should give a \textit{Hilbert space of states} for any manifold
representing ``space'', and a map from one space of states to another
for any cobordism representing ``spacetime'' connecting two space
slices.  Figure \ref{fig:TQFT} shows an example in the case where
space is 1-dimensional and spacetime is 2-dimensional:

\begin{figure}[h]
\begin{center}
\includegraphics{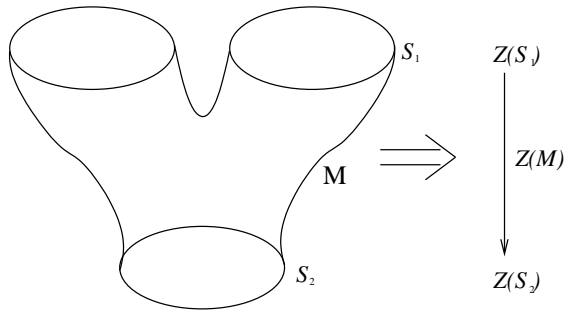}
\end{center}
\caption{\label{fig:TQFT}Effect of a TQFT}
\end{figure}

The TQFT should have the following properties:

\begin{itemize}
\item The Hilbert space assigned to a disjoint union of spaces $S_1
      \amalg S_2$ will be the tensor product of the spaces assigned to
      each, $Z(S_1) \otimes Z(S_2)$, and therefore also
      $Z(\varnothing) = \mathbbm{C}$ (a basic feature of quantum
      theories)

\item The linear maps assigned to cobordisms respect ``composition''
      of spacetimes, so $M_1$ followed by $M_2$ is assigned the map
      $Z(M_2) \circ Z(M_1)$, where ``followed by'' means the ending
      space of $M_1$ is the beginning space of $M_2$.
\end{itemize}

As remarked in Section \ref{sec:ncob}, $\iiCob$ is a free symmetric monoidal category on a
Frobenius object.  In $\V$, such an object is called a
\textit{Frobenius algebra}: in fact, a 2D TQFT $Z$ is equivalent to a
choice of Frobenius algebra, namely the image of the circle uvder $Z$.

In general higher dimensions, no equally straightforward description
of an $n$-dimensional TQFT is known.  To provide one would require a
presentation of $\nCobi$ in terms of generators and relations (for
both objects and morphisms).

Lauda and Pfeiffer \cite{openclosed} do provide such a presentation a
similar, though more complicated, characterization of 2-dimensional
\textit{open-closed} TQFT's.  In these, we do not assume that the
manifolds representing spaces have no boundary.  Lauda's doctoral
thesis \cite{laudathesis} develops this further.

\subsection{The Fukuma-Hosono-Kawai Construction and Connections}\label{sec:fhk}

Frobenius algebras are \textit{semisimple} algebras $A$ (direct sums
of simple algebras).  These are characterized by having a
nondegenerate linear pairing:
\begin{equation}
g: A \otimes A \ra \mathbbm{C}
\end{equation}
If $A$ is a matrix algebra, then such a $g$ is given by the Killing
form, or \textit{trace}: $g(a,b)= \opname{tr}(L_a L_b)$.  The
nondegeneracy of this pairing means that it gives an isomorphism
between $A$ and $A^{\ast}$.

Each algebra $A$ of this kind gives a TQFT whose effects can be
described in an explicit and combinatorial way.  This is the
construction of Fukuma, Honoso, and Kawai \cite{FHK}.  We will be
particularly interested in the case where the semisimple algebra $A$
is the group algebra $\mathbbm{C}[G]$ for some finite group $G$.

Now we want to see how to get a TQFT $Z: \iiCob \ra \V$ from
any such algebra $A$, keeping in mind the example
$A=\mathbbm{C}[G]$. To do this, we first construct a map $\hat{Z}
:\Delta\iiCob \ra \V$, where $\Delta\iiCob$ is the category of
\textit{triangulated} manifolds and cobordisms, then show it is
independent of the choice of triangulation.

To begin with, given a triangulated cobordism $M$ from $S_1$ to $S_2$,
(so $M$, $S_1$ and $S_2$ are all triangulated), label the dual graph
with copies of $A$.

\begin{figure}[h]
\begin{center}
\includegraphics{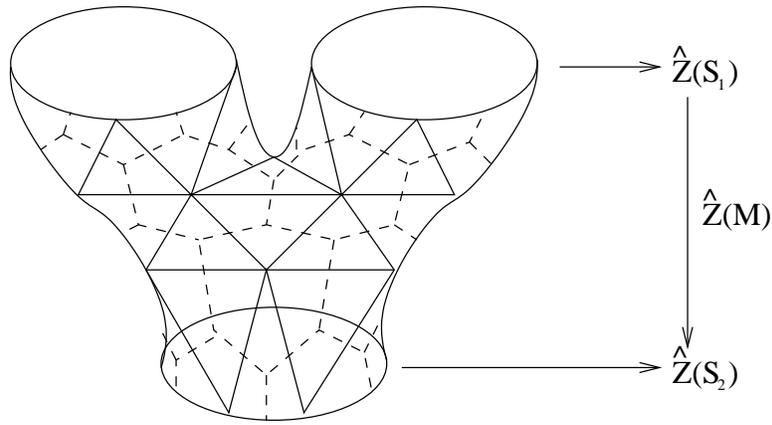}
\end{center}
\caption{\label{fig:FHK}The Fukuma-Honoso-Kawai Construction}
\end{figure}

So each edge of a triangle (hence of the dual graph) is labelled by
$A$ and each face of a triangle (hence each vertex of the dual graph)
by an operator $m$.
\begin{figure}[h]
\begin{center}
\includegraphics{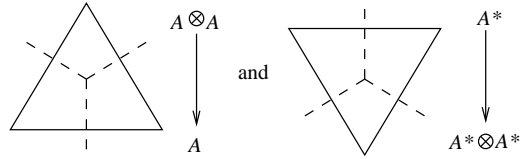}
\end{center}
\caption{\label{fig:trianglemult}Multiplication Operators Assigned to Triangles}
\end{figure}

In the case where the semisimple algebra is $\CG$, we can write
choices of vector in a basis consisting of group elements.  So
labellings of the dual edges can be described in terms of a basis
where the dual edges are labelled with group elements.

\subsection{Pachner Moves in 2D}\label{sec:pachner}

How does $\hat{Z}$, acting on $\Delta\iiCob$, give a TQFT acting on
$\iiCob$?  First, notice that it depends only on the topology of $M$,
and the triangulation on the boundary, not in the interior.

This is because \textbf{Alexander's Theorem} says that to pass between
any two triangulations of the same compact 2-manifold, it is enough to
repeatedly apply the two \textit{Pachner moves}---the \textbf{2-2 move}
and the \textbf{1-3 move} (and their inverses):

\begin{figure}[h]
\begin{center}
\includegraphics{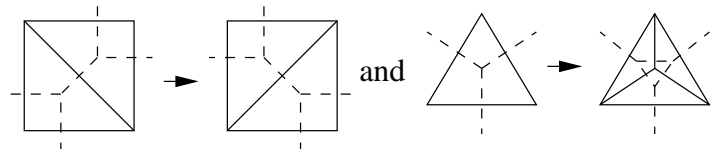}
\end{center}
\caption{\label{fig:pachner}Pachner Moves}
\end{figure}

This will prove that the linear map we construct is independent of the
triangulation chosen.  In particular, the 2-2 move does not affect the
outcome of composition, on applying $\hat{Z}$, since it passes from
\begin{equation}
V \otimes V \otimes V \longrightarrowlim^{1\otimes m} V \otimes V \longrightarrowlim^{m} V
\end{equation}
to
\begin{equation}
V \otimes V \otimes V \longrightarrowlim^{m\otimes 1} V \otimes V \longrightarrowlim^{m} V
\end{equation}

These are the same by associativity.

The 1-3 move has no effect precisely when $(A,\eta,m)$ is semisimple.
This comes from associativity and the ``bubble move'':

\begin{figure}[h]
\begin{center}
\includegraphics{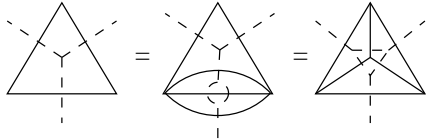}
\end{center}
\caption{\label{fig:bubble}The Bubble Move}
\end{figure}

We can think of the Pachner moves as coming from tetrahedrons.  Given
a triangulation, attach a tetrahedron along one, two, or three
triangular faces.  The move consists of replacing the attached faces
with the remaning faces of the tetrahedron.  We can think of this as
``evolving the triangulation by'' that tetrahedron:
\begin{figure}[h]
\begin{center}
\includegraphics{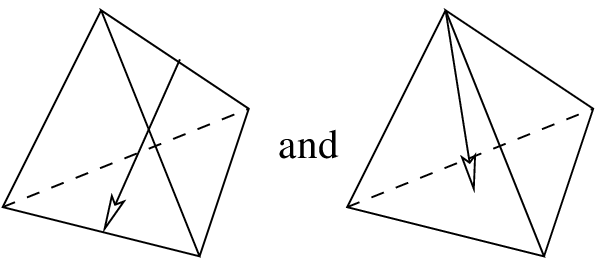}
\end{center}
\caption{\label{fig:pachnertetras}Pachner Moves as Tetrahedra}
\end{figure}

Any two triangulations are homologous---can be connected by a series
of such moves since there is no nontrivial third homology of a 2D
surface: any change in triangulation we want will be the boundary of
some collection of tetrahedra.  (A triangulation of a 2-dimensional
cobordism is a combination of 0, 1, and 2-chains---Pachner moves
correspond to 3-chains).

Now, we know that a TQFT is determined by its effect on the generators
of $\iiCob$, so we want to know the space of states on $S^1$, which is
a generator for objects.  One observation is that the image of the
generator $S^1 \times [0,1]$ is $\id$, the identity map on $Z(S^1)$.

Consider the following triangulation of $S^1 \times [0,1]$:
\begin{figure}[h]
\begin{center}
\includegraphics{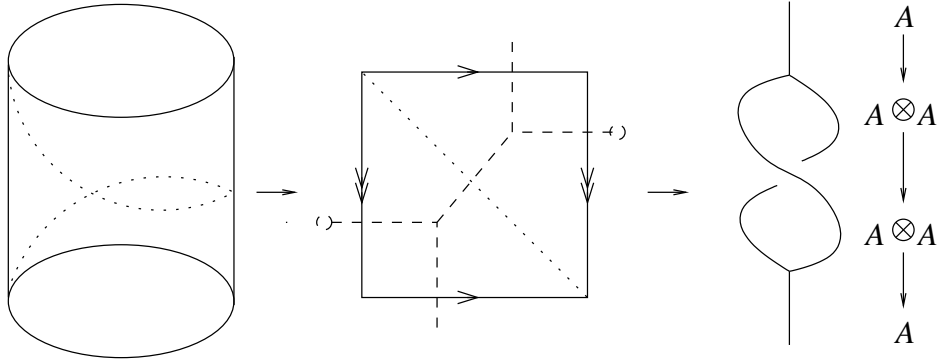}
\end{center}
\caption{\label{fig:centremap}Identity or Projection Operator?}
\end{figure}

$\hat{Z}$ assigns $A$ to the top and bottom circles, but says that we should have
\begin{equation}
m \circ B \circ m^{\dagger} = \id
\end{equation}
on $Z(S) \subset A$.  This means that $Z(S)$ is a subset of the
\textit{centre} of $A$.

We know that the identity map in $A$ must come from the cylinder, so define
\begin{equation}
Z(S^1) = \opname{Ran}(\hat{Z}(S^1 \times [0,1]))
\end{equation}

To get a TQFT $Z$, we restrict $\hat{Z}$ to $Z(S^1)$.  This is a
projection operator, and its range is in $Z(A)$.  Project the space of
states for a triangulated circle onto this to get the space of states
for the circle under $Z$ (note that there is only one way to do this,
independent of which triangulation of the cylinder we use to get the
projection operator).

So it is well-defined to say:
\begin{equation}
Z(M) = \hat{Z}(M) |_{Z(S^1)}
\end{equation}
since we always have $\hat{Z}(M)(Z(S_1)) = Z(S_2)$.  (One can
retriangulate $M$ to compose with the projection before and after,
without changing the result.)

Then one can show that this $Z$ defines a \textit{symmetric monoidal
functor} from $\iiCob$ to $\V$, namely a TQFT.

\subsection{TQFT's and Connections}

The FHK construction of a TQFT has a feature which may not at first be
obvious.  To the circle, $Z$ assigns a Hilbert space, but in a way
that has a canonical choice of basis.  This is $Z(S^1)$, the centre of
the group algebra $\mathbbm{C}[G]$, or simply
$\mathbbm{C}[\opname{Cent}(G)]$, the vector space spanned by the
centre of the group $G$.  So a basis for the space of states is just
the set of ways of assigning to the circle an element of the group $G$
which happens to be in the centre of $G$.

One way to think of this is as a $G$-\textit{connection} on the circle
- so that the space of states is a free vector space on the set of
$G$-connections on $S^1$.  This way of thinking of what $Z$ produces
is good because it will hold up even when we consider manifolds $B$ of
higher dimension (and codimension).  In particular, if a TQFT gives a
space of states from the set of connections on $B$, given a map from
the circle into $B$, any connection assigns to this loop a group
element, or \textit{holonomy}, up to conjugation.

So in order to look at extended TQFT's as examples of a
categorification of the concet of a TQFT, it is useful to take this
point of view relating the TQFT to connections.  We point out,
however, that there is a categorified analog of the FHK construction
more or less directly.  We expect that this would provide a
``state-sum'' point of view on the theory of a connection on a
manifold which our extended TQFT will in fact involve.  In fact, this
is understood to a considerable degree, but this point of view is
awkward because it involves the categorified versions of associativity
- Stasheff's associahedra \cite{stasheff}.  These play the role of
Pachner moves in higher dimensions.  We could proceed with this
categorified version of the construction, when $G$ is a finite group.

It turns out that a natural generalization of the FHK construction
gives a theory equivalent to the (untwisted) Dijkgraaf-Witten model
\cite{DW}.  This is a topological gauge theory, which crucially
involves a (flat) connection on a manifold.  We will discuss this in
more detail in Section \ref{sec:ZGonMan}, and explore how an extended
TQFT can be constructed by taking a categorifed version of the
(quantized) theory of a flat connection on manifolds and cobordisms.

\section{Bicategories and Double Categories}\label{chap:bianddbl}

We will want to give a description of a \textit{\vdb}, which is a
weakened version of the concept of a double category, in order to
describe cobordisms with corners.  Weakening a concept $X$ in category
theory generally involves creating a new concept in which equations in
the original concept are replaced by isomorphisms.  Thus, we say that
the old equations hold only ``up to'' isomorphism in the weak version
of $X$, and say that when they hold with equality, we have a ``strict
$X$''.  Thus, before describing our newly weakened concept, it makes
sense to recall how this process works, and examine the strict form of
the concept we want to weaken.  We also want to see what the weakening
process entails.  So we begin by reviewing bicategories and double
categories.

\subsection{2-Categories}

A category $\catname{E}$ is \textbf{enriched over} a category
$\catname{C}$ (which must have products) when for $x,y \in
\catname{E}$ we have $\hom(x,y) \in \catname{C}$. A special case of
this occurs in ``closed'' categories, which are enriched over
themselves---examples include $\Set$ (since there is a set of maps
between any two sets) and $\catname{Vect}$ (since the linear operators
between two vector spaces form a vector space).

A \textbf{2-category} is a category enriched over $\Cat$.  That is, if
$\catname{C_2}$ is a 2-category, and $x,y \in \catname{C_2})$, then
$\hom(x,y) \in \Cat$.  Thus, there are sets of objects and morphisms
in $\hom(x,y)$ itself, satisfying the usual category axioms.  We describe a
2-category as having \textbf{objects}, \textbf{morphisms} between
objects, and \textbf{2-morphisms} between morphisms.  The morphisms of
$\catname{C_2}$ are the objects of the $\hom$-categories, and the
2-morphisms of $\catname{C_2}$ are the morphisms of the
$\hom$-categories.  We depict these as in Diagram (\ref{xy:2morphism}).
There is a composition operation for morphisms in these $\hom$ categories, which we think of as ``vertical'' composition, denoted $\cdot$, between 2-morphisms.
Furthermore, for all $x,y,z \in \catname{C_2}$, the composition
operation
\begin{equation}
  \circ : \hom(x,y) \times \hom(y,z) \ra \hom(x,z)
\end{equation} must be a functor between $\hom$-categories.  So in
particular this operation applies to both objects and morphisms in
$\hom$ categories, and we think of these as ``horizontal'' composition
for both morphisms and 2-morphisms.  The requirement that this be a
functor means that the \textbf{interchange law} holds:
\begin{equation}\label{eq:bicatinterchange}
(\alpha \circ \beta) \cdot (\alpha' \circ \beta') = (\alpha \cdot \alpha') \circ (\beta \cdot \beta')
\end{equation}

Now, in a 2-category, the associative law holds strictly: that is, for
morphisms $f \in \hom(w,x)$, $g \in \hom(x,y)$, and $h \in \hom
(y,z)$, we have the two possible triple-compositions in $\hom(w,z)$
the same, namely $f \circ (g \circ h) = (f \circ g) \circ h$.  This is
one of the axioms for a category---that is, a category enriched over
$\Set$.  Since a 2-category is enriched over $\Cat$, however, a weaker
version of this rule is possible, since $\hom(w,z)$ is no longer a set
in which elements can only be equal or unequal: it is a category,
where it is possible to speak of isomorphic objects.  This fact leads
to the notion of bicategories.

\subsection{Bicategories}

Once we have the concept of a 2-category, we can \textit{weaken} this
concept, giving the idea of a \textbf{bicategory}.  The definition is
similar to that for a 2-category, but we only insist that the usual
equations should be natural isomorphisms (satisfying some equations).
That is, the following diagrams should commute up to natural
isomorphisms:
\begin{equation}\label{eq:bicatcompos}
  \xymatrix{
    \hom(w,x) \times \hom(x,y) \times \hom(y,z) \ar[d]_{\circ \times 1} \ar[r]^{1 \times \circ} & \hom(w,x) \times \hom(x,z) \ar[d]^{\circ} \\
    \hom(w,y) \times \hom(y,z) \ar[r]_{\circ} & \hom(w,z) \\
  }
\end{equation}
and
\begin{equation}\label{eq:bicatrunit}
  \xymatrix{
    \hom(x,y) \times \catname{1} \ar[dr]^{\pi_1} \ar[d]_{\opname{id} \times \opname{!}} & \\
    \hom(x,y) \times \hom(x,x) \ar[r]^{\circ} & \hom(x,y) \\
  }
\end{equation}
and
\begin{equation}\label{eq:bicatlunit}
  \xymatrix{
    \catname{1} \times \hom(x,y) \ar[dr]^{\pi_2} \ar[d]_{\opname{!} \times \opname{id}}& \\
    \hom(y,y) \times \hom(x,y) \ar[r]^{\circ} & \hom(x,y) \\
  }
\end{equation}

That is: given $(f,g,h) \in \hom(w,x) \times \hom(x,y)
\times \hom(y,z)$, there should be an isomorphism $a_{f,g,h} \in
\hom(w,z)$ with $a_{f,g,h} : (f \circ g) \circ h \ra f \circ
(g \circ h)$; and isomorphisms $r_f: f \circ 1_x$, $l_f : 1_y \circ
f$.  The equations these satisfy are \textit{coherence laws}.
MacLane's Coherence Theorem shows that all such equations follow from
two generating equations: the pentagon identity, and the unitor law:

In a category, the associativity property stated that two composition
operations can be performed in either order and the results should be
equal: equality is the only sensible relation between a pair of
morphisms in a category.  There is an analogous statement for the
associator 2-morphism: two different ways of composing it should yield
the same result (since equality is the only sensible relation between a
pair of 2-morphisms in a bicategory).  This property is the pentagon
identity:
\begin{equation}\label{eq:pentagonid}
\xy
 (0,20)*+{(f \circ g) \circ (h \circ j)}="1";
 (35,4)*+{f \circ (g \circ (h \circ j))}="2";
 (23,-20)*+{ \quad f \circ ((g \circ h) \circ j)}="3";
 (-23,-20)*+{(f \circ (g \circ h)) \circ j}="4";
 (-35,4)*+{((f \circ g) \circ h) \circ j}="5";
     {\ar^{a_{f,g,h \circ j}}     "1";"2"}
     {\ar_{1_f \circ a _{g,h,j}}  "3";"2"}
     {\ar_{a _{f,g \circ h,j}}    "4";"3"}
     {\ar_{a _{f,g,h} \circ 1_j}  "5";"4"}
     {\ar^{a _{f \circ g,h,j}}    "5";"1"}
\endxy
\\
\end{equation}

Similarly, the unit laws satisfy the property that the following
commutes:
\begin{equation}\label{eq:unitorlaws}
  \xymatrix{
    (g \circ 1_y) \circ f \ar[r]^{a_{g,1_y,f}} \ar[d]^{r_g \times 1_f} & g \circ (1 \circ f) \ar[dl]^{1_g \times l_f}\\
    g \circ f & \\
  }
\end{equation}

This last change is the sort of weakening we want to apply to the
concept of a double category.  Following the same pattern, we will
first describe the strict notion in Section \ref{sec:doublecat},
before considering how to weaken it, in Chapter \ref{chap:doublebicat}.
First, however, we will recall a standard, quite general, example of
bicategory, which we will generalize to give examples of double
bicategories in Section \ref{sec:dblspan}.

\subsection{Bicategories of Spans}\label{sec:spanbicat}

Jean B\'enabou \cite{benabou} introduced bicategories in a 1967 paper,
and one broad class of examples introduced there comes from the notion
of a \textit{span}.

\begin{definition}\label{def:span}(\textbf{B\'enabou}) Given any category $\C$, a \textbf{span}
$(S,\pi_1,\pi_2)$ between objects $X_1,X_2 \in \C$ is a diagram in
$\C$ of the form
\begin{equation}
  \xymatrix{
    P_1 & S \ar[l]_{\pi_1} \ar[r]^{\pi_2}  & P_2
  }
\end{equation}
Given two spans $(S,s,t)$ and $(S',s',t')$ between $X_1$ and $X_2$
between a \textbf{morphism of spans} is a morphism $g:S \ra S'$
making the following diagram commute:
\begin{equation}\label{eq:spanmorph}
  \xymatrix{
        & S \ar[dl]_{\pi_1} \ar[dr]^{\pi_2} \ar[d]^{g} &   \\
    X_1 & S' \ar[l]^{\pi'_1} \ar[r]_{\pi'_2} & X_2 \\
  }
\end{equation}

Composition of spans $S$ from $X_1$ to $X_2$ and $S'$
from $X_2$ to $X_3$ is given by a pullback: that is, an
object $R$ with maps $f_1$ and $f_2$ making the following diagram
commute:
\begin{equation}\label{eq:spanpullback}
  \xymatrix{
      &   & R \ar[dl]_{f_1} \ar[dr]^{f_2} &  & \\
      & S \ar[dl]_{\pi_1} \ar[dr]^{\pi_2} &   & S' \ar[dl]_{\pi'_2} \ar[dr]^{\pi'_3} & \\
    X_1 &  & X_2 &   & X_3 \\
  }
\end{equation} which is terminal among all such objects.  That is,
given any other $Q$ with maps $g_1$ and $g_2$ which make the analogous
diagram commute, these maps factor through a unique map $Q \ra
R$.  $R$ becomes a span from $X_1$ to $X_3$ with the maps $\pi_1 \circ f_1$
and $\pi_2 \circ f_2$.
\end{definition}
\mbox{}

The span construction has a dual concept:

\begin{definition}\label{def:cospan} A \textbf{cospan} in $\C$ is a span in $\Cop$,
morphisms of cospans are morphisms of spans in $\Cop$, and composition
of cospans is given by pullback in $\Cop$.  That is, by a pushout
in $\C$.
\end{definition}

One fact about (co)spans which is important for our purposes is that
any category $\C$ with limits (colimits, respectively) gives rise to a
bicategory of spans (or cospans).  This relies in part on the fact
that the pullback is a universal construction (universal properties of
$\Cspan$ are discussed by Dawson, Par\'e and Pronk \cite{unispan}).

\begin{remark}\textbf{\cite{benabou}, ex. 2.6}
\label{thm:spanbicat}Given any category $\C$ with all limits, there is
a bicategory $\Cspan$, whose objects are the objects of $\C$, whose
$hom$-sets of morphisms $\Cspan(X_1,X_2)$ consist of all spans between
$X_1$ and $X_2$ with composition as defined above, and whose
2-morphisms are morphisms of spans.  $\Cspan$ as defined above forms a
bicategory ($\Cosp(\C)$, of cospans similarly forms a bicategory).

This is a standard result, first shown by Jean B\'enabou
\cite{benabou}, as one of the first examples of a bicategory.  We
briefly describe the proof:

The identity for $X$ is $X \lalim^{id} X \ralim^{id}
X$, which is easy to check.

The associator arises from the fact that the pullback is a
\textit{universal} construction.  Given morphisms in $\Cspan$ $f: X
\ra Y$, $g: Y \ra Z$, $h: Z \ra W$, the
composites $((f \circ g) \circ h)$ and $(f \circ (g \circ h))$ are
pullbacks consisting of objects $O_1$ and $O_2$ with maps into $X$ and
$W$.  The universal property of pullbacks gives an isomorphism between
$O_1$ and $O_2$.  These isomorphisms satisfy the pentagon identity
since they are unique (in particular, both sides of the pentagon give
the same isomorphism).

It is easy to check that $\hom(X_1,X_2)$ is a category, since it
inherits all the usual properties from $\C$.
\end{remark}

\subsection{Double Categories}\label{sec:doublecat}

The idea of a double category extends that of a category into two
dimensions in a different way than does the concept of bicategory.  A
double category consists of: 
\begin{itemize}
\item a set $O$ of objects
\item \textit{horizontal} and \textit{vertical} categories, whose sets
      of objects are both $O$
\item for any diagram of the form
     \begin{equation}
      \xymatrix{
       x \ar[r]^{\phi} \ar[d]_{f} & x' \ar[d]^{f'} \\
       y \ar[r]_{\phi'} & y'
      }
\end{equation}
      a collection of square \textit{2-cells}, having horizontal
      source and target $f$ and $f'$, and vertical source
      and target $\phi$ and $\phi'$
\end{itemize} The 2-cells can be composed either horizontally or
vertically in the obvious way.  We denote a 2-cell filling the above
diagram like this:
\begin{equation}\label{xy:doublecat2cell}
 \xymatrix{
  x \ar[r]^{\phi} \ar[d]_{f} & x' \ar[d]^{f'} \\
  y \ar[r]_{\phi'} \uriicell{S} & y'
 }
\end{equation} and think of the composition of 2-cells in terms of
pasting these squares together along an edge.  The resulting 2-cell
fills a square whose boundaries are the corresponding composites of
the morphisms along its edges.

Moskaliuk and Vlassov \cite{dblcatmp} discuss the application of
double categories to mathematical physics, particularly TQFT's, and
dynamical systems with changing boundary conditions---that is, with
inputs and outputs.  Kerler and Lyubashenko \cite{KL} describe
extended TQFT's as ``double pseudofunctors'' between double
categories.  This formulation involves, among other things, a double
category of cobordisms with corners---we return to a weakening of this idea in Section \ref{sec:bicatcwc}

A double category can be thought of as an internal category in $\Cat$.
That is, it is a model of the theory of categories, denoted
$\catname{Th(\Cat)}$, in $\Cat$.  This $\catname{Th(\Cat)}$ consists
of a category containing all finite limits, and having two
distinguished objects called $\Obj$ and $\Mor$ with morphisms of the
form:
\begin{equation}
  \xymatrix{
    \Mor \ar@/^/[r]^{s} \ar@/_/[r]_{t} & \Obj \\
  }
\end{equation}
and
\begin{equation}
  \xymatrix{
    \Ob \ar[r]^{id} & \Mor  \\
  }
\end{equation}
subject to some axioms.  In particular, the composition
operation is a partially defined operation on pairs of morphisms.  In
particular, there is a collection of composable pairs of morphisms,
namely the fibre product $\catname{Pairs} = \Mor \times_{\Ob} \Mor$,
which is a pullback of the two arrows from $\Mor$ to $\Ob$.  So
$\catname{Pairs}$ is an equalizer in the following diagram:
\begin{equation}
  \xymatrix{
      &  & \Mor \ar[dr]^{t} & \\
    \catname{Pairs} \ar[r]^{i} & \Mor^2 \ar[ru]^{\pi_1} \ar[rd]_{\pi_2} & & \Obj \\
      &  & \Mor \ar[ur]^{s} & \\
  }
\end{equation} (Note that we assume the existence of pullbacks, here -
in fact, $\catname{Th(\Cat)}$ is a \textit{finite limits theory}.)  The
composition map $\circ : \catname{Pairs} \ra \Mor$ satisfies
the usual properties for composition.

There is also an identity for each object: there is a map $\Obj
\ralim^{1} \Mor$, such that for any morphism $f \in \Mor$, we
have $1_{s(f)}$ and $1_{t(f)}$ are composable with $f$, and the
composite is $f$ itself.

A model of $\catname{Th(\Cat)}$ in $\Cat$ is a (limit-preserving)
functor
\[
F : \catname{Th(\Cat)} \ra \Cat
\]
This gives a structure having a category $\catname{Ob}$ of objects and
a category $\catname{Mor}$ of morphisms, with two functors $s$
(``source'') and $t$ (``target'') satisfying the usual category
axioms. We can describe composition as a pullback
construction in this category, which makes sense since the functor
preserves finite limits (including pullbacks):
\begin{equation}\label{eq:thcatmodel}
  \xymatrix{
      &   & F(\Mor) \ar[dl]_{c_1} \ar[dr]^{c_2} \ar@/_2pc/[ddll] \ar@/^2pc/[ddrr] &  & \\
      & F(\Mor) \ar[dl]_{s} \ar[dr]^{t} &   & F(\Mor) \ar[dl]_{s} \ar[dr]^{t} & \\
    F(\Obj) &  & F(\Obj) &   & F(\Obj) \\
  }
\end{equation}

A category is a model of the theory $\catname{Th(\Cat)}$ in
$\catname{Set}$, and we understand this to mean that when two
morphisms $f$ and $g$ have the target of $f$ the same as the source of
$g$, there is a composite morphism from the source of $f$ to the
target of $g$.  In the case of a double category, we have a model of
$\catname{Th(\Cat)}$ in $\Cat$, so that $F(\Obj)$ and $F(\Mor)$ are
categories and $F(s)$ and $F(t)$ are functors, we have the same
condition for both objects and morphisms, subject to the
compatibility conditions for these two maps which any functor must
satisfy.

We thus have sets of objects and morphisms in $\catname{Ob}$, which of
course must satisfy the usual axioms.  The same is true for
$\catname{Mor}$.  The category axioms for the double category are
imposed in addition to these properties, for the composition and
identity functors.  Functoriality implies compatibility conditions
between the category axioms in the two directions.  The result is that
we can think of both the objects in $\catname{Mor}$ and the morphisms
in $\catname{Ob}$ as acting like morphisms between the objects in
$\catname{Ob}$, in a way compatible with the source and target maps.
A double category can be, and often is, thought of as including the
morphisms of two (potentially) different categories on the same
collection of objects.  These are the \textit{horizontal} and
\textit{vertical} morphisms, intuitively capturing the picture:
\begin{equation}
 \xymatrix{
  x \ar[r]^{\phi} \ar[d]_{f} & x' \ar[d]^{f'} \\
  y \ar[r]_{\hat{\phi}} & y'
 }
\end{equation}

Here, the objects in the diagram can be thought of as objects in
$F(\Obj)$, the vertical morphisms $f$ and $f'$ can be thought of as
morphisms in $F(\Obj)$ and the horizontal morphisms $\phi$ and
$\hat{\phi}$ as objects in $F(\Mor)$.  (In fact, there is enough
symmetry in the axioms for an internal category in $\Cat$ that we can
adopt either convention for distinguishing horizontal and vertical
morphisms).  However, we also have morphisms in $\Mor$.  We represent
these as two-cells, or {\textit{squares}}, like the 2-cell $S$
represented in (\ref{xy:doublecat2cell}).

The fact that the composition map $\circ$ is a functor means that
horizontal and vertical composition of squares commutes.

\subsection{Topological Examples}\label{sec:homotopy}

We can illustrate simple examples of bicategories and double
categories in a topological setting, namely homotopy theory.  This was
the source of much of the original motivation for higher-dimensional
category theory.  Moreover, as we have already remarked in Section
\ref{sec:ncob}, there are close connections between cobordism and
homotopy.  These examples will turn out to suggest how to describe
{\vdbs} of cobordisms.

Our first example is one of the original motivations for bicategories.

\begin{example}\label{ex:pathbicat} Given a space $S$ in the category
$\catname{Top}$ of topological spaces, we might wish to define a
category $\Path(S)$ whose objects are points of $X$, and whose
morphisms are paths in $S$.  That is, $Path(S)$ has:
\begin{itemize}
  \item objects: points in $S$
  \item morphisms: paths $\gamma : [m,n] \ra S$
\end{itemize} Where such a path is thought of as a morphism from
$\gamma(m)$ to $\gamma(n)$.  These are parametrized paths: so suppose
we are given two paths in $\Path(S)$, say $\gamma_1: [m_1,n_1] \ra S$
from $a$ to $b$ and $\gamma_2: [m_2,n_2] \ra S$ from $b$ to $c$.  Then
the composite is a path $\gamma_2 \circ \gamma_1 : [m_1 , n_1 + n_2 -
m_2 ] \ra S$, given by:
\begin{equation}
\gamma_2 \circ \gamma_1 (x) =  
\begin{cases} 
    \gamma_1(x) & {\text{if }} x \in [m_1, n_1) \\
    \gamma_2(x-n_1+m_2) & {\text{if }} x \in [n_1,n_1+n_2-m_2] \\
  \end{cases}
\end{equation}

This gives a well-defined category $Path(S)$, but has the awkward
feature that our morphisms are not paths, but paths \textit{equipped
with parametrization}.  So another standard possibility is to take
morphisms from $a$ to $b$ to be paths $\gamma: [0,1] \ra S$.$\gamma :
[0,1] \ra X$ such that $\gamma(0)=a$ and $\gamma(1)=b$.  The
obvious composition rule for $\gamma_1 \in
\hom(a,b)$ and $\gamma_2
\in \hom(b,c)$ is that
\begin{equation}
\gamma_2 \circ \gamma_1 (x) =
  \begin{cases} 
    \gamma_1(2x) & {\text{if }} x \in [0,\frac{1}{2}) \\
    \gamma_2(2x-1) & {\text{if }} x \in [\frac{1}{2},1] \\
  \end{cases}
\end{equation} However, this composition rule is not associative, and
resolving this involves the use of a bicategory, either implicitly or
explicitly. We get this bicategory $\Path_2(S)$, by first defining,
for $a,b \in S$, a category $\hom(a,b)$ with:
\begin{itemize}
\item objects: paths from $a$ to $b$
\item morphisms: homotopies between paths, namely a homotopy from
      $\gamma_1$ to $\gamma_2$ is $H : [0,1] \times [0,1] \ra
      S$ such that $H(x,0)=\gamma_1(x)$, $H(x,1)=\gamma(x)$,
      $H(0,y)=a$, $H(1,y)=b$ for all $(x,y) \in [0,1] \times [0,1]$.
\end{itemize}

Then we have a unit law for the identity morphism (the constant
path) at each point, and an associator for composition.  Both of these
are homotopies which reparametrize composite paths.

Finally, we note that, if we define horizontal and vertical
composition of homotopies in the same way as above (in each
component), then this composition is again not associative.  So to get
around this, we say that the bicategory we want has its
$\hom$-categories $\underline{\hom(a,b)}$, where the morphisms are
\textit{isomorphism classes} of homotopies.  The isomorphisms in
question will not be homotopies themselves (to avoid extra
complications), but rather smooth maps which fix the boundary of the
homotopy square.

We call the resulting bicategory $\Path_2(S)$.
\end{example}

A similar construction is possible for a double category.

\begin{example}\label{ex:pathdoublecat} We have seen that a double
category it is rather analogous to a bicategory, so we would like to
construct one analogous to the bicategory in Example
\ref{ex:pathbicat}.  To do this, we construct a model having the following:
\begin{itemize}
\item A category $\Obj$ of objects is the path \textit{category} $Path(S)$:
\item A category $\Mor$ of morphisms: this has the following data:
  \begin{itemize}
  \item objects: paths $\gamma : [m,n]$ in $S$
  \item morphisms: homotopies $H: [p,q] \times [m,n]$ between paths
        (these have source and target maps which are just $s: H(-,-)
        \ra H(-,m)$ and $t: H(-,-) \ra H(-,n)$.
  \end{itemize} These categories have source and target maps $s$ and
  $t$ which are functors from $\Mor$ to $\Obj$.  The object map for $s$
  is just evaluation at $0$, and for $t$ it is evaluation at $1$.  The
  morphism maps for these functors are $s : H(-,-) \ra H(p,-)$
  and $t : H(-,-) \ra H(q,-)$.
\end{itemize}

We call the result the double category of homotopies, $\catname{H}(S)$.
\end{example}

We observe here that the double category $\catname{H}(S)$ is similar
to the bicategory $\Path_2(S)$ in one sense.  Both give a picture in
which objects are points in a topological space, morphisms are
1-dimensional objects (paths), and higher morphisms involve
2-dimensional objects (homotopies).  There are differences, however:
the most obvious is that $\Path_2(S)$ involves only homotopies with
fixed endpoints: its 2D objects are {\textit{bigons}}, whereas in
$\catname{H}(S)$ the 2D objects are ``squares'' (or images of
rectangles under smooth maps).

A more subtle difference, however, is that, in order to make
composition strictly associative in $\catname{H}(S)$, it was necessary to
change how we parametrize the homotopies.  There are no associators
here, and so we make sure composition is strict by not rescaling our
source object (the product of two intervals) as we did in
$\Path_2(S)$.

This is rather unsatisfactory, and in fact improving it leads to a
general definition of a {\textit{\db}}, which has a large class of
examples, namely \textit{double cospans}.  A special, restricted case
of these is the {\db} of cobordisms with corners we want.

\section{Verity Double Bicategories}\label{chap:doublebicat}

The term {\db} seems to have been originally introduced by Dominic
Verity \cite{verity}, and the structure it refers to is the one we
want to use.  There is some ambiguity here since the term {\db}
appears to describe is an internal bicategory in $\Bicat$ (the
category of all bicategories).  This is analogous to the definition of
double category.  Indeed, it is what we will mean by a \db, and we
discuss these in Section \ref{sec:internalbicat}.   Since the two are
closely related, and both will be important for us, we will refer to
{\dbs} in the sense of Verity by the term {\vdbs}, while reserving
\textit{\db} for the former.  For more discussion of the relation
between these, see Section \ref{sec:internalbicat}.

We wish to describe a structure which is sufficient to capture the
possible compositions of cobordisms with corners just as $\iiCob$ does
for cobordisms.  These have ``horizontal'' composition along the
manifolds with boundary which form their source and target.  They also
have ``vertical'' composition along the boundaries of those manifolds
and of the cobordisms joining them (which, together, again form
cobordisms) .  However, to allow the boundaries to vary, we do not
want to consider them as diffeomorphism classes of cobordisms, but
simply as cobordisms.  However, composition is then not strictly
associative, but only up to diffeomorphism.

Thus, we want something like a double category, but with weakened
axioms, just as bicategories were defined by weakening those for a
category.  The concept of a ``weak double category'' has been defined
(for instance, see Marco Grandis and Robert Par\'e \cite{GP1}, and
Martins-Ferreira's \cite{MF} discussion of them as
``pseudo-categories'').  Thomas Fiore \cite{fiore} describes these as
``Pseudo Double Categories'', arising by ``categorification'' of the
theory of categories, and describes examples motivated by conformal
field theory.  A detailed discussion is found in Richard Garner's
Ph.D. thesis \cite{garner}. In these cases, the weakening only occurs
in only one direction---either horizontal or vertical.  That is, the
associativity of composition, and unit laws, in that direction apply
only up to certain higher morphisms, called \textit{associators} and
\textit{unitors}.  In the other direction, the category axioms hold
strictly.  In a sense, this is because the weakening uses the squares
of the double category as 2-morphisms: specifically, squares with two
sides equal to the identity.  Trying to do this in both directions
leads to some difficulties.

In particular, if we have associators for horizontal morphisms given
by squares of the form:
\begin{equation}
\xymatrix{
a \ar[r]^{f;g} \ar[d]_{1_a} \drriicell{a_{f,g,h}} & c \ar[r]^{h} & d \ar[d]^{1_d} \\
a \ar[r]^{f} & b \ar[r]^{g;h} & d
}
\end{equation}
then unless composition of vertical morphisms is strict, then to make
a equation (for instance, the pentagon equation) involving this
square, we would need to use unit laws (or associators) in the
vertical direction to perform this composition.  This would again be a
square with identities on two sides, and the problem arises again.  In
fact, there is no consistent way to do this.  Instead, we need to
introduce a new kind of 2-morphism separate from the squares, as we
shall see in Section \ref{sec:dbdef}.  The result is what Dominic
Verity has termed a {\db} \cite{verity}.

The problem of weakening the concept of a double category so that the
unit and associativity properties hold up to higher-dimensional
morphisms can be contrasted with a different approach.  One might
instead try to combine the notions of \textit{bicategory} and
\textit{double category} in a different way.  This is by ``doubling''
the notion of bicategory, in the same way that double categories did
with the notion of category.  Just as a double category is an internal
category in $\Cat$, the result would be an internal bicategory in $\Bicat$. 

We would like to call this a \textit{\db}: however, this term has
already been used by Dominic Verity to describe the structure we will
mainly be interested in.  Since the former concept is also important
for us in certain lemmas, and is most naturally called a \db, we will
refer to the latter as a \textit{\vdb}.  For more discussion of the
relation between these, see Section \ref{sec:internalbicat}.

\subsection{Definition of a Verity Double Bicategory}\label{sec:dbdef}

The following definition of a {\vdb} is due to Dominic Verity
\cite{verity}, and is readily seen as a natural weakening of the
definition of a double category.  Just as the concept of
\textit{bicategory} weakens that of \textit{2-category} by weakening
the associative and unit laws, {\vdbs} will do the same for double
categories.  The following definition can be contrasted with that for
a double category in Section \ref{sec:doublecat}.

\begin{definition}\label{def:doublebicat} (\textbf{Verity}) A \textbf{\vdb}
$\catname{C}$ is a structure consisting of the following data:

\begin{itemize}
  \item a class of \textbf{objects} $\Obj$,

  \item \textbf{horizontal} and \textbf{vertical bicategories} $\Hor$
        and $\Ver$ having $\Obj$ as their objects

  \item for every square of horizontal and vertical morphisms of the form
  \begin{equation}
    \xymatrix{
      a \ar[r]^{h} \ar[d]_{v} & b \ar[d]^{v'} \\
      c \ar[r]^{h'} & d \\
    }
  \end{equation}
  a class of \textbf{squares} $\Squ$, with maps $s_h, t_h : \Squ
  \ra \Mor(\Hor)$ and $s_v, t_v : \Squ \ra \Mor(\Ver)$,
  satisfying an equation for each corner, namely:
  \begin{eqnarray}\label{eq:squarestmaps}
    s(s_h)&=&s(s_v) \\
    \nonumber t(s_h)&=&s(t_v) \\
    \nonumber s(t_h)&=&t(s_v) \\
    \nonumber t(t_h)&=&t(t_v)
  \end{eqnarray}
\end{itemize}
  The squares should have horizontal and vertical composition
  operations, defining the vertical composite $F \otimes_V G$
  \begin{equation}\label{eq:squarevertcomp}
    \xymatrix{
      x \ar[r] \ar[d] & x' \ar[d] \\
      y \ar[r] \ar[d] \uriicell{F} & y'\ar[d] \\
      z \ar[r] \uriicell{G} & z' \\
    } \qquad = \qquad
    \xymatrix{
      x \ar[r] \ar[d] & x' \ar[d] \\
      z \ar[r] \uriicell{F\otimes_V G} & z'\\
    }
  \end{equation}
  and horizontal composite $F \otimes_H G$:
  \begin{equation}\label{eq:squarehorizcomp}
    \xymatrix{
      x \ar[r] \ar[d] & y \ar[d] \ar[r] & z \ar[d] \\
      x' \ar[r] \uriicell{F} & y' \ar[r] \uriicell{G} & z'
    } \qquad = \qquad
    \xymatrix{
      x \ar[r] \ar[d] & z \ar[d] \\  
      x' \ar[r] \uriicell{F \otimes_H G} & z' \\
    }
  \end{equation} The composites have the usual relation to source and target
  maps, satisfy the interchange law
  \begin{equation}\label{eq:squareinterchangelaw}
    (F \otimes_V F') \otimes_H (G \otimes_V G') = (F \otimes_H G) \otimes_V (F' \otimes_H G')
  \end{equation} and there is a unit for composition of squares: 
  \begin{equation}
    \xymatrix{
      x \ar[r]^{1_x} \ar[d]_{f} & x \ar[d]^{f} \\
      y \ar[r]^{1_y} \uriicell{1_f} & y \\
    }
  \end{equation}
  (and similarly for vertical composition).

  There is a left and right action by the horizontal
  and vertical 2-morphisms on $\Squ$, giving $F \star_V \alpha$,
  \begin{equation}
    \xymatrix{
      x \ar[r] \ar[d] & y \ar[d]^{ }="1" \ar@/^2pc/[d]^{ }="0" \\
      x' \ar[r] \uriicell{F} & y' \\
      \ar@{<=}"0" ;"1"^{\alpha}
    } \qquad = \qquad
    \xymatrix{
      x \ar [r] \ar[d] & y \ar[d] \\
      x' \ar[r] \uriicell{F \star_V \alpha} & y'
    }
  \end{equation}
  (and similarly on the left) and $F \star_H \alpha$,
  \begin{equation}
    \xymatrix{
      x \ar[r]^{ }="1" \ar@/^2pc/[r]^{ }="0" \ar[d] & y \ar[d]  \\
      x' \ar[r] \uriicell{F} & y' \\
      \ar@{=>}"0" ;"1"^{\alpha}
    } \qquad = \qquad
    \xymatrix{
      x \ar [r] \ar[d] & y \ar[d] \\
      x' \ar[r] \uriicell{\alpha \star_H F} & y'
    }
  \end{equation} The actions also satisfy interchange laws:
  \begin{equation}
    (F \otimes_H F') \star_H (\alpha \otimes_V \alpha') = (F \star_H \alpha) \otimes_h (F' \star_H \alpha')
  \end{equation}
(and similarly for the vertical case) and are compatible with composition: 
  \begin{equation}\label{eq:actioncompat}
    (F \otimes_H G) \star_V \alpha = F \otimes_H (G \star_V \alpha)
  \end{equation} (and analogously for vertical composition).  They
  also satisfy additional compatibility conditions: the left and right
  actions of both vertical and horizontal 2-morphisms satisfy the
  ``associativity'' properties,
 \begin{equation}
    \alpha \star (S \star \beta) = (\alpha \star S) \star \beta
  \end{equation}
  for both $\star_H$ and $\star_V$.  Moreover, horizontal
  and vertical actions are independent:
  \begin{equation}\label{eq:actionindep}
    \alpha \star_H (\beta \star_V S) = \beta \star_V (\alpha \star_H S)
  \end{equation}
  and similarly for the right action. 

Finally, the composition of squares agrees with the associators for
composition by the action in the sense that given three composable
squares $F$, $G$, and $H$:
\begin{equation}\label{eq:assocaction}
\xymatrix{
      x \ar[d] \ar[rrr]^{h \circ (g \circ f)}="0" \ar@/^3pc/[rrr]^{(h \circ g) \circ f}="1" & & & y \ar[d] \\
      x' \ar[rrr]_{h' \circ (g' \circ f')}  & \uriicell{(F \otimes_H G) \otimes_H H} & & y' \\
      \ar@{=>}"0" ;"1"^{a_{f,g,h}}
} \qquad = \qquad
\xymatrix{
      x \ar[rrr]^{(h \circ g) \circ f} \ar[d] & & & y \ar[d] \\
      x' \ar[rrr]_{(h' \circ g') \circ f}="0" \ar@/_3pc/[rrr]_{h' \circ (g' \circ f)}="1"  & \uriicell{F \otimes_H (G \otimes_H H)} & & y'\\
      \ar@{=>}"0" ;"1"^{a_{f',g',h'}}
}
\end{equation}
and similarly for vertical composition.  Likewise, unitors in the
horizontal and vertical bicategories agree with the identity for
composition of squares:
\begin{equation}\label{eq:unitaction}
\xymatrix{
      x \ar[d]_{g} \ar[r]_{f} \ar@/^2pc/[rr]^{f}="1" \ar@{}[rr]^{}="0" & y \ar[d]^{g'} \ar[r]_{1_y} & y \ar[d] \ar[d]^{g'} \\
      x' \ar[r]_{f'} \uriicell{F}  & y' \ar[r]_{1_{y'}} \uriicell{1_{g'}} & y' \\
      \ar@{=>}"0"+<0ex,+1.5ex>;"1"+<0ex,-1.5ex>^{l_f}
} \qquad = \qquad
\xymatrix{
      x \ar[d]_{g} \ar[r]^{1_x} &  \ar[d]_{g} \ar[r]^{f} & y \ar[d]^{g'} \\
      x' \ar[r]^{1_{x'}} \ar@/_2pc/[rr]_{f'}="1" \ar@{}[rr]^{}="0" \uriicell{1_{g}}  & x' \uriicell{F} \ar[r]^{f'} & y' \\
      \ar@{=>}"0"+<0ex,-1.5ex> ;"1"+<0ex,+2.5ex>^{r_{f'}}
}
\end{equation}
and similarly for vertical units.
\end{definition}

We will see in Chapter \ref{chap:cobcorn} that the higher categories
defined this way are well suited to dealing with cobordisms with
corners.  In Section \ref{sec:internalbicat} we will consider how this
definition arises as a special case of a broader concept of {\db}
which we define there.  For now, in Section \ref{sec:equiv}, we will
consider how {\vdbs} can give rise to ordinary bicategories.

\subsection{Bicategories from Double Bicategories}\label{sec:equiv}

There are numerous connections between double categories and
bicategories (or their strict form, 2-categories).  One is Ehresmann's
double category of quintets, relating double categories to
2-categories: a double category by taking the squares to be
2-morphisms between composite pairs of morphisms, such as $\alpha:
g'\circ f \ra f' \circ g$.

Furthermore, it is well known that double categories can yield
2-categories in three different ways.  Two obvious cases are when
there are only identity horizontal morphisms, or only vertical
morphisms, respectively, so that squares simply collapse into bigons
with the two nontrivial sides.  Notice that it is also true that a
{\vdb} in which $\Hor$ is trivial (equivalently, if $\Ver$ is trivial)
is again a bicategory.  The squares become 2-morphisms in the obvious
way, the action of 2-morphisms on squares then is just composition,
and the composition rules for squares and bigons are the same.  The
result is clearly a bicategory.

The other, less obvious, case, is when the horizontal and vertical
categories on the objects are the same: this is the case of
\textit{path-symmetric} double categories, and the recovery of a
bicategory was shown by Brown and Spencer \cite{brownspencer}.  Fiore
\cite{fiore} shows how their demonstration of this fact is equivalent
to one involving \textit{folding structures}.

In this case we again can interpret squares as bigons by composing the
top and right edges, and the left and bottom edges.  Introducing
identity bigons completes the structure.  These new bigons have a
natural composition inherited from that for squares.  It turns out
that this yields a structure satisfying the definition of a
2-category.  Here, our goal will be to show half of an analogous
result, that a {\vdb} similarly gives rise to a bicategory when the
horizontal and vertical bicategories are equal.  We will also show
that a double bicategory for which the horizontal (or vertical)
bicategory is trivial can be seen as a bicategory.  The condition that
$\Hor = \Ver$ holds in our general example of double cospans: both
horizontal and vertical bicategories in any $\iiCCosp_0$ are just
$\CCosp$.

\begin{theorem}\label{thm:equiv} Any {\vdb} $(\Obj, \Hor, \Ver, \Squ, \otimes_H, \otimes_V,
\star_H, \star_V)$ for which $\Hor = \Ver$ produces a bicategory by taking squares to be 2-cells.
\end{theorem}
\begin{proof}
We begin by defining the data of this bicategory, which we call
$\catname{B}$.  Its objects and morphisms are the same as those of
$\Hor$ (equivalently, $\Ver$).  We describe the 2-morphisms by
observing that $\catname{B}$ must contain all those in $\Hor$
(equivalently, $\Ver$), but also some others, which correspond to the
squares in $\Squ$.

In particular, given a square
\begin{equation}
 \xymatrix{
  a \ar[r]^{f} \ar[d]_{g} & b \ar[d]^{g'} \\
  c \ar[r]_{f'} \uriicell{S} & d
 }
\end{equation}
there should be a 2-morphism
\begin{equation}
 \xymatrix{
   a \ar@/^1pc/[rr]^{g'\circ f}="0" \ar@/_1pc/[rr]_{f'\circ g}="1" && d \\
   \ar@{=>}"0"+<0ex,-2.5ex> ;"1"+<0ex,+2.5ex>^{S}
 }
\end{equation}

The composition of squares corresponds to either horizontal or
vertical composition of 2-morphisms in $\catname{B}$, and the relation
between these two is given in terms of the interchange law in a
bicategory:

Given a composite of squares,
\begin{equation}
  \xymatrix{
    x \ar[r]^{f} \ar[d]_{\phi_x} & y \ar[d]^{\phi_y} \ar[r]^{g} & z \ar[d]^{\phi_z} \\
    x' \ar[r]_{f'} \uriicell{F} & y' \ar[r]_{g'} \uriicell{G} & z'
 }
\end{equation}
there will be a corresponding diagram in $\catname{B}$:
\begin{equation}
 \xymatrix{
    x \ar@{}[rr]^{}="1"\ar[r]^{f} \ar@/_2pc/[rr]_{\phi_x \circ f'}="0" & y \ar@{}[rr]^{}="3" \ar[r]^{\phi_y} \ar@/^2pc/[rr]^{\phi_z \circ g}="2" & y' \ar[r]^{g'} & z' \\
   \ar@{=>}"1"+<0ex,-1.5ex> ;"0"+<0ex,+2.5ex>^{F}
   \ar@{=>}"2"+<0ex,-2.5ex> ;"3"+<0ex,+1.5ex>^{G}
 }
\end{equation}

Using horizontal composition with identity 2-morphisms
(``whiskering''), we can write this as a vertical composition:
\begin{equation}
 \xymatrix{
   x \ar@/^2pc/[rrr]^{\phi_z \circ g \circ f}="0" \ar[rrr]_{g' \circ \phi_y \circ f}="1" \ar@/_2pc/[rrr]_{g' \circ f' \circ \phi_x}="2" & & & z' \\
   \ar@{=>}"0"+<0ex,-2.5ex> ;"1"+<0ex,+2.5ex>^{G \circ \opname{1}_{f}}
   \ar@{=>}"1" ;"2"^{\opname{1}_{g'} \circ F}
 }
\end{equation}

So the square $F \otimes_H G$ corresponds to $(\opname{1} \circ G)
\cdot (F \circ \opname{1})$ for appropriate identities
$\opname{1}$.  Similarly, the vertical composite of $F' \otimes_V G'$
must be the same as $(\opname{1} \circ F) \cdot (G \circ
\opname{1})$.  Thus, every composite of squares, which can all be
built from horizontal and vertical composition, gives a corresponding
composite of 2-morphisms in $\catname{B}$, which are generated by
those corresponding to squares in $\Squ$, subject to the relations
imposed by the composition rules in a bicategory.

To show the {\vdb} gives a bicategory, it now suffices to show that
all such 2-morphisms not already in $\Hor$ arise as squares (that is,
the structure is closed under composition).  So suppose we have any
composable pair of 2-morphisms which arise from squares.  If the
squares have an edge in common, then we have the situation depicted
above (or possibly the corresponding form in the vertical direction).
In this case, the composite 2-morphism corresponds exactly to the
composite of squares, and the axioms for composition of squares ensure
that all 2-morphisms generated this way are already in our bicategory.
In particular, the unit squares become unit 2-morphisms when composed
with left and right unitors.

Now, if there is no edge in common to two squares, the 2-morphisms in
$\catname{B}$ must be made composable by composition with identities.
In this case, all the identities can be derived from 2-morphisms in
$\Hor$, or from identity squares in $\Squ$ (inside commuting
diagrams).  Clearly, any identity 2-morphism can be factored this way.
Then, again, the composite 2-morphisms in $\catname{B}$ will
correspond to the composite of all such squares and 2-morphisms in
$\Squ$ and $\Hor$.

Finally, the associativity condition (\ref{eq:assocaction}) for the
action of 2-morphisms on squares ensures that composition of squares
agrees with that for 2-morphisms, so there are no extra squares from
composites of more than two squares.
\end{proof}
 
This allows us to think of $\nCob$ not only as a {\vdb}, but in the more
familiar form of a bicategory.

It is also worth considering here the situation of a double bicategory
with horizontal bicategory trivial (i.e. in which all horizontal
morphisms and 2-morphisms are identities).  In this case, one can
define a 2-morphism from a square with and bottom edges being
identities, whose source is the object whose identity is the
corresponding edge, and similarly for the target.  The composition
rules for squares in the vertical direction, then, are just the same
as those for a bicategory.  Likewise, the axioms for action of a
2-morphism on a square reduce to the composition laws for a bicategory
if one replaces the square by a 2-cell.

Next we describe the class of examples we will use
to develop a double bicategory of cobordisms with corners.

\subsection{Double Cospans}\label{sec:dblspan}

Now we construct a class of examples of double bicategories.  These
examples are analogous to the example of bicategories of spans,
discussed in Section \ref{sec:spanbicat}.  These span-ish examples of
{\vdbs} are will give the {\vdb} of cobordisms with corners as a
special case, which is similar in flavour to the topological examples
of bicategories and double categories in Section \ref{sec:homotopy}.
However, these will be based on \textit{cospans}.  Cospans in $\C$ are
the same as spans in the opposite category, $\Cop$.  In Remark
\ref{thm:spanbicat} we described B\'enabou's demonstration that
$\Cspan$ is a bicategory for any category $\C$ with
pullbacks. Similarly, there is a bicategory of cospans in a category
$\C$, with pushouts.

There will be an analogous fact about our example of a {\db}: double
cospans, described explicitly in Section \ref{sec:internalbicat}.
Here, we are interested in a more restricted structure:

\begin{definition}\label{def:cspan20}For a category $\C$ with finite
limits, the {\vdb} $\iiCCosp_0$, has:
\begin{itemize}
\item the objects are objects of $\C$
\item the horizontal and vertical bicategories $\Hor = \Ver$ are both
      equal to a sub-bicategory of $\CCosp$, which includes only
      invertible cospan maps
\item the squares are isomorphism classes of commuting
      diagrams of the form:
\begin{equation}\label{xy:cspan2}
  \xymatrix{
    X \ar[d]^{i_X} \ar[r]_{\iota_X} & S \ar[d]^{I} & Y \ar[l]^{\iota_Y} \ar[d]_{i_Y} \\
    T_X \ar[r]_{J_X}  & M  & T_Y   \ar[l]^{J_Y} \\
    X' \ar[u]_{i_Y'} \ar[r]_{\iota_{X'}} & S' \ar[u]_{I'} & Y' \ar[u]^{i_{Y'}} \ar[l]^{\iota_{Y'}}
  }
\end{equation}

\end{itemize} where two diagrams of the form (\ref{xy:cspan2}) are
isomorphic if they differ only in the middle objects, say $M$ and
$M'$, and their maps to the edges, and if there is an isomorphism $f:M
\ra M'$ making the combined diagram commute.

The action of 2-morphisms $\alpha$ in $\Hor$ and $\Ver$ on squares is by
composition in diagrams of the form:
\begin{equation}\label{xy:cspan2action}
  \xymatrix{
      & S_2 \ar[dl]_{\pi_1} \ar[dr]^{\pi_2} & \\
    X & S_1 \ar[l]_{\pi_1} \ar[r]^{\pi_2} \ar[u]^{\alpha} & Y \\
    T_X \ar[u]^{p_1} \ar[d]_{p_2} & M \ar[u]^{P_1} \ar[d]_{P_2} \ar[l]_{\Pi_1} \ar[r]^{\Pi_2} & T_Y \ar[u]_{p_1} \ar[d]^{p_2} \\
    X'& S' \ar[l]_{\pi_1} \ar[r]^{\pi_2}  & Y'
  }
\end{equation} (where the resulting square is as in \ref{xy:cspan2},
with $S_2$ in place of $S$ and $\alpha \circ P_1$ in place of $P_1$).

Composition (horizontal or vertical) of squares of cospans is, as in
$\iiCCosp$, given by composition (by pushout) of the three spans of
which the square is composed.  The composition operators for diagrams
of cospan maps are by the usual ones in $\CCosp$.
\end{definition}

\begin{remark} Notice that $\Hor$ and $\Ver$ as defined are indeed
bicategories: eliminating all but the invertible 2-morphisms leaves a
collection which is closed under composition by pushouts.
\end{remark}

We show more fully that this is a {\vdb} in Theorem
\ref{thm:cspanthm}, but for now we note that the definition of
horizontal and vertical composition of squares is defined on
equivalence classes.  One must show that this is well defined.  We
will get this result indirectly as a result of Lemmas
\ref{lemma:span2bicat} and \ref{lemma:doublebicat}, but it is
instructive to see directly how this works in $\CCosp$.

\begin{lemma}The composition of squares in Definition \ref{def:cspan20} is well-defined.
\end{lemma}
\begin{proof}
Suppose we have two representatives of a square, bounded by horizontal
cospans $(S,\pi_1,\pi_2)$ from $X$ to $Y$ and $(S',\pi_1,\pi_2)$ from
$X'$ to $Y'$, and vertical cospans $(T_X,p_1,p_2)$ from $X$ to $X'$ and
$(T_Y,p_1,p_2)$ from $Y$ to $Y'$.  The middle objects $M_1$ and $M_2$
as in the diagram (\ref{xy:cspan2}).  If we also have a composable
diagram---one which coincides along an edge (morphism in $\Hor$ or
$\Ver$) with the first---then we need to know that the pushouts are
also isomorphic (that is, represent the same composite square).

In the horizontal and vertical composition of these squares, the maps
from the middle object $M$ of the new square to the middle objects of
the new sides (given by composition of cospans) arise from the universal
property of the pushouts on the sides being composed (and the induced
maps from $M$ to the corners, via the maps in the cospans on the other
sides).  Since the middle objects are defined only up to isomorphism
class, so is the pushout: so the composition is well defined, since
the result is again a square of the form (\ref{xy:cspan2}).
\end{proof}

We use this, together with Lemmas \ref{lemma:span2bicat} and
\ref{lemma:doublebicat}, (proved in Section \ref{sec:internalbicat})
to show the following:

\begin{theorem}\label{thm:cspanthm} If $\C$ is a category with finite
colimits, then $\iiCCosp_0$ is a {\vdb}.
\end{theorem}
\begin{proof}
In the construction of $\iiCCosp_0$, we take isomorphism classes of
double cospans as the squares.  We also restrict to invertible cospan maps
in the horizontal and vertical bicategories.

That is, take 2-isomorphism classes of morphisms in $\M$ in the \db, where
the 2-isomorphisms are invertible cospan maps, in both horizontal and
vertical directions.  We are then effectively discarding all morphisms
and 2-morphisms in $\B$, and the 2-morphisms in $\M$ except for the
invertible ones.  In particular, there may be ``squares'' of the form
(\ref{xy:cspan2}) in $\iiCCosp$ with non-invertible maps joining their
middle objects $M$, but we have ignored these, and also ignore
non-invertible cospan maps in the bicategories on the edges.  Thus, we
consider no diagrams of the form (\ref{xy:cospanmap1}) except for
invertible ones, in which case the middle objects (say, $M$ and $M'$) are
representatives of the same isomorphism class.  Similar reasoning
applies to the 2-morphisms in $\B$.

The resulting structure we get from discarding these will again be a
\db. In particular, the new $\M$ and $\B$ will be bicategories, since
they are, respectively, just a category and a set made into a discrete
bicategory by adding identity morphisms or 2-morphisms as needed.  On
the other hand, for the composition, source and target maps to be
weak 2-functors amounts to saying that the structures built from the
objects, morphisms, and 2-cells respectively are again bicategories,
since the composition, source, and target maps satisfy the usual
axioms.  But the same argument applies to those built from the
morphisms and 2-cells as within $\M$ and $\B$.  So we have a \db.

Next we show that the horizontal and vertical action conditions
(Definition \ref{def:actionconds} of section \ref{sec:decatfy}) hold
in $\iiCCosp$.  A square in $\iiCCosp$ is a diagram of the form
(\ref{xy:cspan2}), and a 2-cell is a map of cospans.  Given a square
$M_1$ and 2-cell $\alpha$ with compatible source and targets as in the
action conditions, we have a diagram of the form shown in
(\ref{xy:cspan2action}).  Here, $M_1$ is the square diagram at the
bottom, whose top row is the cospan containing $S_1$.  The 2-cell
$\alpha$ is the cospan map including the arrow $\alpha: S_1 \ra
S_2$.  There is a unique square built using the same objects as $M_1$
except using the cospan containing $S_2$ as the top row.  The map to
$S_2$ from $M$ is then $\alpha \circ P_1$.

To satisfy the action condition, we want this square $M_2$, which is
the candidate for $M_1 \star_V \alpha$, to be unique.  But suppose
there were another $M'_2$ with a map to $S_2$.  Since we are in
$\iiCCosp_0$, $\alpha$ must be invertible, which would give a map from
$M'_2$ to $S_1$.  We then find that $M'_2$ and $M_2$ are
representatives of the same isomorphism class, so in fact this is the
same square.  That is, there is a unique morphism in $\B$ taking
$M_1$ to $M_2$ (a diagram of the form \ref{xy:cospanmap2}, oriented
vertically) with invertible cospan maps in the middle and bottom
rows.  This is the unique filler for the pillow diagram required by
definition \ref{def:actionconds}.

The argument that $\iiCCosp_0$ satisfies the action compatibility
condition is similar.

So $\iiCCosp_0$ is a {\db} in which, there there is at most one unique
morphism in $\M$, and at most unique morphisms and 2-morphisms in
$\B$, for any specified source and target, and the horizontal and
vertical action conditions hold.  So $\iiCCosp_0$ can be interpreted
as a {\vdb} (Lemma \ref{lemma:doublebicat}).
\end{proof}

\begin{remark}
We observe here that the compatibility condition (\ref{eq:assocaction})
relating the associator in the horizontal and vertical bicategories to
composition for squares can be seen from the fact that the associators
are maps which come from the universal property of pushouts.  This is
by the parallel argument to that we gave for spans in Section
\ref{sec:spanbicat}.  The same argument applies to the middle objects
of the squares, and gives associator isomorphisms for that
composition.  Since these become the identity when we reduce to
isomorphism classes, we get a commuting pillow as in
(\ref{eq:assocaction}).  A similar argument shows the compatibility
condition for the unitor, (\ref{eq:unitaction}).
\end{remark}

Note that the analogous theorem beginning with a category $\C$ with
finite limits and using spans is equivalent to this case, by taking
$\Cop$.

In Section \ref{sec:cobcat} we use a similar argument to obtain a
{\vdb} of cobordisms with corners.  First, however, we must see how
these are defined.  This is the task of Chapter \ref{chap:cobcorn}.
In Appendix \ref{sec:doublecospan} we show that $\CCosp$ is a \vdb.  For
now, we will examine how cobordisms form a special topological example
of this sort of \vdb.

\section{Cobordisms With Corners}\label{chap:cobcorn}

Our motivation here for studying {\vdbs} is to provide the right
formal structure for our special example of higher categories of
cobordisms.  The objects in these categories are manifolds of some
dimension, say $k$.  In this case, the morphisms are
$(k+1)$-dimensional cobordisms between these manifolds: that is,
manifolds with boundary, such that the boundary decomposes into two
components, with one component as the source, and one as the target.
The 2-cells are equivalence classes $(k+2)$-dimensional cobordisms
between $(k+1)$-dimensional cobordisms: these can be seen as manifolds
with corners, where the corners are the $k$-dimensional objects.
Specifically, with these as with the cobordisms in our definition of
$\nCobi$, only the highest-dimensional level consists of isomorphism
classes.  This means that composition of the horizontal and vertical
cobordisms will need to be weak, which is why we use {\vdbs} as
defined in Definition \ref{def:doublebicat}.

Observe that we could continue building a ladder in which, at each
level, the $j+1$-cells are cobordisms between the $j$-cells, which are
cobordisms between the $(j-1)$-cells.  The two levels we consider here
are enough to give a {\vdb} of $n$-dimensional cobordisms with
corners, where we think of the top dimension ($k+2$ in the above) as
$n$.  We will see that these can be construed using the double cospan
construction of Section \ref{sec:dblspan}.

\subsection{Collars on Manifolds with Corners}\label{sec:collarman}

Here we will use a modified form of our construction from Section
\ref{sec:dblspan} of a {\vdb} $\iiCCosp$ in order to show an example
of a {\vdb} of cobordisms with corners, starting with $\C$ the
category of smooth manifolds.  To begin with, we recall that a smooth
manifold with corners is a topological manifold with boundary,
together with a certain kind of $C^\infty$ structure.  In particular,
we need a maximal compatible set of coordinate charts $\phi : \Omega
\ra [0,\infty)^n$ (where $\phi_1$, $\phi_2$ are compatible if
$\phi_2 \circ \phi_1^{-1}$ is a diffeomorphism).  The fact that the
maps are into the positive sector of $\mathbbm{R}^n$ distinguishes a
manifold with corners from a manifold.

J\"anich \cite{jan68} introduces the notion of $\br{n}$-manifold,
reviewed by Laures \cite{laures}.  This is build on a manifold with faces:
\begin{definition}A \textit{face} of a manifold with corners is the
closure of some connected component of the set of points with just one
zero component in any coordinate chart.  An
$\br{n}$-\textit{manifold} is a manifold with faces together with an
$n$-tuple $(\partial_0 M, \dots, \partial_{n-1}M)$ of faces of $M$,
such that
\begin{itemize}
\item $\partial_0 M \cup \dots \partial_{n-1} M = \partial M$
\item $\partial_i M \cap \partial_j M$ is a face of $\partial_i M$ and
      $\partial_j M$
\end{itemize}
\end{definition}

The case we will be interested in here is the case of
$\br{2}$-manifolds.  In this notation, a $\br{0}$-manifold is just a
manifold without boundary, a $\br{1}$-manifold is a manifold with
boundary, and a $\br{2}$-manifold is a manifold with corners whose
boundary decomposes into two components (of codimension 1), whose
intersections form the corners (of codimension 2).  We can think of
$\partial_0 M$ and $\partial_1 M$ as the ``horizontal'' and
``vertical'' part of the boundary of $M$.

\begin{example}Let $M$ be the solid 3-dimensional illustrated in Figure
\ref{fig:cobcorners-labelled}.  The boundary decomposes into 2-dimensional
manifolds with boundary.  Denote by $\partial_0 M$ the boundary
component consisting of the top and bottom surfaces, and $\partial_1
M$ be the remaining boundary component (a topological annulus).

In this case, $\partial_0 M$ is the disjoint union of the manifolds
with corners $S$ (two annuli) and $S'$ (topologically a three
punctured sphere); $\partial_1 M$ is the disjoint union of two
components, $T_X$ (which is topologically a three-punctured sphere)
and $T_Y$ (topologically a four-punctured torus).

Then we have $\partial_0 M \cup \partial_1 M = \partial M$.  Also,
$\partial_0 M \cap \partial_1 M$ is a 1-dimensional manifold without
boundary, which is a face of both $\partial_0 M$ and $\partial_1 M$
(in fact, the shared boundary).  In particular, it is the disjoint
union $X \cup Y \cup X' \cup Y'$.
\end{example}

We have described a {\vdb} of double cospans in a category with all
pushouts.  We could then form such a system of cobordisms with corners
in a category obtained by co-completing $\catname{Man}$, so that all
pushouts exist.  The problem with this is that the pushout of two
cobordisms $M_1$ and $M_2$ over a submanifold $S$ included in both by
maps $S \ralim^{i_1} M_1$ and $S \ralim^{i_2} M_2$ may
not be a cobordism.  If the submanifolds are not on the boundaries,
certainly the result may not even be a manifold: for instance, two
line segments with a common point in the interior.  So to get a {\vdb}
in which the morphisms are smooth manifolds with boundary, certainly
we can only consider the case where we compose two cobordisms by a
pushout along shared submanifolds $S$ which are components of the
boundary of both $M_1$ and $M_2$.

However, even if the common submanifold is at the boundary, there is
no guarantee that the result of the pushout will be a smooth manifold.
In particular, for a point $x \in S$, there will be a neighborhood $U$
of $x$ which restricts to $U_1 \subset M_1$ and $U_2 \subset M_2$ with
smooth maps $\phi_i : U_i \ra [0,\infty)^n$ with $\phi_i(x)$
on the boundary of $[0,\infty)^n$ with exactly one coordinate equal to
$0$.  One can easily combine these to give a homeomorphism $\phi: U
\ra \mathbbm{R}^n$, but this will not necessarily be a
diffeomorphism along the boundary $S$.

To solve this problem, we use the \textit{collaring theorem}: For any
smooth manifold with boundary $M$, $\partial M$ has a \textit{collar}:
an embedding $f : \partial M \times [0,\infty) \ra M$, with
$(x,0) \mapsto x$ for $x \in \partial M$.  This is a well-known result
(for a proof, see e.g. \cite{hirsch}, sec.  4.6).  It is an easy
corollary of this usual form that we can choose to use the interval
$[0,1]$ in place of $[0,\infty)$ here.

Gerd Laures (\cite{laures}, Lemma 2.1.6) describes a generalization of
this theorem to $\br{n}$-manifolds, so that for any $\br{n}$-manifold
$M$, there is an $n$-dimensional cubical diagram
($\br{n}$-\textit{diagram}) of embeddings of cornered neighborhoods of
the faces.  It is then standard that one can compose two smooth
cobordisms with corners, equipped with such smooth collars, by gluing
along $S$.  The composite is then the topological pushout of the two
inclusions.  Along the collars of $S$ in $M_1$ and $M_2$, charts
$\phi_i : U_i \ra [0,\infty)^n$ are equivalent to charts into
$\mathbbm{R}^{n-1} \times [0,\infty)$, and since the the composite has
a smooth structure defined up to a diffeomorphism\footnote{Note that
the precise smooth structure on this cobordism depends on the collar
which is chosen, but that there is always such a choice, and the
resulting composites are all equivalent up to diffeomorphism.  That
is, they are equivalent up to a 2-morphism in the bicategory.  So
strictly speaking, the composition map is not a functor but an
anafunctor.  It is common to disregard this issue, since one can
always define a functor from an anafunctor by using the axiom of
choice.  This is somewhat unsatisfactory, since it does not generalize
to the case where our categories are over a base in which the axiom of
choice does not hold, but this is not a problem in our example.  This
issue is discussed further by Makkai \cite{makkai}.} which is the identity
along $S$.

\subsection{Cobordisms with Corners}\label{sec:cobcat}

Suppose we take the category $\catname{Man}$ whose objects are smooth
manifolds with corners and whose morphisms are smooth maps.  Naively,
would would like to use the cospan construction from Section
\ref{sec:dblspan}, we obtain a {\vdb} $\opname{2Cosp}(\catname{Man})$.
While this approach will work with the category $\catname{Top}$,
however, it will not work with $\catname{Man}$ since this does not
have all colimits.  In particular, given two smooth manifolds with
boundary, we can glue them along their boundaries in non-smooth ways,
so to ensure that the pushout exists in $\catname{Man}$ we need to
specify a smoothness condition.  This requires using collars on the
boundaries and corners.

For each $n$, we define a {\vdb} within $\catname{Man}$, which we will
call $\nCob$:

\begin{definition}The {\vdb} $\nCob$ is given by the following data:
\begin{itemize}
\item The objects of $\nCob$ are of the form $P = \hat{P} \times I^2$
      where $\hat{P}$ may be any $(n-2)$ manifolds without boundary
      and $I=[0,1]$.
\item The horizontal and vertical bicategories of $\nCob$ have
  \begin{itemize}
  \item objects: as above
  \item morphisms: cospans $P_1 \ralim^{i_1} S
        \lalim^{i_2} P_2$ where $S = \hat{S} \times I$ and
        $\hat{S}$ may be any of those cospans of $(n-1)$-dimensional
        manifolds-with-boundary which are cobordisms with collars such
        that the $\hat{P_i} \times I$ are objects, the maps are
        injections into $S$, a manifold with boundary,
        such that $i_1(P_1) \cup i_2(P_2) = \partial S \times I$,
        $i_1(P_1) \cap i_2(P_2) = \emptyset$,
  \item 2-morphisms: cospan maps which are diffeomorphisms of the form
        $f \times \opname{id}: T\times[0,1] \ra T' \times
        [0,1]$ where $T$ and $T'$ have a common boundary, and $f$ is a
        diffeomorphism$T \ra T'$ compatible with the source
        and target maps---i.e. fixing the collar.
  \end{itemize} where the source of a cobordism $S$ consists of the
  collection of components of $\partial S \times I$ for which the
  image of $(x,0)$ lies on the boundary for $x \in \partial S$, and
  the target has the image of $(x,1)$ on the boundary
\item squares: diffeomorphism classes of $n$-dimensional manifolds $M$
      with corners satisfying the properties of $M$ in the diagram of
      equation (\ref{xy:cspan2}), where isomorphisms are
      diffeomorphisms preserving the boundary
\item the action of the diffeomorphisms on the ``squares'' (classes of
      manifolds $M$) is given by composition of diffeomorphisms of the
      boundary cobordisms with the injection maps of the boundary $M$
\end{itemize} The source and target objects of any cobordism are the
collars, embedded in the cobordism in such a way that the source
object $P = \hat{P} \times I^2$ is embedded in the cobordism $S =
\hat{S} \times I$ by a map which is the identity on $I$ taking the
first interval in the object to the interval for a horizontal
morphism, and the second to the interval for a vertical morphism.  The
same condition distinguishing source and target applies as above.

Composition of squares works as in $\opname{2Cosp}(\catname{C})_0$.
\end{definition}

We will see that $\nCob$ is a {\vdb} in Section \ref{sec:bicatcwc}, but for
now it suffices to note that since it is composed of double cospans,
we can hope to define composition to be just that in the \vdb
$\iiCCosp_0$ where $\C$ is the category of manifolds with corners.
The proof that this is a {\vdb} will entail showing that $\nCob$ is
closed under this composition.

\begin{lemma}\label{lemma:ncobclosedcomp}Composing horizontal
morphisms in $\nCob$ this way produces another horizontal morphism in
$\nCob$.  Similarly, composition of vertical morphisms produces a
vertical morphism, and composition of squares produces another square.
\end{lemma}
\begin{proof}
The horizontal and vertical morphisms are products of the interval $I$
with $\br{1}$-manifolds, whose boundary is $\partial_0 S$), equipped
with collars.  Suppose we are given two such cobordisms $S_1$ and
$S_2$, and an identification of the source of $S_2$ with the target of
$S_1$ (say this is $P = \hat{P} \times I$).  Then the composite $S_2
\circ S_1$ is topologically the pushout of $S_1$ and $S_2$ over $P$.
Now, $P$ is smoothly embedded in $S_1$ and $S_2$, and any point in the
pushout will be in the interior of either $S_1$ or $S_2$ since for any
point on $\hat{P}$ each end of the interval $I$ occurs as the boundary
of only one of the two cobordisms.  So the result is smooth.  Thus,
$\iiCob$ is closed under such composition of morphisms.

The same argument holds for squares, since it holds for any
representative of the equivalence class of some manifold with corners,
$M$, and the differentiable structure will be the same, since we
consider equivalence up to diffeomorphisms which preserve the collar
exactly.
\end{proof}

This establishes that composition in $\nCob$ is well defined, and
composites are again cobordisms in $\nCob$.  We show that it is a {\vdb}
in Section \ref{sec:bicatcwc}.

\begin{example}\label{ex:cobcospan} We can represent a typical
manifestation of the diagram (\ref{xy:cspan2}) as in Figure
\ref{fig:cobcorners-labelled}.

\begin{figure}[h]
\begin{center}
\includegraphics{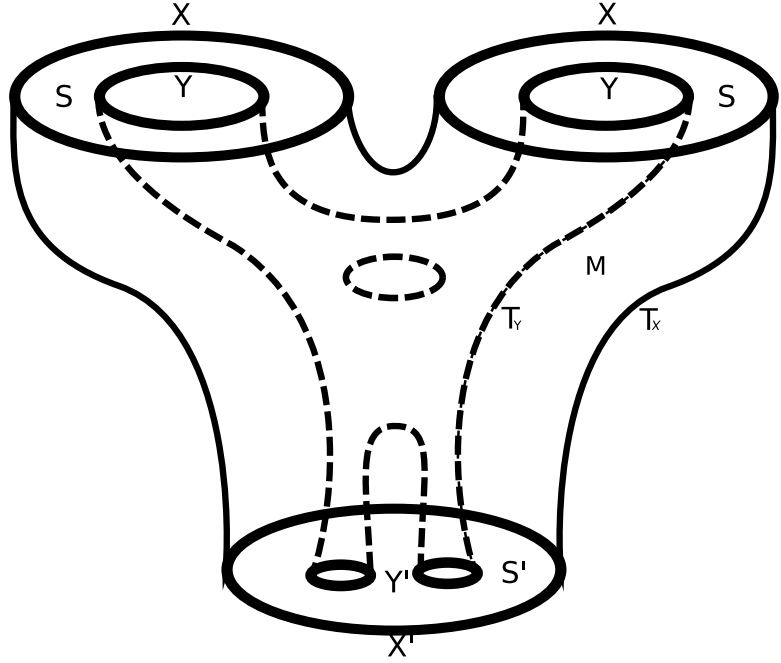}
\end{center}
\caption{\label{fig:cobcorners-labelled}A Square in $\nCob$ (Thickened Lines Denote Collars)}
\end{figure}

Consider how this picture is related to (\ref{xy:cspan2}).  In the
figure, we have $n=3$, so the objects are (compact, oriented)
$1$-dimensional manifolds, thickened by taking a product with $I^2$.
$X$ (top, solid lines) and $Y$ (top, dotted lines) are both isomorphic
to $(S^1 \cup S^1)\times I^2$, while $X'$ and $Y'$ (bottom, solid and dotted
respectively) are both isomorphic to $S^1 \times I^2$.

The horizontal morphisms are (thickened) cobordisms $S$, and $S'$,
which are a pair of thickened annuli and a two-holed disk,
respectively, with the evident injection maps from the objects $X, Y,
X', Y'$.  The vertical morphisms are the thickened cobordisms $T_X$
and $T_Y$.  In this example, $T_X$ happens to be of the same form as
$S'$ (a two-holed disk), and has inclusion maps from $X$ and $X'$, the
two components of its boundary, as the ``source'' and ``target'' maps.
$T_Y$ is homotopy equivalent to a four-punctured torus, where the four
punctures are the components of its boundary: two circles in $Y$ and
two in $Y'$, which again have the obvious inclusion maps.  Reading
from top to bottom, we can describe $T_Y$ as the story of two (thick)
circles which join into one circle, then split apart, then rejoin, and
finally split apart again.

Finally, the ``square'' in this picture is the manifold with corners,
$M$, whose boundary has four components, $S, S', T_X, \text{and} T_Y$,
and which has corners precisely along the boundaries of these
manifolds.  These boundaries' components are divided between the
objects $X, Y, X', Y'$.  The embeddings of these thickened manifolds
and cobordisms gives a specific way to equip $M$ with collars.

Given any of the horizontal or vertical morphisms (thickened
cobordisms $S$, $S'$, $T_X$ and $T_Y$), a 2-morphism would be a
diffeomorphism to some other cobordism equipped with maps from the
same boundary components (objects), which fixes the collar on that
cobordism (the embedded object).  Such a diffeomorphism is necessarily
a homeomorphism, so topologically the picture will be similar after
the action of such a 2-morphism, but we would consider two such
cobordisms as separate morphisms in $\Hor$ or $\Ver$.
\end{example}

\begin{remark} We note the resemblance between this example and
$\Path(S)_2$ and $\catname{H}(S)$ defined previously.  In those cases,
we are considering manifolds embedded in a topological space $S$, and
only a low-dimensional special case (the square $[0,1] \times [0,1]$
is a manifold with corners).  Instead of homotopies, which make sense
only for embedded spaces, $\nCob$ has diffeomorphisms.  However, in
both cases, we consider the squares to be \textit{isomorphism classes}
of a certain kind of top-dimensional object (homotopies or
cobordisms).  This eliminates the need to define morphisms or cells in
our category of dimension higher than 2.  We may omit this restriction
if we move to the more general setting of a {\db}, as described in
Section \ref{sec:internalbicat}.
\end{remark}

We conclude this section by illustrating composition in both
directions in $\nCob$, and in particular illustrating the interchange
law (\ref{eq:squareinterchangelaw}) for cobordisms with corners.
Figure \ref{fig:interchangelaw} shows four cobordisms with corners,
arranged to show three examples of horizontal composition and three of
vertical composition.  The vertical composites, denoted by
$\otimes_V$, can be seen as ``gluing'' the vertically stacked
cobordisms along the boundary between them, which is the bottom face
of the cobordisms on top, and the top face of those on bottom.  The
horizontal composites, denoted by $\otimes_H$, are somewhat less
obvious.  In the figure, they can be seen as ``gluing'' the right-hand
cobordism along a common face.  In each case, the common face is the
``inside'' face of the left-hand cobordism, and the ``outside'' face
of the right-hand one.

\begin{figure}[h]
\begin{center}
\includegraphics{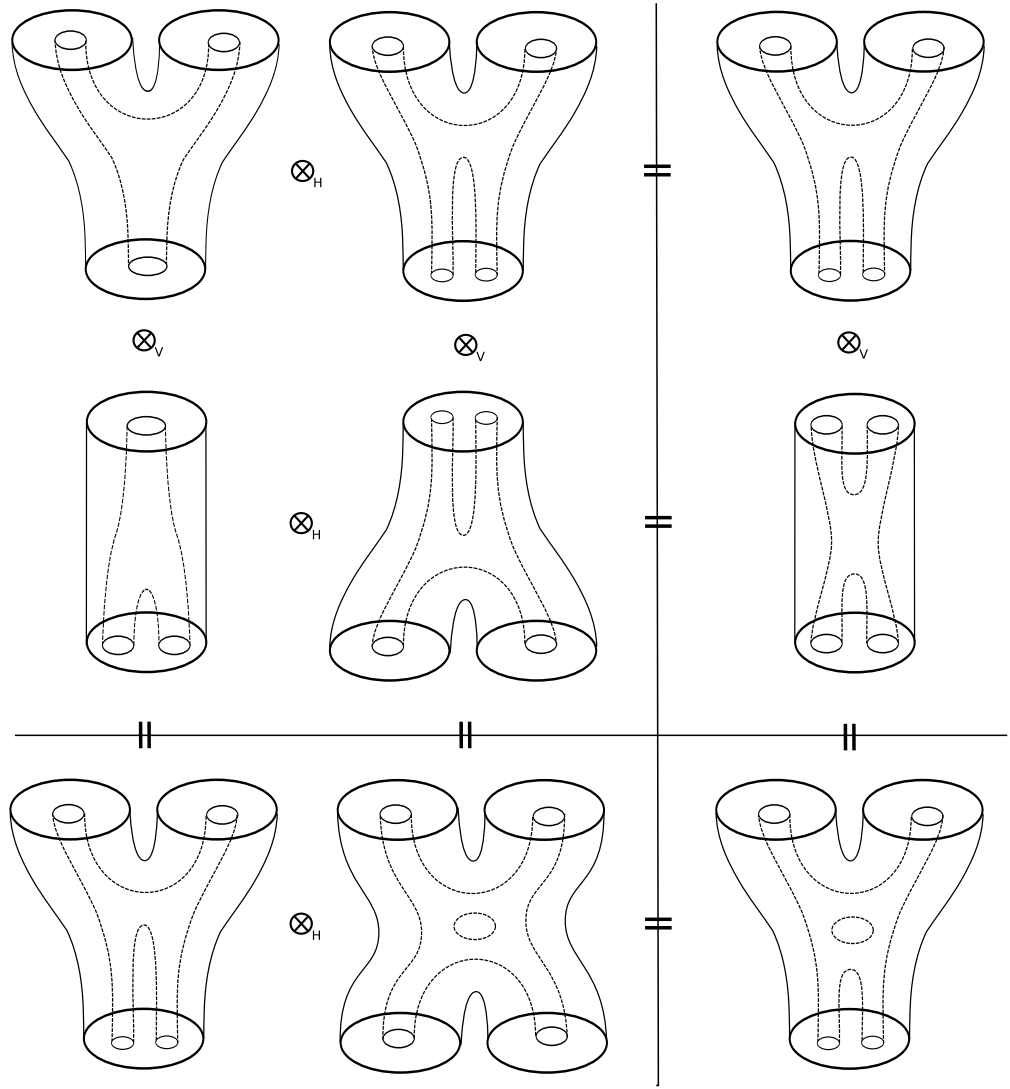}
\end{center}
\caption{\label{fig:interchangelaw}Compositions in $\nCob$ Satisfy the Interchange Law}
\end{figure}

\subsection{A Bicategory Of Cobordisms With Corners}\label{sec:bicatcwc}

Now we want to show that cobordisms of cobordisms form a {\vdb} under
the composition operations we have described.  We have shown in
Theorem \ref{thm:cspanthm} that there is a {\vdb} denoted $\iiCCosp_0$
for any category $\C$ with finite colimits.  We want to show that the
reduction from the full $\iiCCosp_0$ to just the particular cospans in
$\nCob$ leaves this fact intact.
 
The argument that double cospans form a {\vdb} can be slightly
modified to show the same about cobordisms with corners, which are
closely related.  We note that there are two differences.  First, the
category of manifolds with corners does not have all finite colimits,
or indeed all pushouts.  Second, we are not dealing with all double
cospans of manifolds with corners, so $\nCob$ is not $\iiCCosp_0$ for
any $\C$.  In fact, the second difference is what allows us to deal
with the first.

\begin{theorem}\label{thm:maintheorem} $\nCob$ is a {\vdb}. 
\end{theorem}
\begin{proof}
First, recall that objects in $\nCob$ are manifolds with corners of
the form $P = \hat{P} \times I^2$ for some manifold $\hat{P}$, and
notice that both horizontal and vertical morphisms are cospans.  In
general, if we have two cospans in the category of manifolds with
corners sharing a common object, we cannot take a pushout and get a
manifold with corners.  However, we are only considering a
subcollection of all possible cospans of smooth manifolds with corners,
all all those we consider have pushouts which are again smooth
manifolds with corners (Lemma \ref{lemma:ncobclosedcomp}).

In particular, since composition of squares is as in $\iiCCosp_0$,
before taking diffeomorphism classes of manifolds $M$ in $\nCob$, we
would again get a {\db} made from cobordisms with corners, together
with the embeddings used in its cospans, and collar-fixing
diffeomorphisms.  This is shown by arguments identical to those used
in the proof of Lemma \ref{lemma:span2bicat}.

When we reduce to diffeomorphism classes of these manifolds, then just
as in the proof of nTheorem \ref{thm:cspanthm}, we can cut down this
{\db} to a structure, and the result will satisfy the horizontal and
vertical action conditions, giving a \vdb, since it satisfies the
conditions of Lemma \ref{lemma:doublebicat}.

So in fact, by the same arguments as in these other cases, $\nCob$ is
a \vdb.
\end{proof}

By the argument of Section \ref{sec:equiv}, this means that we can
also think of $\nCob$ as a bicategory, which we will do for the
purposes of constructing an Extended TQFT as a weak 2-functor.  To do this,
we next describe, in Chapter \ref{chap:2hilb} some versions of a
bicategory of 2-vector spaces, and in particular 2-Hilbert spaces.

\section{2-Vector Spaces}\label{chap:2hilb}

\subsection{Kapranov-Voevodsky 2-Vector Spaces}\label{sec:KV}

We want to find a way of describing an extended TQFT---one acting on
manifolds with corners.  We would like to find something analogous to
Atiyah's characterization of a TQFT as a functor between a category of
cobordisms and a category of vector spaces.  We have now established
that there is a bicategory $\nCob$ of in which we can interpret
objects as manifolds, morphisms as cobordisms, and 2-morphisms as
cobordisms between cobordisms (which are diffeomorphism classes of
manifolds with corners).  The next constituent we need is a bicategory
to take the place of the category of vector spaces.  There are several
notions of a bicategory of ``2-vector spaces'' available, and each
gives rise to a notion of an extended TQFT as a 2-functor into this
bicategory.

There are two major philosophies of how to categorify the concept
``vector space''.  A Baez-Crans (BC) 2-vector space is a category
object in $\V$---that is, a category having a vector space of objects
and of morphisms, where source, target, composition, etc. are linear
maps.  This is a useful concept for some purposes---it was developed
to give a categorification of Lie algebras.  The reader may refer to
the paper of Baez and Crans \cite{BC} for more details.  However, a BC
2-vector space turns out to be equivalent to a 2-term chain complex
and, this is not the concept of 2-vector space which concerns us here.

The other, and prior, approach is to define a 2-vector space as a
category having operations such as a monoidal structure analogous to
the addition on a vector space.  In particular, We will restrict our
attention to \textit{complex} 2-vector spaces, though the
generalization to an arbitrary base field $K$ is straghtforward.

This ambiguity about the correct notion of ``2-vector space'' is
typical of the problem of categorificiation.  Since the categorified
setting has more layers of structure, there is a choice of level to
which the structure in the concept of a vector space should be lifted.
Thus in the BC 2-vector spaces, we have literal vector addition and
scalar multiplication within the objects and morphisms.  In KV
2-vector spaces and their cousins, we only have this for morphisms,
and for objects there is a categorified analog of these operations, as
wel shall see.

Indeed, there are different sensible generalizations of vector space
even within this second philosophy, however.  Josep Elgueta
\cite{elgueta} shows several different types of ``generalized''
2-vector spaces, and relationships among them.  In particular, while
KV 2-vector spaces can be thought of as having a \textit{set} of basis
elements, a generalized 2-vector space may have a general
\textit{category} of basis elements.  The free generalized 2-vector
space on a category is denoted $\catname{Vect[\mathcal{C}]}$.  Then KV
2-vector spaces arise when $\mathcal{C}$ is a discrete category with
only identity morphisms.  This is essentially a set $S$ of objects.
Thus it should not be surprising that KV 2-vector spaces have a
structure analogous to free vector spaces generated by some finite set
- which are isomorphic to $\mathbbm{C}^k$.

All such concepts of 2-vector space are $\mathbbm{C}$-linear additive
categories with some properties, so we begin by explaining this.  To
begin with, we have \textit{additivity} for categories, the equivalent
of linear structure in a vector space.  This is related to biproducts,
which are both categorical products and coproducts, in compatible
ways.  The motivating example for us is the \textit{direct sum}
operation in $\V$.  Such an operation plays the role in a 2-vector
space which vector addition plays in a vector space.  To be precise:

\begin{definition}
A \textbf{biproduct} for a category $\C$ is an operation giving, for
any objects $x$ and $y$ in $\catname{V}$ an object $x \oplus y$
equipped with morphisms $\iota_x,\iota_y$ from $x$ and $y$
respectively into $x \oplus y$; and morphisms $\pi_x, \pi_y$ from $x
\oplus y$ into $x$ and $y$ respectively, which satisfy the
biproduct relations:
\begin{equation}\label{eq:biproduct1}
\pi_x \circ \iota_x = \id_x \text{  and  } \pi_y \circ \iota_y = \id_y
\end{equation}
and similarly for $y$, and
\begin{equation}\label{eq:biproduct2}
\iota_x \circ \pi_x + \iota_y \circ \pi_y = \id_{x \oplus y}
\end{equation}
\end{definition}
Whenever biproducts exist, they are always both products and
coproducts.

\begin{definition} A \textbf{$\mathbbm{C}$-linear additive category}
is a category $\catname{V}$ with biproduct $\oplus$, and such that
that for any $x, y \in \catname{V}$, $\hom(x,y)$ is a vector space
over $\mathbbm{C}$, and composition is a bilinear map.  A
$\mathbbm{C}$-linear functor between $\mathbbm{C}$-linear categories
is one where morphism maps are $\mathbbm{C}$-linear.
\end{definition}

The standard example of this approach is the Kapranov-Voevodsky (KV)
definition of a 2-vector space \cite{KV}, which is the form we shall
use (at least when the situation is finite-dimensional).  To motivate
the KV definition, consider the idea that, in categorifying, one
should replace the base field $\mathbbm{C}$ with a monoidal category.
Specifically, it turns out, with $\V$, the category of finite dimensional complex vector
spaces.  This leads to the following replacements for concepts in elementary linear algebra:
\begin{itemize}
\item {Vectors = $k$-tuples of scalars} $\mapsto$ {2-vectors = $k$-tuples of vector spaces}
\item {Addition} $\mapsto$ {Direct Sum}
\item {Multiplication} $\mapsto$ {Tensor Product}
\end{itemize}

So just as $\mathbbm{C}^k$ is the standard example of a complex vector
space, $\V^k$ will be the standard example of a 2-vector space.
However, the axiomatic definition allows for other possibilities:

\begin{definition} A \textbf{Kapranov--Voevodsky 2-vector space} is a
$\mathbbm{C}$-linear additive category in which every object can be
written as a finite biproduct of simple objects (i.e. objects $x$
where $\hom(x,x) \cong \mathbbm{C}$).  A \textbf{2-linear map} between
2-vector spaces is a $\mathbbm{C}$-linear functor which preserves
biproducts.
\end{definition}

\begin{remark} It is a standard fact that preserving biproducts and
preserving exact sequences are equivalent in this setting: in a KV
2-vector space, every object is equivalent to a direct sum of simple
objects, so every exact sequence splits.  The above definition of a
2-linear map is sometimes given in the equivalent form requiring that
the functor preserve exact sequences.  Indeed, since every object is a
finite biproduct of simple objects, a 2-vector space is an
\textit{abelian} category.  For more on these, see Freyd \cite{freyd}.
\end{remark}

Now, it is worth mentioning that Yetter shows \cite{yet} (in his
Proposition 13), that the original definition of Kapranov and
Voevodsky gives an equivalent result to a definition of a 2-vector
space $\catname{V}$ as a finitely semi-simple $\V$-module.  A
$\V$-module $\catname{V}$ is finitely semi-simple if there is a finite
set $S \subset Ob(\catname{V})$ of simple objects, such that every
objects of $\catname{V}$ is a finite product of objects in $S$.  The
advantage of this definition is simply that it is a straightforward
categorification of the usual definition of a vector space as a
$\mathbbm{C}$-module.

The reader is referred to Yetter's paper (Definition 2) for a precise
version of the definition of a $\V$-module, but remark that to be a
$\V$-module requires that $\catname{V}$ has an ``action'' of $\V$ on
it.  That is, there is a functor
\begin{equation}
\odot : \V \times \catname{V} \ra \catname{V}
\end{equation}
which satisfies the usual module axioms only up to two isomorphisms,
similar to the associator and unitor, which satisfy some further
coherence conditions.  We will see the meaning this action when we
consider a standard example, where this is literally a tensor product.

\begin{example}The standard example \cite{KV} of a KV 2-vector space
highlights the analogy with the familar vector space $\mathbbm{C}^k$.
The 2-vector space $\V^k$ is a category whose objects are $k$-tuples
of vector spaces, maps are $k$-tuples of linear maps.  The
\textit{additive} structure of the 2-vector space $\V^k$ comes from
applying the direct sum in $\V$ componentwise.

Note that there is an equivalent of \textit{scalar multiplication},
using the tensor product:
\begin{equation}
V \otimes 
\begin{pmatrix}
  V_1 \\ 
  \vdots \\
  V_k \\
\end{pmatrix}
=
\begin{pmatrix}
  V \otimes V_1 \\
  \vdots \\
  V \otimes V_k
\end{pmatrix}
\end{equation}
and
\begin{equation}
\begin{pmatrix}
  V_1 \\ 
  \vdots \\
  V_k
\end{pmatrix}
 \oplus 
\begin{pmatrix}
  W_1 \\
  \vdots \\
  W_k
\end{pmatrix}
= 
\begin{pmatrix}
  V_1 \oplus W_1 \\
  \vdots \\
  V_k \oplus W_k
\end{pmatrix}
\end{equation}

As the correspondence with linear algebra would suggest, 2-linear maps
$T: \catname{Vect^k} \ra\catname{Vect^l}$ amount to $k \times l$
matrices of vector spaces, acting by matrix multiplication using the
direct sum and tensor product instead of operations in $\mathbbm{C}$:
\begin{equation}\label{eq:kv2linmatrix}
\begin{pmatrix}
T_{1,1} & \dots & T_{1,k} \\
\vdots & & \vdots \\
T_{l,1} & \dots & T_{l,k} \\
\end{pmatrix}
\begin{pmatrix}
  V_1 \\ 
  \vdots \\
  V_k
\end{pmatrix}
=
\begin{pmatrix}
  \bigoplus_{i=1}^k T_{1,i} \otimes V_i \\ 
  \vdots \\
  \bigoplus_{i=1}^k T_{l,i} \otimes V_i \\ 
\end{pmatrix}
\end{equation}

The natural transformations between these are matrices of linear
transformations:
\begin{equation}\label{eq:kvnattransmatrix}
\alpha = \begin{pmatrix}
\alpha_{1,1} & \dots & \alpha_{1,k} \\
\vdots & & \vdots \\
\alpha_{l,1} & \dots & \alpha_{l,k} \\
\end{pmatrix}
:
\begin{pmatrix}
T_{1,1} & \dots & T_{1,k} \\
\vdots & & \vdots \\
T_{l,1} & \dots & T_{l,k} \\
\end{pmatrix}
\longrightarrow
\begin{pmatrix}
T'_{1,1} & \dots & T'_{1,k} \\
\vdots & & \vdots \\
T'_{l,1} & \dots & T'_{l,k} \\
\end{pmatrix}
\end{equation}
where each $\alpha_{i,j} : T_{i,j} \ra T'_{i,j}$ is a linear
map in the usual sense.

These natural transformations give 2-morphisms between 2-linear maps, so
that $\V^k$ is a bicategory with these as 2-cells:
  \begin{equation}\label{xy:2vs2cell}
    \xymatrix{
      \catname{Vect^k} \ar@/^1pc/[r]^{F}="0"
      \ar@/_1pc/[r]_{G}="1" & \catname{Vect^l} \\ \ar@{=>}"0"
      ;"1"^{\alpha}
    }
  \end{equation}

In our example above, the finite set of simple objects of which every
object is a sum is the set of 2-vectors of the form
\begin{equation}
\begin{pmatrix}
0 \\
\vdots \\
\mathbbm{C} \\
\vdots \\
0
\end{pmatrix}
\end{equation}
which have the zero vector space in all components except one (which
can be arbitrary).  These are like categorified ``standard basis
vectors'', so we call them \textit{standard basis 2-vectors} . Clearly
every object of $\V^k$ is a finite biproduct of these objects, and
each is simple (its vector space of endomorphisms is 1-dimensional).
\end{example}

The most immediately useful fact about KV 2-vector spaces is the
following well known characterization:

\begin{theorem}\label{thm:kvvn} Every KV 2-vector space is equivalent as a category
to $\V^k$ for some $k \in \mathbbm{N}$.
\end{theorem}
\begin{proof}
Suppose $\catname{K}$ is a KV 2-vector space with a basis of simple
objects $X_1 \dots X_k$.  Then we construct an equivalence $E:
\catname{K} \ra \V^k$ as follows:

$E$ should be an additive functor with $E(X_i) = V_i$, where $V_i$ is
the $k$-tuple of vector spaces having the zero vector space in every
position except the $i^{th}$, which has a copy of $\mathbbm{C}$.  But
any object $X$, is a sum $\bigoplus_{i} X_i^{n_i}$, so by linearity
(i.e. the fact that $E$ preserves biproducts) $X$ will be sent to the
sum of the same number of copies of the $V_i$, which is just a
$k$-tuple of vector spaces whose $i^{th}$ component is
$\mathbbm{C}^{n_i}$.  So every object in $K$ is sent to an $k$-tuple
of vector spaces.  By $\mathbbm{C}$-linearity, and the fact that
hom-vector spaces of simple objects are one-dimensional, this
determines the images of all morphisms.

But then the weak inverse of $E$ is easy to construct, since sending
$V_i$ to $X_i$ gives an inverse at the level of objects, by the same
linearity argument as above. At the level of morphisms, the same
argument holds again.
\end{proof}

This is a higher analog of the fact that every finite dimensional
complex vector space is isomorphic to $\mathbbm{C}^k$ for some $k \in
\mathbbm{N}$.  So, indeed, the characterization of 2-vector spaces in
our example above is generic: every KV 2-vector space is equivalent to
one of the form given.  Moreover, our picture of 2-linear maps is also
generic, as shown by this argument, analogous to the linear algebra
argument for representation of linear maps by matrices:

\begin{lemma}\label{lemma:kv2linmatrix} Any 2-linear map $T: \V^n \ra
\V^m$ is naturally isomorphic to a map of the form (\ref{eq:kv2linmatrix}).
\end{lemma}
\begin{proof}
Any 2-linear map $T$ is a $\mathbbm{C}$-linear additive functor
between 2-vector spaces.  Since any object in a 2-vector space can be
represented as a biproduct of simple objects---and morphisms
likewise---such a functor is completely determined by its effect on
the basis of simple objects and morphisms between them.

But then note that since the automorphism group of a simple object is
by definition just all (complex) multiples of the identity morphism,
there is no choice about where to send any such morphism.  So a
functor is complely determined by the images of the basis objects,
namely the 2-vectors $V_i = (0,\dots,\mathbbm{C},\dots,0) \in \V^n$,
where $V_i$ has only the $i^{th}$ entry non-zero.

On the other hand, for any $i$, $T(V_i)$ is a direct sum of some
simple objects in $\V^m$, which is just some 2-vector, namely a
$k$-tuple of vector spaces.  Then the fact that the functor is
additive means that it has exactly the form given.
\end{proof}

And finally, the analogous fact holds for natural transformations
between 2-linear maps:

\begin{lemma}\label{lemma:kvnattransmatrix}Any natural transformation
$\alpha : T \ra T'$ from a 2-linear map $T : \V^n \ra \V^m$ to a
2-linear map $T' : \V^n \ra \V^m$, both in the form
(\ref{eq:kv2linmatrix}) is of the form (\ref{eq:kvnattransmatrix}).
\end{lemma}
\begin{proof}
By Lemma \ref{lemma:kv2linmatrix}, the 2-linear maps $T$ and $T'$ can
be represented as matrices of vector spaces, which act on an object
in $\V^n$ as in (\ref{eq:kv2linmatrix}).  A natural transformation
$\alpha$ between these should assign, to every object $X \in \V^n$, a
morphism $\alpha_X : T(X) \ra T'(X) $ in $\V^m$, such that the usual
naturality square commutes for every morphism $f : X \ra Y$ in $\V^n$.

Suppose $X$ is the $n$-tuple $(X_1, \dots , X_n)$,
where the $X_i$ are finite dimensional vector spaces.  Then
\begin{equation}\label{eq:TX}
T(X) = (\oplus_{k=1}^{n} V_{1,k} \otimes X_k, \dots,\oplus_{k=1}^{n} V_{m,k} \otimes X_k )
\end{equation}
where the $V_{i,j}$ are the components of $T$, and similarly
\begin{equation}
T'(X) = (\oplus_{k=1}^{n} V'_{1,k} \otimes X_k, \dots,\oplus_{k=1}^{n} V'_{m,k} \otimes X_k )\end{equation}
where the $V'_{i,j}$ are the components of $T'$.

Then a morphism $\alpha_X : T(X) \ra T'(X)$ consists of an $m$-tuple
of linear maps:
\begin{equation}
\alpha_j : \oplus_{k=1}^{n} V_{j,k} \otimes X_k \ra \oplus_{k=1}^{n} V'_{j,k} \otimes X_k
\end{equation}
but by the universal property of the biproduct, this is the same as
having an $(n \times m)$-indexed set of maps
\begin{equation}
\alpha_{jk} : V_{j,k} \otimes X_k \ra \oplus_{r=1}^{n} V'_{j,r} \otimes X_r
\end{equation}
and by the dual universal property, this is the same as having $(n
\times n \times m)$-indexed maps
\begin{equation}
\alpha_{jkr} : V_{j,k} \otimes X_k \ra V'_{j,r} \otimes X_r
\end{equation}
However, we must have the naturality condition for every morphism
$f:X \ra X'$:
\begin{equation}
\xymatrix{
  T(X) \ar[d]_{\alpha_{X}} \ar[r]^{T(f)} & T(X') \ar[d]^{\alpha_{X'}} \\
  T'(X) \ar[r]_{T'(f)} & T'(X')\\ 
}
\end{equation}
Note that each of the arrows in this diagram is a morphism in $\V^m$,
which are linear maps in each component---so in fact we have a
separate naturality square for each component.

Also, since $T$ and $T'$ act on $X$ and $X'$ by tensoring
with fixed vector spaces as in (\ref{eq:TX}), one has $T(f)_i =
\oplus_i f_i \otimes 1_{V_{ij}}$, having no effect on the $V_{ij}$.
We want to show that the components of $\alpha$ affect \textit{only}
the $V_{ij}$.

Additivity of all the functors involved implies that the assignment
$\alpha$ of maps to objects in $\V^n$ is additive.  So consider the
case when $X$ is one of the standard basis 2-vectors, having
$\mathbbm{C}$ in one position (say, the $k^{th}$), and the zero vector
space in every other position.  Then, restricting to the naturality
square in the $k^{th}$ position, the above condition amounts to having
$m$ maps (indexed by $j$):
\begin{equation}
\alpha_{j,k} : V_{j,k} \ra V'_{j,k}
\end{equation}
So by linearity, a natural transformation is determined by an $n\times
m$ matrix of maps as in (\ref{eq:kvnattransmatrix}).
\end{proof}

The fact that 2-linear maps between 2-vector spaces are functors
between categories recalls the analogy between linear algebra and
category theory in the concept of an \textit{adjoint}.  If $V$ and $W$
are inner product spaces, the adjoint of a linear map $F : V \ra V$ is
a map $F^{\dagger}$ for which $\inprod{Fx,y} =
\inprod{x,F^{\dagger}y}$ for all $x \in V_1$ and $y \in V_2$.  A
(right) adjoint of a functor $F : \catname{C} \ra \catname{D}$ is a
functor $G : \catname{D} \ra \catname{C}$ for which $\hom_D(Fx,y)
\cong \hom_C(x,Gy)$ (and then $F$ is a left adjoint of $G$).

In the situation of a KV 2-vector space, the categorified analog of
the adjoint of a linear map is indeed an adjoint functor.  (Note that
since a KV 2-vector space has a specified basis of simple objects, it
makes sense to compare it to an inner product space.)  Moreover, the
adjoint of a functor has a matrix representation which is much like
the matrix representation of the adjoint of a linear map.  We
summarize this as follows:

\begin{theorem}\label{thm:biadjoint} Given any 2-linear map $F : V \ra
W$, there is a 2-linear map $G : W \ra V$ which is both a left and
right adjoint to $F$, and $G$ is unique up to natural isomorphism.
\end{theorem}
\begin{proof}
By Theorem \ref{thm:kvvn}, we have $V \simeq \V^n$ and $W \simeq \V^m$
for some $n$ and $m$.  By composition with these equivalences, we can
restrict to this case.  But then we have by Lemma
\ref{lemma:kv2linmatrix} that $F$ is naturally isomorphic to some
2-linear map given by matrix multiplication by some matrix of vector
spaces $[F_{i,j}]$:
\begin{equation}
\begin{pmatrix}
F_{1,1} & \dots & F_{1,n} \\
\vdots & & \vdots \\
F_{m,1} & \dots & F_{m,n} \\
\end{pmatrix}
\end{equation}

We claim that a (two-sided) adjoint functor $F^{\dagger}$ is given by the ``dual
transpose matrix'' of vector spaces $[F_{i,j}]^{\dagger}$:
\begin{equation}
\begin{pmatrix}
F^{\dagger}_{1,1} & \dots & F^{\dagger}_{1,m}\\
\vdots & & \vdots \\
F^{\dagger}_{n,1} & \dots & F^{\dagger}_{n,m} \\
\end{pmatrix}
\end{equation}
where $F^{\dagger}_{i,j}$ is the vector space dual $(F_{j,i})^{\ast}$
(note the transposition of the matrix).

We note that this prescription is symmetric, since $[T]^{\dagger
\dagger} = [T]$, so if $G$ is always left adjoint of $F$, then $F$ is
also a left-adjoint of $G$, hence $G$ a right adjoint of $F$.  So if
this prescription gives a left adjoint, it gives a two-sided adjoint.
Next we check that it does.

Suppose $x=(X_i) \in \V^n$ is the 2-vector with vector space $X_i$ in
the $i^{th}$ component, and $y=(Y_j) \in \V^m$ has vector space $Y_j$
in the $j^{th}$ component.  Then $Fx \in \V^m$ has $j^{th}$ component
$\oplus_{k=1}^n V_{k,j} \otimes X_k$.  Now, a map in $\V^m$ from $Fx$
to $y$ consists of a linear map in each component, so it is an
$m$-tuple of maps:
\begin{equation}
f_j: \oplus_{k=1}^n V_{k,j} \otimes X_k \ra Y_j
\end{equation}
for $j = 1 \dots m$.  But since the direct sum (biproduct) is a
categorical product, this is the same as an $m \times n$ matrix of
maps:
\begin{equation}
f_{kj}: V_{k,j} \otimes X_k \ra Y_j
\end{equation}
for $k = 1 \dots n$ and $j = 1 \dots m$, and $\hom(Fx,y)$ is the
vector space of all such maps.

By the same argument, a map in $\V^n$ from $x$ to $Gy$ consists of an
$n \times m$ matrix of maps:
\begin{equation}
g_{jk}: :X_k \ra V_{j,k}^{\ast} \otimes Y_j \cong \hom(V_{j,k},Y_j)
\end{equation}
for $k = 1 \dots n$ and $j = 1 \dots m$, and $\hom(x,Gy)$ is the
vector space of all such maps.

But then we have a natural isomorphism $\hom(Fx,y) \cong \hom(x,Gy)$
by the duality of $\hom$ and $\otimes$, so in fact $G$ is a right
adjoint for $F$, and by the above argument, also a left adjoint.

Moreover, no other nonisomorphic matrix defines a 2-linear map with
these properties, and since any functor is naturally isomorphic to
some matrix, this is the sole $G$ which works.
\end{proof}

We conclude this section by giving an example of a 2-vector space
which we shall return to again later.  It is motivated by the attempt
to generalize the FHK construction of a TQFT from a group, as
described in Section \ref{sec:fhk}.  During the construction of the
vector space assigned to a circle, one makes use of the group algebra
of a finite group $G$---the set of complex linear combinations of
group elements.  There is a categorified analog:

\begin{example}\label{ex:grp2alg} As an example of a KV 2-vector space, consider the
\textit{group 2-algebra} on a finite group $G$, defined by analogy
with the group algebra:

The group algebra $\CG$ consists of the set of elements formed as
formal linear combinations elements of $G$:
\begin{equation}
b = \sum_{g \in G} b_g \cdot g
\end{equation}
where all but finitely many $b_g$ are zero.  We can think of these as
complex functions on $G$.  The algebra multiplication on $\CG$ is
given by the multiplication in $G$:
\begin{equation}
b \star b' = \sum_{g,g' \in G} (b_g b'_{g'}) \cdot g g' \\
\end{equation}
This does not correspond to the multiplication of functions on $G$,
but to \textit{convolution}:
\begin{equation}
(b \star b')_g = \sum_{h\cdot h' = g}b_h b'_{h'}
\end{equation}

Similarly, the \textit{group 2-algebra} $A=\VG$ is the
\textit{category} of $G$-graded vector spaces.  That is, direct
sums of vector spaces associated to elements of $G$:
\begin{equation}
V = \bigoplus_{g \in G} V_g
\end{equation}
where $V_g \in \V$ is a vector space.  This is a $G$-graded vector
space.  We can take direct sums of these pointwise, so that $(V
\oplus V')_g = V_g \oplus V'_g)$, and there is a ``scalar'' product
with elements of $\V$ given by $(W \otimes V)_g = W \otimes V_g$.
There is also a \textit{group 2-algebra} product of $G$-graded vector
spaces, involving a convolution on $G$:
\begin{equation}
(V \star V')_h = \bigoplus_{g \cdot g' = h}V_g \otimes V'_{g'}
\end{equation}

The category of $G$-graded vector spaces is clearly a KV 2-vector
space, since it is equivalent to $\V^k$ where $k = |G|$.  However, it
has the additional structure of a 2-algebra because of the group
operation on the finite set $G$.
\end{example}

The analogy between group algebras and group 2-algebras highlights one
motivation for thinking of 2-vector spaces.  This is the fact that, in
quantum mechanics, one often ``quantizes'' a classical system by
taking the Hilbert space of $\mathbbm{C}$-valued functions on its
phase space.  Similarly, one approach to finding a higher-categorical
version of a quantum field theory is to take $\V$-valued functions.
We have noted in Section \ref{sec:fhk} that, given a finite group, the
Fukuma-Hosono-Kawai construction gives 2D TQFT, whose Hilbert space of
states on a circle is just $\mathbbm{C}[G]$.  For this reason, we
expect that Example \ref{ex:grp2alg} should be relevant to
categorifying this theory.  However, it is not quite sufficient, as we
discuss in Section \ref{sec:kv2vsgrpd}.

\subsection{KV 2-Vector Spaces and Finite Groupoids}\label{sec:kv2vsgrpd}

The group 2-algebra of Example \ref{ex:grp2alg} shows that we can get
a 2-vector space as a category of functions from some finite set $S$
into $\V$, and this may have extra structure if $S$ does.  However,
this is somewhat unnatural, since $\V$ is a category and $S$ a mere
set.  It seems more natural to consider functor categories into $\V$
from some category $\C$.  These are the generalized 2-vector spaces
described by Elgueta \cite{elgueta}.  Then the above way of looking at
a KV 2-vector space can be reduced to the situation when $\C$ is a
discrete category with a finite set of elements.

However, there are interesting cases where $\C$ is not of this form,
and the result is still a KV vector space.  A relevant class of
examples, as we shall show, come from special kinds of groupoids.

\begin{definition}\label{def:essfingpd} An \textbf{essentially finite}
groupoid is one which is equivalent to a finite groupoid.  A
\textbf{finitely generated} groupoid is one with a finite set of
objects, and all of whose morphisms are generated under composition by
a finite set of morphisms.  An \textbf{essentially finitely generated}
groupoid is one which is equivalent to a finitely generated one.
\end{definition}

We first show that finite groupoids are among the special
categories $\C$ we want to consider:

\begin{lemma}\label{lemma:fgfckv} If $\X$ is an essentially finite
groupoid, the functor category $[\X,\V]$ is a KV 2-vector space.
\end{lemma}
\begin{proof}
To begin with, we note that $\V$ is trivially a KV 2-vector space.  In
particular, it is a $\mathbbm{C}$-linear additive category, which we
use to give $[\X,\V]$ the same structure.

Define a biproduct $\oplus$ on $[\X,\V]$ as follows. Given two
functors $F_1,F_2 : \X \ra \V$, define for both objects and morphisms,
\begin{equation}
(F_1 \oplus F_2) (x) = F_1(x) \oplus F_2(x)
\end{equation}
where we are using both the direct sum of vector spaces, and the fact
that linear maps between vector spaces inherit a direct sum.  The
projections and injections are defined pointwise.  Since the biproduct
axioms (\ref{eq:biproduct1}) and (\ref{eq:biproduct2}) hold pointwise,
this is indeed a biproduct.

Now $\X$ is equivalent to a skeleton of itself, $\underline{\X}$,
which contains a single object in each isomorphism class.  Since $\X$
is essentially finite, this is also a finite set of objects, and each
object has a finite set of endomorphisms.  Since these are all
invertible, $\X$ is in fact equivalent to a finite coproduct of finite
groups, thought of as single-object categories.

But then a functor $F: \underline{\X} \ra \V$ is just a direct sum of
functors from these groups.  A functor from a group $G$ (as a
one-object category) to $\V$ is just a finite dimensional
representation of $G$.  Now, Schur's Lemma states that the only
intertwining operators from an irreducible representation to itself
are multiples of the identity.  That is, it ensures that all such
representations are simple objects.  On the other hand, every
representation is a finite direct sum of irreducible ones.

So in particular, the finite dimensional representations of a finite
group form a KV 2-vector space.  A direct sum of such categories is
again a KV 2-vector space, and so $[\underline{\X},\V]$ is one.

But $[\X,\V]$ is equivalent to this, so it is a KV 2-vector space.
\end{proof}

We notice that we are speaking here of groupoids, and any groupoid
$\X$ is equivalent to its opposite category $\X^{op}$, by an
equivalence that leaves objects intact and replaces each morphism by
its inverse.  So there is no real difference between $[\X,\V]$, the
category of $\V$-valued functors from $\X$, and $[\X^{op},\V]$, the
category of \textit{$\V$-valued presheaves} (henceforth simply
``$\V$-presheaves'') on $\X$, where we emphasize that unlike ordinary
presheaves, these are functors into $\V$, rather than $\Set$.  So we
have shown that $\V$-presheaves on a groupoid $\X$ form a KV 2-vector
space.  We will work with these examples from now on.

Since many results about presheaves are well known, we will find it
convenient to use this terminology for objects of $[\X,\V]$ for the
sake of compatibility, and to highlight the connection to these
results.  We will ignore the distinction in the sequel, but remind
readers here that our uses of the term ``presheaf'' are valid only
because we are working with groupoids.  The importance of
$\Set$-valued presheaves to topos theory, and the richness of existing
results for these, is one reason to keep this relationship in mind.

Now we want to show a result analogous to a standard result for
presheaves (see, e.g. MacLane and Moerdijk \cite{macmoer}, Theorem
1.9.2).  This is that functors between underlying groupoids induce
2-linear maps between the 2-vector spaces of $\V$-presheaves on them.

\begin{theorem}\label{thm:2mapadjoints}If $\X$ and $\Y$ are
essentially finite groupoids, a functor $f:\X \ra \Y$ gives two
2-linear maps between KV 2-vector spaces:
\begin{equation}
f^{\ast} : [ \Y, \V ] \ra [\X,\V]
\end{equation}
called ``pullback along $f$'' and
\begin{equation}
f_{\ast} : [\X,\V] \ra [\Y,\V]
\end{equation}
called ``pushforward along $f$''.  Furthermore, $f_{\ast}$ is the
(two-sided) adjoint to $f^{\ast}$.
\end{theorem}
\begin{proof}
First we define, for any functor $F: \Y \ra \V$,
\begin{equation}
f^{\ast}(F) = F \circ f
\end{equation}
which is a functor from $\X$ to $\V$.  This is just the pullback of
$F$ along $f$.

To show that this is a 2-linear map (that is, a biproduct-preserving,
$\mathbbm{C}$-linear functor), we first note that it is trivially
$\mathbbm{C}$-linear since a linear combination of maps in some
$\hom$-category in $[\Y,\V]$ is taken by $f^{\ast}$ to the
corresponding linear combination in the $\hom$-category in $[\X,\V]$,
where maps are now between vector spaces thought of over $x \in \X$.

To check that the functor $f^{\ast} : [\Y,\V] \ra [\X,\V]$ preserves
biproducts, note that for any $x \in \X$ we have that $f^{\ast}(F_1
\oplus F_2)(x) = (F_1 \oplus F_2)(f(x)) = F_1(f(x)) \oplus F_2(f(x))
= (f^{\ast}F_1 \oplus f^{\ast}F_2)(x)$.

So indeed there is a 2-linear map $f^{\ast}$.  But then by Theorem
\ref{thm:biadjoint}, there is a two-sided adjoint of $f^{\ast}$,
denoted $f_{\ast}$.
\end{proof}

\begin{remark} The argument in this proof for the existence of the
adjoint to $f^{\ast}$ uses Theorem \ref{thm:biadjoint}.  While no such
theorem exists for $\Set$-valued presheaves, there is a corresponding
theorem defining a ``pushforward'' of presheaves of sets.  In fact,
the only major difference between what we have shown for
$\V$-presheaves and the standard results for $\Set$-presheaves is that
the left and right adjoint are the same.  This means that the
``pushforward'' map is an \textit{ambidextrous adjunction} for the
pullback (for much more on the relation between ambidextrous
adjunctions and TQFTs, see Lauda \cite{laudaambidjunction}).

It seems useful, then, to have another approach to the ``pushforward''
map than the matrix-dependent view of Theorem \ref{thm:biadjoint}.
Fortunately, there is a more instrinsic way to describe the 2-linear
map $f_{\ast}$, the adjoint of $f^{\ast}$, and we know this must be
the same as the one given in matrix form.
\end{remark}

\begin{definition}\label{def:pushcolimit}For a given $y \in \Y$,
define the diagram $D_y$ whose objects are objects $x \in \X$ equipped
with maps $f(x) \ra y$ in $\Y$, and whose morphisms are morphisms $a :
x \ra x'$ whose images make the triangles
\begin{equation}
\xymatrix{
f(x) \ar[d] \ar[r]^{f(a)} & f(x') \ar[dl] \\
y
}
\end{equation}
in $\Y$ commute.  Given a $\V$-presheaf $G$ on $\X$, define
$f_{\ast}(G)(y) = \opname{colim}G(D_y)$---a colimit in $\V$.

The pushforward of a morphism $b:y \ra y'$ in $\Y$, $f_{\ast}(G)(b) :
f_{\ast}(G)(y) \ra f_{\ast}(G)(y')$ is left to the reader.
\end{definition}

This definition of the pushforward involved the diagram $D$, which is
the \textit{comma category} of objects $x \in \X$ equipped with maps
from $f(x)$ to $y$.  This is the appropriate categorical equivalent of
a \textit{preimage}---rather than requiring $f(x) = y$, one accepts
that they may be isomorphic, in different ways.  So this is a
categorified equivalent of taking a sum over a preimage.  It needs to
be confirmed directly that it really is the adjoint.

\begin{theorem}\label{thm:pushcolimit} This $f_{\ast}$ is a 2-linear
map, and a two-sided adjoint for $f^{\ast}$.
\end{theorem}
\begin{proof}
The given $f_{\ast}$ certainly defines a $\V$-presheaf $f_{\ast}G$ on
$\Y$, and the operation of taking colimits is functorial and preserves
biproducts, so $f_{\ast}$ is a 2-linear map.

Consider the effect of $f_{\ast}$ on a 2-vector $G:\X \ra \V$ by
describing $f_{\ast}G : \Y \ra \V$.  If $F : \Y \ra \V$ is as above,
there should be a canonical isomorphism between $[G,f^{\ast}(F)]$ (a
hom-set in $[\X,\V]$) and $[p_{\ast}(G),F]$ (a hom-set in $[\Y,\V]$).

The hom-set $[G,f^{\ast}(F)]$ is found by first taking the pullback of
$F$ along $f$.  This gives a presheaf on $\X$, namely $F(f(-))$.  The
hom-set is then the set of natural transformations $\alpha : G \ra
f^{\ast}F$.  Each such $\alpha$, given an object $x$ in $\X$, picks a
linear map $\alpha_x : F(f(x)) \ra G(x)$ (subject to the naturality
condition).

For an object $y$ in $\Y$, pulling back $F$ onto $\X$ gives the vector
space $F(y)$ at each object $x$ with $f(x)=y$.  This is the presheaf
$f^{\ast}F$.  So an element of $[f^{\ast}F,G]$ is an assignment, to
every $x \in \X$, a linear map $f^{\ast}F = F(y) \ra G(x)$.

To get the equivalence required for adjointness, given a linear map $h
: f_{\ast}G(y) \ra F(y)$, one should get a collection of maps $h_x :
G(x) \ra F(y)$ for each object $x$ in $D$ (which commute with all
arrows in $D$).  But $f_{\ast}(G)(x)$ was defined to be a colimit,
hence there is a unique compatible map $i_x$ from each $G(x)$ into it,
so take $h_x = h \circ i_x : G(x) \ra F(y)$.  This gives a map from
$[p_{\ast}(G),F]$ to $[G,f^{\ast}(F)]$.  To see that this is an
equivalence, note that the colimit is a \textit{universal} object with
the specified maps.  So given the collection of $h_x$, one gets the map
$h$ from the universal property.

So $f_{\ast}$ is a left adjoint to $f^{\ast}$.  By Theorem
\ref{thm:biadjoint}, it is therefore also a right-adjoint.
\end{proof}

\begin{remark}\label{rk:pullpushunit} For future reference, we will
describe the pair of adjoint functors, $f^{\ast}$ and $f_{\ast}$ in
even more detail, since this is used in the construction of our
extended TQFT in Chapter \ref{chap:connexttqft}.  Since we will want
to make use of the simplifying fact that any groupoid is equivalent to
a skeletal groupoid, it is particularly helpful to consider this case.

A skeletal groupoid has exactly one object in each isomorphism class,
so it is equivalent to a disjoint union of one-object groupoids -
which can be interpreted as groups.  Since $\X$ and $\Y$ are
essentially finite, these are finite groups.  So a $\V$-presheaf on
$\X$ is a functors which assigns a vector space $V_x$ to each object
$x \in \X$, and a linear map $V \ra V$ for each morphism (i.e. group
element).  This is just a representation of the finite group
$\opname{Aut}(x)$ on $\V_x$.

If $\X$ and $\Y$ are skeletal, then $f : \X \ra \Y$ on objects
is just a set map, taking objects in $\X$ to objects in $\Y$.  For
morphisms, $f$ gives, for each object $x \in \X$, a homomorphism from
the group $\hom(x,x)$ to the group $\hom(f(x),f(x))$.

So the pullback $f^{\ast}$ is fairly straightforward: given $F : \Y
\ra \V$, the pullback $f^{\ast}F = F \circ f : \X \ra \V$ assigns to
each $x \in \X$ the vector space $F(f(x))$, and gives a representation
of $\opname{Aut}(x)$ on this vector space where $g : x \ra x$ acts by
$f(g)$.  This is the \textit{pullback representation}.  In the special
case where $f$ gives an inclusion of groups, this is usually called
the ``restricted representation''.

The adjoint process to the restriction of representations is generally
called finding the \textit{induced representation} (see, e.g. Burrows
\cite{burrow} for a classical discussion of this when $f$ is an
inclusion).  We will use the same term for the general case when $f$
is just a homomorphism, and slightly generalize the usual description.

The pushforward $f_{\ast}$, recall, assigns each object the vector
spaces which is the colimit of its essential preimage.  For any
presheaf $V$, this is determined by the colimit for each component of
that essential preimage.  In particular, in the simple case where $\X$
and $\Y$ are discrete (i.e. have only identity isomorphisms, so they
can be thought of as sets, and the essential preimage is just the
usual preimage for sets), for each $y \in \Y$,
\begin{equation}
f_{\ast}F(y) \in \Y = \bigoplus_{g: f(x) \ra y} F(x)
\end{equation}
So we just get the biproduct of all vector spaces over the preimage.

In any component, which can be seen as a group $H$, the colimit is
again a direct sum over the components of the essential preimage, but
each component of the essential preimage amounts to the induced
representiation of $F(x)$ under the homomorphism given by $f$.  So the
colimit is a direct sum of such representations.

To see what this does, consider the case where $\X$ and $\Y$ are just
single groups, so we have a group homomorphism $f : G \ra H$, and we
have a representation of $G$ on $V$.  Now such a representation is the
same as a representation of the group algebra $\CG$ on $V$ - i.e. it
makes $V$ into a $\CG$-module.  Furthermore, $f$ induces an algebra
homomorphism $f : \mathbbm{C}[G] \ra \mathbbm{C}[H]$.

To get a $\mathbbm{C}[H]$-module from $V$ (i.e. in order to produce a
representation of $H$, the pushforward of $V$), we first allow
$\mathbbm{C}[H]$ to act freely on $V$.  Then, to be the pushforward -
that is, the colimit of the diagram $D_y$ described above - we must
take the quotient under the relation that all morphisms coming from
$G$ act on $V$ by letting $f(g)$ have the same action as $g$.  Taking
the quotient, we get $f_{\ast}V = \mathbbm{C}[H]
\otimes_{\mathbbm{C}[G]} V$.

Then for general groupoids, we have a direct sum of such components:
\begin{equation}\label{eq:inducedrep}
f_{\ast}F(y) \in \Y = \bigoplus_{g:f(x) \ra y} \mathbbm{C}[H_y] \otimes_{\mathbbm{C}[G_x]} V
\end{equation}
where $H_y = \opname{Aut}(y)$ and $G_x = \opname{Aut}(x)$.
\end{remark}

\begin{remark}\label{rk:units} To describe an adjunctions, we should
describe its unit and counit.  To begin with, we give a description of
the ``pull-push'':
\begin{equation}
f_{\ast} \circ f^{\ast} : [\Y,\V] \ra [\Y,\V]
\end{equation}

The unit
\begin{equation}
\eta : \mathbbm{1}_{[\Y,\V]} \Longrightarrow f_{\ast} \circ f^{\ast}
\end{equation}
is a natural transformation which, for each $V \in [\Y,\V]$ gives a
morphism.  This is itself a natural transformation between functors:
\begin{equation}
\Delta_y : V(y) \ra f_{\ast} \circ f^{\ast} V(y)
\end{equation}
This takes $V(y)$ into the colimit described above by a diagonal map.
If there is no special symmetry (the discrete case) and the
colimit is just the direct sum $V \oplus \dots \oplus V$, this map is
obvious.  If not, there is a canonical map into the colimit (a
quotient space) which factors through the direct sum with the diagonal
map.  This is because the map from $V(y)$ to the pullback on any
bject in its essential preimage in $\X$ is evidently the identity,
and then one uses the injection into the colimit.

Now consider the other, ``push-pull'' side of the (two-sided)
adjunction, $f^{\ast} \circ f_{\ast}$.  Here, we first push a presheaf
$V'$ from $\X$ to $\Y$, then pull back up to $\X$, has a similar
effect on the vector spaces.

Here we start with a presheaf $V'$ on $\X$.  The "push-pull" along $f$
just takes every vector space on an object and replaces it by a
colimit over the diagram consisting of all objects with the same image
in $\Y$, and morphisms agreeing with these maps:
$\opname{colim}D_{f(x)}$.  This is because this is the result of
pushing $V'$ along $f$ at $f(x)$, which is then pulled back to $x$.

Then the unit
\begin{equation}
\eta : \mathbbm{1}_{[\X,\V]} \Longrightarrow f^{\ast} \circ f_{\ast}
\end{equation}
is a natural transformation giving, for any presheaf $V'$, a morphism
(i.e. natural transformation of functors):
\begin{equation}
\iota_{x} : V'(x) \ra \opname{colim}D_{f(x)}
\end{equation}
This is just the canonical map into the colimit.
\end{remark}

Now, in Section \ref{sec:CY} we discuss a generalization of 2-vector spaces
based on the fields of measurable Hilbert spaces discussed by Crane
and Yetter \cite{CYmeas}.  This generalization has much in common with
KV 2-vector spaces, but corresponds to infinite dimensional Hilbert
spaces in the way that they correspond to the finite dimensional case.

\subsection{2-Hilbert Spaces}\label{sec:CY}

The KV 2-vector spaces we have discussed so far are a categorified
analog of finite dimensional vector spaces.  However, there are
situations in which this is insufficient, and analogs of infinite
dimensional vector spaces are needed.  Moreover, and perhaps more
important, we have not yet discussed the equivalent of an inner
product on 2-vector spaces.

In fact, both of these issues are closely related to applications to
quantum mechanics.  A standard way to describe a quantum mechanical
system, starting with the corresponding classical system, involves
$L^2$ spaces, which in general will be infinite dimensional, and
possess an inner product.  The relationship is that the Hilbert space
of states of the quantum system is $L^2(X)$, where $X$ is the phase
space of a classical system.  A possible motivation for trying to find
a higher analog for Hilbert spaces is to reproduce this framework for
quantizing a classical theory in the categorified setting.

The form of an inner product on a KV 2-vector space is not difficult
to infer from the intuition that categorification corresponds to
replacing sums and products in vector spaces by $\oplus$ and $\otimes$
in 2-vector spaces, together with the fact that a KV 2-vector space
has a specified basis.  However, we will put off describing it until
we have discussed infinite dimensional 2-vector spaces, since we can
put the expression in a more general form.

One approach to infinite-dimensional 2-vector spaces is developed by
Crane and Yetter \cite{CYmeas}, who develop a 2-category called
$\catname{Meas}$.  This is a 2-category of categories, functors, and
natural transformations, but in particular, the objects are all of the
form $\catname{Meas(X)}$ for some measurable space $X$.  This object
can be interpreted as infinite-dimensional 2-vector spaces associated
to $X$, analagous to the Hilbert space $L^2(X)$.  A simplified form of
the definition is as follows:

\begin{definition}Suppose $(X,\mathcal{M})$ is
a measurable space, so that $X$ is a set and $\mathcal{M}$ is a
sigma-algebra of measurable subsets of $X$.  Then $\catname{Meas(X)}$
is a category with:
 \begin{itemize}
\item Objects: \textit{measurable fields of Hilbert spaces} on
      $(X,\mathcal{M})$: i.e. $X$-indexed families of Hilbert spaces
      $\mathcal{H}_x$ such that the preimage of any $H \in \Hilb$ is
      measurable.
\item Morphisms: \textit{measurable fields of bounded linear maps}
      between Hilbert spaces.  That is, an X-indexed family
\[
f_x : \mathcal{H}_x \ra \mathcal{K}_x
\]
      so that $||f||$, the operator norm of $f$, is measurable.  The
      field $f$ is \textit{bounded} if $||f_x||$ is bounded.
\end{itemize}
\end{definition}

\begin{remark}The original definition given by Crane and Yetter is
somewhat different, in the way it specifies how to identify when a
function selecting $v_x \in \mathcal{H}_x \forall x \in X$ is
measurable.  This somewhat simplified definition should suffice for
our later discussion, since we return to these ideas only briefly in
Chapter \ref{chap:QG}.
\end{remark}

The construction of fields of Hilbert spaces is due to Jacques Dixmier
\cite{dixmier}, although he described them, not as categories, but
merely as Hilbert spaces with a particular decomposition in terms of
the measurable space $X$.  As with $L^2$ spaces, to get what we will
call a 2-Hilbert space, we need to have a standard measure on $X$.
This is used to define a
\textit{direct integral} of Hilbert spaces:
\begin{equation}
H = \int_X^{\oplus}\mathcal{H}_x \mathd \mu(x)
\end{equation}
As a vector space, this is the direct sum of all $\mathcal{H}_x$.  The
measure enters when we define its inner product:
\begin{equation}
\br{\phi|\psi} = \int_X \br{\phi_x|\psi_x} \mathd \mu (x)
\end{equation}
We will use this notation to define 2-linear maps of 2-Hilbert spaces.

The 2-vector space $\catname{Meas(X)}$ is the category of all
measurable fields of Hilbert spaces on $X$.  Then we have the
2-category of all such categories:

\begin{definition}The 2-category $\catname{Meas}$ is the collection of
all categories $\catname{Meas(X)}$, with functors between them, and
natural transformations between functors.
\end{definition}

Crane and Yetter describe how functors between such categories arise from:
\begin{itemize}
\item a measurable field of Hilbert spaces $\mathcal{K}_{(x,y)}$ on $X \times Y$
\item a $Y$-family $\mu_y$ of measures on $X$
\end{itemize}
Given these things, there is a functor $\Phi_{\mathcal{K},\mu_y}:
\catname{Meas(X)} \ra \catname{Meas(Y)}$ any field
$\mathcal{H}_x$ on $X$:
\begin{equation}
\Phi_{\mathcal{K},\mu_y}(\mathcal{H})_y = \int_X^{\oplus}\mathcal{H}_x \otimes \mathcal{K}_{(x,y)}\mathd \mu_y(x)
\end{equation}

This is a generalization of the 2-linear maps between
Kapranov-Voevodsky 2-vector spaces: summing over an index set in
matrix multiplication is a special case of integrating over $X$, when
$X$ is a finite set with counting measure (and all the vector spaces
$\mathbbm{C}^n$ which appear as components in a 2-vector or 2-linear
map are equipped with the standard inner product).  Indeed, these
functors generalize the matrices (\ref{eq:kv2linmatrix}).  Yetter
conjectures that all functors between categories like
$\catname{Meas(X)}$ are of this form.

The 2-maps are ways to get from one functor to another.  In this case,
given $\Phi_{\mathcal{K},\mu_y}$ and $\Phi_{\mathcal{K'},\nu_y}$, if
there is such a 2-map, it will be given by:
\begin{itemize}
\item A measurable field of bounded linear operators
\begin{equation}
\alpha_{(x,y)}: \mathcal{K}_{(x,y)} \ra \mathcal{K'}_{(x,y)}
\end{equation}
\item A $Y$-indexed family $\bigl{\{}\bigl{(}\frac{\mathd \nu}{\mathd \mu}\bigr{)}_y\bigr{\} \text{s.t.} y \in Y}$, the \textit{Radon-Nikodym} derivatives of $\nu_y$ w.r.t. $\mu_y$ (or, equivalently, )
\end{itemize} Once again, the KV 2-vector space situation is a special
case as above.

Now, just as integration is used to define the inner product on
$L^2(X)$, the direct integral gives a categorified equivalent of an
\textit{inner product} of fields of Hilbert spaces:
\begin{equation}
\br{\mathcal{H}|\mathcal{H'}} = \int_X^{\oplus}\mathcal{H}_x^{\ast} \otimes \mathcal{H'}_x\mathd \mu(x)
\end{equation}
So in particular, the inner product is given by linearity, and the
fact that, for $\phi_i \in \mathcal{H}$ and ${\phi '}_i \in
\mathcal{H'}$:
\begin{equation}
\br{\phi_1^{\ast} \otimes {\phi '}_1|\phi_2^{\ast} \otimes {\phi '}_2} = \int_X \br{\phi_1^{\ast}|\phi_2^{\ast}} \cdot \br{\phi'_1|\phi'_2} \mathd \mu(x)
\end{equation}
where $\phi^{\ast}$ is the dual of $\phi$, namely $\br{\phi|-}$.

%

We will mostly consider the finite-dimensional (Kapranov-Voevodsky)
2-vector spaces, which remain better understood than these infinite
dimensional 2-Hilbert spaces in the style of Crane and Yetter.
However, we return to these ideas to justify some of the physical
motivation for this paper in Chapter \ref{chap:QG}.

\section{Extended TQFTs as 2-Functors}\label{chap:connexttqft}

We began in our preliminary section by discussing Atiyah's description
of an $n$-dimensional TQFT as a functor \begin{equation} Z : \nCobi
\ra \Hilb\ \end{equation} The development since that point has been
aimed at setting up what we need to give a parallel description of an
\textit{extended TQFT} in terms of 2-functors.  This concept extends
the definition of a TQFT to more general manifolds with corners, and
is due to Ruth Lawrence.

One of the values of TQFT's has been as a method for finding
invariants of manifolds, and in particular, for 3-manifolds
(potentially with boundary).  This is closely connected to the subject
of knot theory, since knots are studied by their complement in some
3-manifold.  One way to think of the invariants which appear this way
is as ways of cutting up the manifold into pieces, assigning algebraic
data to the pieces, and then recombining it.  The possibility of
recombining the pieces unambiguously to form the invariant for the
whole manifold is precisely what we want to express as some form of
\textit{functoriality}.

By now we have considered two examples of the process of
categorification.  The first involved passing from $\nCobi$, a
category of manifolds and cobordisms between them, to $\nCob$, a
(double) bicategory in which we allowed cobordisms between cobordisms.
The second case was the passage from $\V$, the category of vector
spaces and linear maps, to $\iiV$ with 2-vector spaces, 2-linear maps,
and natural transformations.

In the first case we saw that in both $\nCobi$ and $\nCob$, each level
of structure involves entities of one higher dimension than the
previous level.  So in $\nCob$, the objects (manifolds) have
codimension one higher in the total spaces represented by the
(isomorphism class of) cobordisms, than is the case in $\nCobi$.  One
sometimes says that categorification allows us to ``go up a
dimension'', or rather codimension.  This theme appears in what is
probably the prototypical example of higher categories (and indeed
categories of any kind), namely homotopy theory. where we consider
homotopies between spaces, homotopies between homotopies, and so
forth.

We want to use this to develop the following definition:
\begin{definition}An \textbf{extended TQFT} is a weak 2-functor
\begin{equation}
Z : \nCob \ra \iiH
\end{equation}
\end{definition}

So in particular, such a $Z$ assigns:
\begin{itemize}
\item To an $(n-2)$-manifold, a 2-Hilbert space (i.e. a
      $\mathbbm{C}$-linear additive category)
\item To an $(n-1)$-manifold, a 2-linear map between 2-Hilbert spaces
      (an exact $\mathbbm{C}$-linear functor) 
\item To an $n$-manifold, a 2-natural transformation between 2-linear maps
\end{itemize} Where all this data satisfies the conditions for a weak
2-functor (e.g. it preserves composition and units up to coherent
isomorphism, and so forth).  To take this as a definition seems
reasonable enough, but we then need to show how particular examples of
extended TQFT's satisfy this definition.

\subsection{$Z_G$ on Manifolds: The Dijkgraaf-Witten Model}\label{sec:ZGonMan}

Here we want to consider explicit construction of some extended TQFT's
based on a finite group $G$.  We saw in Section \ref{sec:fhk} that the
Fukuma-Hosono-Kawai construction gave a way to define a regular 2D
TQFT for any finite group.  In that case, space of states for a circle
which is just the centre of the group algebra $\mathbbm{C}[G]$.  In
particular, this means that the space of states has a basis consisting
of elements of the group $G$.  Each state therefore consists of some
linear combination of group elements.  Extending this to higher
dimensions is somewhat nonobvious, but turns out to be related to the
Dijkgraaf-Witten (DW) model \cite{DW}.  This can also be described as
a \textit{topological gauge theory}.

The DW model describes a flat connection on a manifold $B$ (we use $B$
rather than $M$ here for consistency with our previous notation).
Being flat, the nontrivial information about a connection is that
which depends only on the topology of $B$.  In particular, all the
information available about the connection comes in the form of
holonomies of the connection around loops.  The holonomy is an element
of the gauge group $G$, which is the ``symmetry group'' of some field.
The element assigned to a loop gives the element of $G$ by which the
field would be transformed if it is ``parallel transported'' around
that loop.  We then define:

\begin{definition} A \textit{flat $G$-bundle} on a connected, pointed
manifold $B$ is a homomorphism
\begin{equation}
A: \pi_1{B} \ra G
\end{equation}
We denote the set of all such functions as $\mathcal{A}_0(B)$.
\end{definition}
This definition is different from the usual concept of a
``$G$-bundle'' equipped with a flat connection in terms of fibre
bundles, but the two concepts are equivalent, as established by
Thurston \cite{thurston}.

Generally, a flat $G$-bundle on $B$ takes loops in $B$ into elements
$G$.  For any loop $\gamma$ in $B$, it assigns an element $A(\gamma)
\in G$.  This is the \textit{holonomy} around the loop $\gamma$.  The
$G$-connection is flat if the holonomy assigned to a loop is invariant
under homotopy. In particular, any contractible loop must have trivial
holonomy.  On the other hand, nontrivial elements of the fundamental
group of $M$ may correspond to nontrivial elements of $G$.  These are
thought of as describing the ``parallel transport'' of some object, on
which $G$ acts as a symmetry group, around the loop.  The usual
picture in gauge theory has this object being the fibre of some
bundle, such as a vector space, so that $G$ is a Lie group such as
$SO(3)$.  However, the same picture applies when $G$ is finite.

However, instead of the \textit{set} of flat bundles here, we want
to categorify this usual picture, to extend TQFT's to give a functor
into $\iiV$.  So there must be a category to take the place of
$\mathcal{A}_0$, which has morphisms as well as objects.  Fortunately,
the structure of gauge theory which we have not captured in the
definition of $\mathcal{A}_0$ does precisely this.

The principle here is that the fundamental group is too restrictive,
and we should instead use the \textit{fundamental groupoid} of $B$,
and describe connections as functors.

\begin{definition}
The \textbf{fundamental groupoid} $\Pi_1(B)$ of a space $B$ is a groupoid
with points of $B$ as its objects, and whose morphisms from $x$ to $y$
are just all homotopy classes paths in $B$ starting at $x$ and ending
at $y$.  
\end{definition}

The operation of taking $\Pi_1$ of a space can be thought of as a form
of categorifying: instead of spaces considered as sets of points (with
some topology), we now think of them as categories, whose set of
objects is just the original space.  In fact, these categories are
groupoids, since we consider paths only up to homotopy, so every
morphism is invertible.  Moreover, a loop can be thought of as an
automorphism of the chosen base point in $B$, so the fundamental group
$\pi_1(B)$ is just the group of automorphisms of a single object in
$\Pi_1(B)$.

Then, following the principle that a connection gives a group element
in $G$ for each such loop, we can generalize this to the whole of
$\Pi_1(B)$:

\begin{definition}
A \textbf{flat connection} is a functor
\begin{equation}
A : \Pi_1(B) \ra G
\end{equation}
where $G$ is thought of as a one-object groupoid (hence every $b \in
B$ is sent to the unique object).  A \textbf{gauge transformation}
$\alpha : A \ra A'$ from one connection to another is a natural
transformation of functors: it assigns to each point $x \in B$ a group
element in such a way that for each path $\gamma : x \ra y$ the
naturality square
\begin{equation}
\xymatrix{
\star \ar[r]^{A(\gamma)} \ar[d]_{\alpha(x)} & \star \ar[d]^{\alpha(y)} \\
\star \ar[r]_{A'(\gamma)} & \star \\
}
\end{equation}
commutes.
\end{definition}

\begin{remark}
Using the notation that $[ C_1, C_2 ]$ is the category whose objects
are functors from $C_1$ to $C_2$ and whose morphisms are natural
transformations, then we can say that flat connections and natural
transformations form the objects and morphisms of the category
\begin{equation}
\fc{B}
\end{equation}
\end{remark}

Physicaly, a gauge transformation can be thought of as a change, at
each point in $B$, of the way of measuring the internal degrees of
freedom of the object which is transformed by $G$.  In gauge theory,
two connections which are related by a gauge transformation are
usually considered to describe physically indistinguishable states -
the differences they detect are due only to the system of measurement
used.

We stop here to note that this definition is somewhat different from
the usual notion of a smooth connection on a bundle---indeed, we have
not used any concept of smoothness.  To make all these connections
into smooth connections on a definite bundle would be impossible.
What we have described would have to be a sum over all possible
bundles.  However, for discrete $G$, we can ignore this issue.

So then the ``configuration space'' for an $(n-2)$-dimensional
manifold $B$ in our extended TQFT will be a category whose objects are
flat $G$-connections on $B$ and whose morphisms (all invertible) are
gauge transformations between connections.

\begin{remark} If $\gamma : x \ra x$ in $\Pi_1(B)$ is a loop, and $A$
and $A'$ are two connections related by a gauge transformation
$\alpha$, we have $A'(\gamma) = \alpha(x)^{-1} A(\gamma) \alpha(x)$ -
that is, the holonomies assigned by the two connections around the
loop are conjugate.  So physically distinct holonomies correspond to
conjugacy classes in $G$.

In particular, in the case of 1-dimensional manifolds, if $B$ is just
a circle, then the space of states of the field in the DW model has a
basis consisting just of elements in the centre of $G$.  (We remark
here that this is the same as the TQFT for the FHK construction, which
we have obtained now in a different way.)
\end{remark}

But indeed, any category, and in particular the groupoid $\Pi_1(B)$,
is equivalent to its skeleton.  If $B$ is connected, all points are
related by paths, so $\Pi_1(B) \cong \pi_1(B)$: the fundamental group,
as a single-object category, is equivalent to the path category.
However, the gauge transformations for connections measured from a
fixed base point are determined by a single group element at the base
point, which acts on holonomies around any loop by conjugation.

The groupoid $\fc{B}$, the configuration space for our theory, is the
``moduli stack'' of connections \textit{weakly} modulo gauge
transformations.  This is a categorified equivalent of the usual
physical configuration space, which consists of the set of equivalence
classes of flat connetions modulo gauge transformations.  Instead of
imposing equations between connections related by gauge
transformations, however, we simply add isomorphisms connecting these
objects.  This is the ``weak quotient'' of the space of connections by
the action of a group.

Finally, using this, we can define a 2-vector space associated to any
manifold:
\begin{definition}\label{def:ZGonB}
For any compact manifold $B$, and finite group $G$, define $Z_G(B)$ to
be the functor category $\Z{B}$.
\end{definition}
as we verify in the following theorem.

\begin{theorem}\label{thm:ZB2VS}
For any compact manifold $B$, and finite group $G$, the functor
category $Z_G(B)=\Z{B}$ is a Kapranov-Voevodsky 2-vector space.
\end{theorem}
\begin{proof}
First, note that for any space $B$,
\begin{equation}
\Pi_1(B) \equiv \coprod_{i=1}^{n}(\pi_1(B_i))
\end{equation}
where the sum is taken over all path components of $B$.  That is,
objects in $\Pi_1(B)$ are by definition isomorphic if and only if they
are in the same path component.  But this groupoid is equivalent to a
skeletal version which has just one object for each isomorphism
class---that is, one object for each path component.  The
automorphisms for the object corresponding to path component $B_i$ are
then just the equivalence classes of paths from any chosen point to
itself---namely, $\pi_1(B_i)$.

Moreover, if $B$ is a compact manifold, so is each component $B_i$,
which is also connected.  But the fundamental group for a compact,
connected manifold is finitely generated.  So in particular, each
$\pi_1(B_i)$ is finitely generated, and there are a finite number of
components.  So $\Pi_1(B)$ is an essentially finitely generated
groupoid.

But if $\Pi_1(B)$ is essentially finitely generated, then since $G$ is
a finite group, $\fc{B}$ is an essentially finite groupoid.  This is
because each functor's object map is determined by the images of the
generators, and there are finitely many such assignments.  Similarly,
$\Pi_1(B)$ is equivalent to a skeleton of itself, and a natural
transformation in this case is just given by a group element in $G$
for each component of $B$, so there are finitely many of these.  By
Lemma \ref{lemma:fgfckv} this means that $\Z{B}$ is a KV 2-vector
space.
\end{proof}

So we have a KV 2-vector space for each manifold, which is defined as
$\V$-valued functors, on the groupoid $\fc{B}$.  As remarked earlier,
we will describe these as $\V$-presheaves, since $\fc{B}$ is
isomorphic to $\fc{B}^{\opname{op}}$.

\begin{example}\label{ex:ZonS1}
Consider the circle $S^1$.

The 2-vector space assigned to the cricle by our TQFT $Z_G$ is the
Hilbert space of of flat connections modulo gauge transformations, on
the circle:
\begin{equation}
\Z{S^1}
\end{equation}
Now, $\fc{S^1}$ looks like the group $G$ equipped with the adjoint
action on itself, in the following sense.  The fundamental group of
the circle is $\mathbbm{Z}$, and $\Pi_1(S^1)$ is thus equivalent to
$\mathbbm{Z}$ as a one-object category.  Then taking maps into $G$, we
note that each functor takes the unique object of $\mathbbm{Z}$ to the
unique object of $G$, and thus is determined entirely by the image of
$1 \in \mathbbm{Z}$.  This will be some morphism $g \in G$ (i.e. an
element of the group $G$), so we simply denote the corresponding
functor by $g$.

A natural transformation between two functors $g$ and $g'$ assigns to
the single object in $\mathbbm{Z}$ a morphism $h \in G$---that is, it
is again a group element.  This must satisfy the naturality condition
that $g'h = hg$, or simply $g' = hgh^{-1}$.  So there is a natural
transformation between functors for each conjugacy relation of this
kind.

So $\fc{S^1}$ is equivalent to a groupoid whose objects correspond to
elements of the group $G$, and whose morphisms are conjugacy relations
between elements (which are clearly all invertible).  This is also
known as $G$ weakly modulo $G$, or $G /\!\!/ G$.  Another
equivalent category is the skeleton of this, whose set of objects is
the set of conjugacy classes of $G$.  Each such object has a group of
automorphisms $\opname{Stab}(g)$, the stabilizer of any element in it.

Finally, the 2-vector space corresponding to the circle is the
category of functors from $G /\!\!/ G$ into $\V$.  This gives a vector
space for each object (element of $G$).  It also assigns an
isomorphism in $\V$ for each isomorphism in $G /\!\!/ G$: the functors
must be equivariant under conjugation by any $h \in G$.  So the
adjoint action of G on itself is already built into this 2-vector
space, and an object of $\ZVG$ is functor $F : G \ra \V$, which comes
equipped with natural isomorphisms
\begin{equation}
R_g: F \ra gFg^{-1}
\end{equation}
such that
\begin{equation}
R_g R_h = R_{gh}
\end{equation}
Where $gFg^{-1}$ is a functor whose image vector space at a point $h$
under $F$ becomes the image of the point $ghg^{-1}$.

So we have \textit{$G$-equivariant functors} as the objects of the
2-vector space, and all $G$-equivariant natural transformations
between them as the morphisms.

As a 2-vector space, this category of $G$-equivariant functors can be
described in terms of its irreducible objects---since every other
functor is isomorphic to a direct sum of these.  Any equivariant
functor will have the same value on every element of each conjugacy
class in $G$, but an irreducible one will only assign nonzero to
elements of ONE conjugacy class.

Moreover, since the action of $G$ by conjugation gives linear
isomorphisms between the vector spaces over elements of $\fc{S^1}$,
and since $\fc{S^1}$ is equivalent to its skeleton, we can think of
this functor as specifying as a \textit{conjugacy classes}, and single
vector space $V$ assigned to it, together with a \textit{linear
representation} of $G$ on $V$.

So the objects of $Z{S^1}$ can be seen as consisting of pairs: a
conjugacy class in $G$, and a representation of $G$.
\end{example}

This example of the circle returns to a previous remark about Example
\ref{ex:grp2alg}, the ``group 2-algebra'' $\V[G]$, the generalization
of the group algebra $\mathbbm{C}[G]$.  As seen in Section
\ref{sec:fhk}, a TQFT based on the finite group $G$ assigns
$\ZCG$ to the circle.  So one expects a categorified version to assign
something like the centre of $\V[G]$ to a circle.  What was not
obvious in Example \ref{ex:grp2alg} was exactly what this is.

Irreducible elements of $\ZCG$ are indeed specified by conjugacy
classes of $G$, but as we see here, a difference appears because we
think of functions on $G$ not precisely as a group, but as a groupoid
of connections.  Since the objects are the elements of $G$, and the
morphisms are conjugations (as distinct from the view of a group as a
one-object category), we get something new.  The new ingredient is the
representation of $G$.  We return to this fact for infinite $G$ in
Chapter \ref{chap:QG}.

\begin{example}
Consider the torus $T^2 = S^1 \times S^1$.  We want to find
\begin{equation}
Z_G(T^2) = \Z{T^2}
\end{equation}
This will be equivalent to the category we get if we replace the
fundamental groupoid $\Pi_1(T^2)$ by the equivalent skeletal groupoid.
This is just the fundamental group of $T^2$, which is isomorphic to
$\mathbbm{Z}^2$.  So we simplify here by using this version.

The category $\fc{T^2}$ has, as objects, functors from $\Pi_1(T^2)$ to
$G$ (both seen as a categories with one object), and morphisms which
consist of natural transformations.  A functor $F \in [ \mathbbm{Z}^2
, G ]$ is then equivalent to a group homomorphism from $\mathbbm{Z}^2$
to $G$.  Since $\mathbbm{Z}^2$ is the free abelian group on the two
generators $(1,0)$ and $(0,1)$, the functor $F$ is determined by the
images of these two generators.  The only restriction on $F$ is that
since $\mathbbm{Z}^2$ is abelian, the images $g_1 = F(1,0)$ and $g_2 =
F(0,1)$ must commute.

So the objects of $\fc{T^2}$ are indexed by commuting pairs
of elements $(g_1,g_2) \in G^2$.

A natural transformation $g : F \ra F'$ assigns to the single
object $\star$ of $\mathbbm{Z}^2$ a morphism in $G$---that is, a group
element $h$.  This must satisfy the naturality condition that this
commutes for every $a \in \mathbbm{Z}^2$:
\begin{equation}
\xymatrix{
  \star \ar[d]^{h} \ar[r]^{F(a)} & \star \ar[d]^h \\
  \star \ar[r]_{F'(a)} & \star
}
\end{equation}
Equivalently, since $h$ is invertible, we can write this in the form
$h F(a) h^{-1} = F'(a)$ for all $a$.  This will be true for all $a$ in
$\mathbbm{Z}^2$ as long as it is true for $(1,0)$ and $(0,1)$.

In other words, functors $F$ and $F'$ represented by $(g_1,g_2) \in
G^2$ and $(g'_1,g'_2) \in G^2$, the natural transformations $\alpha :
F \ra F'$ correspond to group elements $h \in G$ which act in both
components at once, so $(h^{-1}g_1h,h^{-1}g_2h) = (g'_1,g'_2)$.

So we have that the groupoid $\fc{T^2}$ is equivalent to $A /\!\!/ G$,
where $A = \{(g_1,g_2) \in G^2 : g_1g_2 = g_2g_1\}$, and the action of
$G$ on $A$ comes from the action on $G^2$ as above.

So the 2-vector space $Z_G(T^2)$ is just the category of
$\V$-presheaves on $A$, equivariant under the given action of $G$.
This assigns a vector space to each connection $(g_1,g_2)$ on $T^2$,
and an isomorphism of these vector spaces for each gauge
transformation $h : (g_1,g_2) \mapsto (h^{-1}g_1h,h^{-1}g_2h)$.
Equivalently (taking a skeleton of this), we could say it gives a
vector space for each equivalence class $[(g_1,g_2)] \in G^2$ under
simultaneous conjugation, and a representation of $G$ on this vector
space.
\end{example}

Both of these examples conform to a general pattern, which should be
clear by now:

\begin{theorem}\label{thm:ZBmatrixsize} The KV 2-vector space $Z_G(B)$
for any connected manifold $B$ is equivalent to $\V^n$, where $n$ is
\begin{equation}
\sum_{[A] \in \mathcal{A}/G} |\{\text{irreps of} \opname{Aut}(A)\}|
\end{equation}
where the sum is over equivalence classes of connections on $B$, and
$\opname{Aut}(A) \subset G$ is the subgroup of $G$ which leaves $A$
fixed. 
\end{theorem}
\begin{proof}
The groupoid $\fc{B}$ is equivalent to its skeleton $S$.  This has as
objects the gauge equivalence classes of connections on $B$, and on
each object, a group of morphisms isomorphic to the group of gauge
transformations fixing a representative (i.e. the automorphism group
of any object in the original $\fc{B}$).  Now we want to consider
$[S,\V]$, which is equivalent to $\Z{B}$.  We know $[S,\V]$ is a KV
vector space, hence equivalent to some $\V^n$, where $n$ is the number
of nonisomorphic simple objects.  So consider what these are.

A functor $F : S \ra \V$ assigns a vector space to each equivalence
class of connections (i.e. each object), but also a representation of
the group of automorphisms of that object.  This is $\opname{Aut}(A)$.
Note that two functors giving inequivalent representations cannot have
a natural isomorphism between them.  On the other hand, any
representation of $\opname{Aut}(A)$ is a direct sum of irreducible
representations.  So a simple objects in $\Z{B}$ amount to a choice of
$[A]$, and an irreducible representation of $\opname{Aut}(A)$.  The
theorem follows immediately.
\end{proof}

The next thing to consider is how $Z_G$ will act on cobordisms.

\subsection{$Z_G$ on Cobordisms: 2-Linear Maps}\label{sec:ZGonCob}

We have described a construction which builds an extended TQFT from a
finite group $G$.  This takes a manifold $M$---possibly with boundary
or corners---and produces a 2-vector space of states on it.  This
involved a 2-step construction: first one finds $\fc{M}$, the moduli
stack of flat connections; then one takes $\Z{M}$, which is the
2-vector space having $\fc{M}$ as basis.

This begins to describe the extended TQFT $Z : \nCob \ra
\catname{2Vect}$ that we are interested in.  However, $Z$ is to be a
2-functor, and so far we have only described what it does to objects
of $\catname{Top}$.  This tells us its effect on objects in $\nCob$,
and goes some way to describing its effect on morphisms, but recall
that a morphism in $\nCob$ can be seen as a cospan in $\catname{Top}$.
A cobordism (``space'') from a boundary $B$ to a boundary $B'$ is the
cospan given by inclusion maps:
\begin{equation}\label{eq:zstep0}
\xymatrix{
  & S & \\
  B \ar[ur]^{\iota} & & B' \ar[ul]_{\iota'} \\
}
\end{equation}

Our construction amounts to a sequence of functorial operations, which
therefore give a corresponding sequence of spans (or cospans) in three
different catgories.  Next we will consider each of these steps in
turn, remarking on the co- or contravariance of the operation at each
step.

The first step is the operation of taking the fundamental groupoid.
This is somewhat more elaborate than the fundamental group of a
(pointed) space, but it is closely related.  Since any inclusion of
spaces gives an inclusion of points, and also of paths, we again have
a cospan:
\begin{equation}\label{eq:zstep1}
\xymatrix{
  & \Pi_1(S) & \\
\Pi_1(B) \ar[ur]^{\iota} & & \Pi_1(B') \ar[ul]_{\iota'} \\
}
\end{equation}
(Where we are abusing notation somewhat by using the same notation for
the inclusion maps of spaces and path groupoids.)

In the next step, we apply a contravariant functor, $[ - , G]$.
Recall that we are thinking of the group $G$ as the category with one
object $\star$ and the elements of $G$ as morphisms.  Taking
functors into $G$ is contravariant, since if we have a functor $F : X
\ra Y$, then any from $Y$ into $G$ becomes a map from $X$ into
$G$ by pullback along $F$ (i.e. $\psi \mapsto \psi \circ F =
F^{\ast}\psi$).  That is, we get a functor $F^{\ast} : [ Y , G ] \ra [
X , G ]$.  So at this stage of the construction we have a span:
\begin{equation}\label{eq:zstep2}
\xymatrix{
  & \fc{S} \ar[dl]^{\pi} \ar[dr]_{\pi'} & \\
\fc{B}  & & \fc{B'}  \\
}
\end{equation}
For convenience here we have made the convention that the pullback
maps along the inclusions are denoted $\iota^{\ast} = \pi$ and
$\iota'^{\ast} = \pi'$.

Finally, to this span, we apply another functor, namely $[ - , \V ]$.
This is contravariant for the same reason as $[ - , G ]$, and thus we
again have a cospan:
\begin{equation}\label{eq:zstep3}
\xymatrix{
  & \Z{S} & \\
\Z{B} \ar[ur]^{\pi^{\ast}} & & \Z{B'} \ar[ul]_{\pi'^{\ast}} \\
}
\end{equation}

We now recall that the pullbacks $\pi^{\ast}$ and $\pi'^{\ast}$ have
adjoints: this is a direct consequence of Theorem
\ref{thm:2mapadjoints}.  This reveals how to transport a $\V$-presheaf
on $\fc{B}$ along this cospan.  In fact, it gives two 2-linear maps,
which are adjoint.  Having written the cobordism as a morphism from
$B$ to $B'$, we find a corresponding 2-linear map, though we observe
that the adjoint is equally well defined.  We first do a pullback
along $\pi$, giving a $\V$-presheaf on $S$.  Then we use the adjoint
map $\pi'_{\ast}$.  So we have the following:

\begin{definition}\label{def:ZGonS} For any cobordism $S : B \ra B'$
between compact manifolds, and finite group $G$, define $Z_G(S)$ to be
the 2-linear map:
\begin{equation}\label{eq:zstep4}
(\pi')_{\ast} \circ \pi^{\ast} : Z_G(B) \longrightarrow Z_G(B')
\end{equation}
\end{definition}

Here we have used the notation of Definition \ref{def:ZGonB}.  Note
that again by Theorem \ref{thm:2mapadjoints}, both of these functors
are 2-linear maps, so the composite $\pi'_{\ast} \circ
\pi^{\ast}$ is also a 2-linear map.  It remains to show that $Z_G$
preserves horizontal composition of functors \textit{weakly}---that
is, up to a natural isomorphism.

\begin{remark} We can think of the pullback-pushforward construction
as giving---in the language of quantum field theory---a ``sum over
histories'' for evolving a 2-vector built from the space of
connections.  Each 2-vector in $\Z{B}$ picks out a vector space for
each $G$-connection on $B$.  The 2-linear map we have described tells
us how to evolve this 2-vector along a cobordism (i.e. a change of
spatial topology).  First we consider the pullback to $\Z{S}$, which
gives us a 2-vector consisting of all assignments of vector spaces to
connections on $S$ which restrict to the chosen one on $B$.  Each of
these could be considered a ``history'' of the 2-vector along the
cobordism.  We then ``push forward'' this assignment to $B'$, which
involves a colimit.  This is more general than a sum, though so one
could describe this as a ``colimit of histories''.  It takes into
account the symmetries between individual ``histories''
(i.e. connections on the cobordism, which are related by gauge
transformations).
\end{remark}

It still needs to be seen that this operation is compatible with
composition of cobordisms.  Now, a composite of two cobordisms is a
special case of a composite of cospans.  This is a composition in a
bicategory cobordisms---either the horizontal or vertical bicategory
in the {\vdb} defined in Chapter \ref{chap:cobcorn}.  It is given by a
pushout as described in Definition \ref{def:cospan}:
\begin{equation}\label{xy:cobordcomposite}
  \xymatrix{
      &   & S' \circ S  &  & \\
      & S  \ar[ur]_{i_S}  &   & S' \ar[ul]^{i_{S'}} & \\
    B_1 \ar[ur]_{i_1} &  & B_2 \ar[ul]^{i_2} \ar[ur]_{i'_1} &   & B_3 \ar[ul]^{i'_2} \\
  }
\end{equation}

When we take the groupoid of connections, however, the corresponding
diagram of spans between groupoids of connections weakly mod gauge
transformations contains a weak pullback square.  This is since the
objects are now groupoids, it makes sense to speak of two connections
being gauge equivalent, whereas the manifolds in cobordisms are sets,
where elements can only be equal or unequal.  So for connections on
$S$ and $S'$, it is possible for the restrictions to the same set
$B_2$ to be isomorphic, rather than merely equal.  Thus, we should
consider this larger groupoid, the weak pullback, whose objects come
with a specified isomorphism between the two restrictions:

\begin{equation}\label{xy:conncompos}
 \xymatrix{
   & & \fc{S'\circ S} \ar_{P_S}[ld]\ar^{p_{S'}}[rd]\ar@/_2pc/_{p_1 \circ P_S}[ddll]\ar@/^2pc/^{p'_2 \circ P_{S'}}[ddrr] & & \\
   & \fc{S}\ar^{p_1}[ld]\ar_{p_2}[rd] \ar@{=>}[rr]^{\alpha}_{\sim} & & \fc{S'}\ar^{p'_1}[ld]\ar_{p'_2}[rd] & \\
   \fc{B_1} & & \fc{B_2} & & \fc{B_3} \\
}
\end{equation}

That this is a weak pullback square of functors between groupoids
means that this diagram commutes up to the natural isomorphism $\alpha
: p_2 \circ P_S \longrightarrow p'_1 \circ P_{S'}$.  In the case of
groupoids, a weak pullback can be seen as an example of a
\textit{comma category} (the concept, though not the name, introduced
by Lawvere in his doctoral thesis \cite{lawverethesis}).  We briefly
discuss this next before stating the theorem regarding composition.

\begin{remark}\label{rk:commasquare} In general, suppose we have a
diagram of categories $\catname{A} \ralim^{F} \catname{C} \lalim^{G}
\catname{B}$.  Then an object in the comma category $(F \downarrow G)$
consists of a triple $(a,f,b)$, where $a \in \catname{A}$ and $b \in
\catname{B}$ are objects, and $f : F(a) \ra G(b)$ is a morphism in
$\catname{C}$.  A morphism in $(F \downarrow G)$ consists of a pair of
morphisms $(h,k) \in \catname{A} \times \catname{B}$ making the square
\begin{equation}\label{eq:commacatmor}
\xymatrix{
F(a_1) \ar[r]^{f_1} \ar[d]_{F(h)} & G(b_2) \ar[d]^{G(k)} \\
F(a_2) \ar[r]_{f_2} & G(b_2)
}
\end{equation}
commute.  Note that in a weak pullback, the morphisms $f$ would be
required to be an \textit{isomorphism}, but when we are talking about
a weak pullback of groupoids, these conditions are the same.

The comma category has projection functors which complete the (weak)
pullback square for the two projections:
\begin{equation}
\xymatrix{   & (F \downarrow G) \ar_{P_A}[ld] \ar^{P_B}[rd] & \\
   \catname{A}\ar_{F}[rd] \ar@{=>}[rr]^{\alpha}_{\sim} & & \catname{B}\ar^{G}[ld] \\
   & \catname{C} & \\
}
\end{equation}
such that $(F \downarrow G)$ is a universal object (in $\Cat$) with
maps into $\catname{A}$ and $\catname{B}$ making the resulting square
commute up to a natural isomorphism $\alpha$.  This satisfies the
universal condition that, given any other category $\catname{D}$ with
maps to $\catname{A}$ and $\catname{B}$, there's an equivalence
between $[\catname{D},\catname{C}]$ and the comma category
$(P_A^{\ast},P_B^{\ast})$ (where $P_S*$ and $P_T*$ are the functors
from $\catname{D}$ to $\catname{B}$ which factor through $P_S$ and
$P_T$ respectivery).  This equivalence arises in a natural way.  This
is the weak form of the universal property of a pullback.

So suppose we restrict to the case of a weak pullback of groupoids.
This is equivalent to the situation where $\catname{A}$, $\catname{B}$
and $\catname{C}$ are skeletal - that is, each is just a disjoint
union of groups.  Then the set of objects of $(F \downarrow G)$ is a
disjoint union over all the morphisms of $\catname{C}$ (which are all
of the form $g: x \ra x$ for some object $x$) of all the pairs of
objects $a \in \catname{A}$ and $b \in \catname{B}$ with $g: F(a) \ra
G(b)$.  In particular, since we assume $\catname{C}$ is skeletal, this
means $F(a) = G(b)$, though there will be an instance of this pair in
$(F \downarrow G)$ for each $g$ in the group of morphisms on this
object $F(a) = G(b)$.

So as the set of objects in $(F \downarrow G)$ we have a disjoint
union of products of sets---for each $c \in \catname{C}$, we get
$|\hom(c,c)|$ copies of $F^{-1}(c) \times G^{-1}(c)$.  The set of
morphisms is just the collection of commuting squares as in
(\ref{eq:commacatmor}) above.

Note that if we choose a particular $c$ and $g: c \ra c$, and choose
objects $a$, $b$ with $F(a)=c$, $G(b)=c$, and if $H=\opname{Aut}(a)$,
$K=\opname{Aut}(b)$ and $M=\opname{Aut}(c)$, then the group of
automorphisms of $(a,g,b) \in (F \downarrow G)$ is isomorphic to the
fibre product $H \times_M K$.  In particular, it is a subgroup of the
product group $H \times K$ consisting of only those pairs $(h,k)$ with
$F(h) g = g G(k)$, or just $F(h) = g G(k) g^{-1}$.  We can call it $H
\times_M K$, keeping in mind that this fibre product depends on $g$.
Clearly, the group of automorphisms of two isomorphic objects in $(F
\downarrow G)$ are isomorphic groups.
\end{remark}

In our example, the connections on $S$ and $S'$ need only restrict to
gauge-equivalent connections on $B_2$---since two such connections can
be ``pasted'' together using a gauge transformation.  Moreover, we
note that since all categories involved in our example are groupoids,
we have the extra feature that every morphism mentioned must be
invertible.  This is what makes this a \textit{weak} pullback rather
than a \textit{lax} pullback, where $\alpha$ is only a natural
transformation.

We are interested in the weak pullback square in the middle of
(\ref{xy:conncompos}), since the two 2-linear maps being compared
differ only by arrows in this square.  The square as given is a weak
pullback, with the natural isomorphism $\alpha$ ``horizontally''
across the square.  When considering a corresponding square of
categories of $\V$-presheaves, the arrows are reversed.  So, including
the adjoints of $p_2^{\ast}$ and $p_{S'}^{\ast}$, namely
$(p_2)_{\ast}$ and $(p_{S'})_{\ast}$, we have the square:
\begin{equation}\label{eq:adjunctionmatesquare}
\xymatrix{
   & \Z{S'\circ S} \ar@<1ex>^{(p_{S'})_{\ast}}[rd] & \\
   \Z{S} \ar^{(p_S)^{\ast}}[ur] \ar@<1ex>^{(p_2)_{\ast}}[rd] & & \Z{S'} \ar@<1ex>^{(p_{S'})^{\ast}}[ul] \\
   & \Z{B_2} \ar_{(p'_1)^{\ast}}[ur] \ar@<1ex>^{(p_2)^{\ast}}[ul] &
} \end{equation}

Note that there are two squares here---one by taking only the ``pull''
morphisms $(-)^{\ast}$ from the indicated adjunctions, and the other
by taking only the ``push'' morphisms $(-)_{\ast}$.  The first is just
the square of pullbacks along morphisms from the weak pullback square
of connection groupoids.  Comparing these is the core of the following
theorem, which gives one of the necessary properties for $Z_G$ to be a
weak 2-functor.

\begin{theorem}\label{thm:zcomposfunc}The process $Z_G$ weakly
preserves composition.  In particular, there is a natural isomorphism
\begin{equation}
\beta_{S',S} : Z_G(S' \circ S) \ra Z_G(S') \circ Z_G(S)
\end{equation}
\end{theorem}
\begin{proof}
The process $Z_G$ acts by on $S' \circ S$ by taking the spans of
groupoids in \ref{xy:conncompos}, and giving 2-linear maps:
\begin{equation}
(p'_2 \circ P_{S'})_{\ast} \circ (p_1 \circ P_S)^{\ast}
\end{equation}
On the other hand, $Z_G(S') \circ Z_G(S)$ is found in the same diagram to be
\begin{equation}
(p'_2)_{\ast} \circ (p'_1)^{\ast} \circ (p_2)_{\ast} \circ (p_1)^{\ast}
\end{equation}
So we want to show there is a natural isomorphism:
\begin{equation}
\beta_{S',S} : (p'_2 \circ P_{S'})_{\ast} \circ (p_1 \circ P_S)^{\ast} \ra
(p'_2)_{\ast} \circ (p'_1)^{\ast} \circ (p_2)_{\ast} \circ (p_1)^{\ast}
\end{equation}
It suffices to show that there is an isomorphism between the upper and
lower halves of the square in the middle:
\begin{equation}
\gamma: (P_{S'})_{\ast} \circ (P_S)^{\ast} \ra (p'_1)^{\ast} \circ(p_2)_{\ast}
\end{equation}
since then $\beta_{S',S}$ is obtained by tensoring with identities.

Now, as we saw when discussing comma squares, the objects of the weak
pullback $\fc{S' \circ S}$ consist of pairs of connections, $A \in
\fc{S}$, and $A' \in \fc{S'}$, together with a morphism in ${B_2}$, $g
: p_2(A) \ra {p'}_1(A')$.  The morphisms from $(A_1,g_1,{A'}_1)$ to
$(A_2,g_2,{A'}_2)$ in the weak pullback are pairs of morphisms, $(h,k) \in
\fc{S} \times  \fc{S'}$, making the square
\begin{equation}
\xymatrix{
p_2(A_1) \ar[r]^{g_1} \ar[d]_{p_2(h)} & {p'}_1({A'}_2) \ar[d]^{{p'}_1(k)} \\
p_2(A_2) \ar[r]_{g_2} & {p'}_1({A'}_2)
}
\end{equation}
commute.

We may assume that the groupoids we begin with are skeletal---so the
objects consist of gauge equivalence classes of connections.  Then
recall from Remark \ref{rk:commasquare} that in this weak pullback the
set of objects in $\fc{S' \circ S}$ is a disjoint union of products of
sets - for each $c \in \fc{B_2}$, we get $|\hom(c,c)|$ copies of
${p_2}^{-1}(c) \times {p'}_1^{-1}(c)$.

So first taking a $\V$-presheaf $F$ on $\fc{S}$, we get that
$(P_S)^{\ast}F$ is a $\V$-presheaf on $\fc{S' \circ S}$.  Now over any
fixed object (connection) $A$, we have a set of objects in $\fc{S'
\circ S}$ which restrict to it: there is one for each choice $(g,A')$
which is compatible with $A$ in the sense that $(A,g,A')$ is an object
in the weak pullback - that is, $g: p_2(A) \ra {p'}_1(A')$.  Each
object of this form is assigned the vector space $F(A)$ by
$(P_S)^{\ast}F$.

Further, there are isomorphisms between such objects, namely pairs
$(h,k)$ as above.  There are thus no isomorphisms except between
objects $(A,g_1,A')$ and $(A,g_2,A')$ for some fixed $A$ and $A'$.
For any such fixed $A$ and $A'$, objects corresponding to $g_1$ and
$g_2$ are isomorphic if
\begin{equation}
g_2 p_2(h)  = {p'}_1(k) g_1
\end{equation}.
Denote the isomorphism class of any $g$ by $[g]$.

Then if $G_A$ is the group of automorphisms of any gauge equivalence
class of connections $A$, and for notational convenience $M$ is here
the group of automorphisms of $p_2(A)$ in $\fc{B_2}$ (note that this
$M$ depends on $A$, which we are considering fixed for now), we get:
\begin{equation}
(P_{S'})_{\ast} \circ (P_S)^{\ast}F (A') = \bigoplus_{A} \Bigl{(}
\bigoplus_{[g]: p_2(A) \ra {p'}_1(A')} \mathbbm{C}[G_{A'}]
\otimes_{\mathbbm{C}[G_A \times_{M} G_{A'}]} F(A) \Bigr{)}
\end{equation}
since $G_A \times_M G_{A'}$ is the automorphism group of the object in
$\fc{S' \circ S}$ which restricts to $A$ and $A'$ by gluing along $g$.
The outside direct sum here is written over all connections $A$ on
$S$, but note that the only ones which contribute any factor are those
for which $g : p_2(A) \ra {p'}_1(A')$ for some $g$.  The inside direct
sum is over all isomorphism classes of elements $g$ for which this
occurs: in the colimit, vector spaces over objects with isomorphisms
between them are identified.

Note that in the direct sum over $[g]$, there is a tensor product term
for each class $[g]: p_2(A) \ra {p'}_1(A')$.  By the definition of the tensor
product over an algebra, we can pass elements of $\mathbbm{C}[G_A
\times_{M} G_{A'}]$ through the tensor product.  These are generated
by pairs $(h,k) \in G_A \times G_{A'}$ where the images of $h$ and $k$
are conjugate by $g$ so that $p_2(h) g = g {p'}_1(k)$.  These are just
automorphisms of $g$: so this says we are considering objects only up
to these isomorphisms.

This is the result of the ``pull-push'' side of the square applied to
$F$.  Now consider the ``push-pull'' side: $(p'_1)^{\ast}
\circ(p_2)_{\ast}$.

First, pushing down to $B_2$, we get, on any connection $A''$ on $B_2$
(whose automorphism group is $M$):
\begin{equation}
(p_2)_{\ast}F(A'') = \bigoplus_{p_2(A) =A''} \mathbbm{C}[M] \otimes_{\mathbbm{C}[G_A]}F(A)
\end{equation}

Then, pulling this back up to $S'$, we get (with $M$ again the
symmetry group of $p_2(A)$) that:
\begin{equation}
(p'_1)^{\ast} \circ(p_2)_{\ast}F(A') = \bigoplus_{g : p_2(A) \ra {p'}_1(A')} \Bigl{(} \mathbbm{C}[M] \otimes_{\mathbbm{C}[G_{A}]} F(A)  \Bigr{)}
\end{equation}
Now we define a natural isomorphism
\begin{equation}
\gamma_{S,S'} :  (P_{S'})_{\ast} \circ (P_S)^{\ast} \ra (p'_1)^{\ast} \circ(p_2)_{\ast}
\end{equation}
as follows.  For each $A'$, this must be an isomorphism between the above vector
spaces.  The first step is to observe that there is a 1-1
correspondence \textit{between} the terms of the first direct sums, and then
secondly to note that the corresponding terms are isomorphic.

Since the outside direct sums are over all connections $A$ on $S$ for
which $p_2(A) = {p'}_1(A')$, it suffices to get an isomorphism between
each term.  That is, between 
\begin{equation}\label{eq:pullpushcomponentvs}
\bigoplus_{[g]: p_2(A) \ra {p'}_1(A')} \mathbbm{C}[G_{A'}] \otimes_{\mathbbm{C}[G_A \times_{M} G_{A'}]} F(A)
\end{equation}
and
\begin{equation}\label{eq:pushpullcomponentvs}
\mathbbm{C}[M] \otimes_{\mathbbm{C}[G_{A}]} F(A) 
\end{equation}

In order to define this isomorphism, first note that both of these
vector spaces are in fact $\mathbbm{C}[G_{A'}]$-modules.  An element
of $G_{A'}$ acts on (\ref{eq:pullpushcomponentvs}) in each component
by the standard group algebra multiplication, giving an action of
$\mathbbm{C}[G_{A'}]$ by extending linearly.  An element $g \in
G_{A'}$ acts on (\ref{eq:pushpullcomponentvs}) by the action of
${p'}_1(g)$ on $\mathbbm{C}[M]$.  Two $g \in [g]$ have the same action
on this tensor product, since they differ precisely by $(h,k) \in G_A \times
G_{A'}$, so that $g_2 p_2(h) = {p'}_1(k) g_1$.

Also, we notice that, in (\ref{eq:pullpushcomponentvs}), for each $g
\in M$, the corresponding term of the form $\mathbbm{C}[G_{A'}]
\otimes_{\mathbbm{C}[G_A \times_{M} G_{A'}]} F(A)$ is generated by
elements of the form $k \otimes v$, for $k \in \mathbbm{C}[G_{A'}]$.  
and $v \in F(A)$.  These are subject to the relations that, for any
$(h,k_1) \in \mathbbm{C}[G_{A}] \times \mathbbm{C}[G_{A'}]$ such that
$p_2(h) = g^{-1} {p'}_1(k_1) g$:
\begin{equation}
kk_1 \otimes v = k (h,k_1) \otimes v = k \otimes (h,k_1) v = k \otimes hv
\end{equation}
since elements of $\mathbbm{C}[G_{A}] \times \mathbbm{C}[G_{A'}]$ act
on $F(A)$ and $\mathbbm{C}[G_{A'}]$  by their projections into
the first and second components respectively.

Now, we define the map $\gamma_{A,A'}$.  First, for any element of the form
$k \otimes v \in \mathbbm{C}[G_{A'}] \otimes_{\mathbbm{C}[G_A
\times_{M} G_{A'}]} F(A)$ in the $g$ component of the direct sum (\ref{eq:pullpushcomponentvs}):
\begin{equation}
\gamma_{A,A'} (k\otimes v) = {p'}_1(k) g^{-1} \otimes v
\end{equation}
which we claim is in $\mathbbm{C}[M] \otimes_{\mathbbm{C}[G_{A}]}
F(A)$.  This map extends linearly to the whole space.

To check this is well-defined, suppose we have two representatives
$k_1 \otimes v_1$ and $k_2 \otimes v_2$ of the class $k \otimes v$.
So these differ by an element of $\mathbbm{C}[G_A \times_{M} G_{A'}]$,
say $(h,k)$, so that
\begin{equation}
k_1 = k_2 k
\end{equation}, 
and
\begin{equation}
h v_1 =  v_2
\end{equation}
where
\begin{equation}
p_2(h) = g {p'}_1(k) g^{-1}
\end{equation}

But then
\begin{eqnarray}
\gamma_{A,A'}(k_1 \otimes v_1) & = &  {p'}_1(k_1) g^{-1} \otimes v_1 \\
\nonumber & = & {p'}_1(k_2 k) g^{-1} \otimes v_1 \\
\nonumber & = & {p'}_1(k_2) g^{-1} g {p'}_1(k) g^{-1} \otimes v_1 \\
\nonumber & = & {p'}_1(k_2) g^{-1} p_2(h) \otimes v_1\\
\nonumber & = & {p'}_1(k_2) g^{-1} \otimes h v_1
\end{eqnarray}
while on the other hand,
\begin{eqnarray}
\gamma_{A,A'} (k_2 \otimes v_2) & = & {p'}_1(k_2) g^{-1} \otimes v_2\\
\nonumber & = & {p'}_1(k_2) g^{-1} \otimes h v_1
\end{eqnarray}
But these are representatives of the same class in $\mathbbm{C}[M]
\otimes_{\mathbbm{C}[G_{A}]} F(A)$, so $\gamma$ is well defined on
generators, and thus extends linearly to give a well-defined function
on the whole space.

Now, to see that $\gamma$ is invertible, note that given an element $m
\otimes v \in \mathbbm{C}[M] \otimes_{\mathbbm{C}[G_{A}]} F(A)$ (where
we are fixing $A$, since both 2-vectors decompose into components
corresponding to connections $A$), we can define
\begin{equation}
\gamma^{-1}(m \otimes v) = 1 \otimes v \in  \bigoplus_{[g]: p_2(A) \ra {p'}_1(A')} \mathbbm{C}[G_{A'}]
\otimes_{\mathbbm{C}[G_A \times_{M} G_{A'}]} F(A)
\end{equation}
in the component coming from the isomorphism class of $g = m^{-1}$ (we
will denote this by $(1 \otimes v)_{m^{-1}}$ to make this explicit,
and in general an element in the class of $g$ will be denoted with
subscript $g$ whenever we need to refer to $g$).

Now we check that this is well-defined.  Given $m_1 \otimes v_1$ and
$m_2 \otimes v_2$ representing the same element $m \otimes v$ of
$\mathbbm{C}[M] \otimes_{\mathbbm{C}[G_{A}]} F(A)$, we must have $h_1
\in G_A$ with
\begin{equation}
m_1 p_2(h_1) = m_2
\end{equation}
and
\begin{equation}
h_1 v_2 = v_1
\end{equation}
But then applying $\gamma^{-1}$, we get:
\begin{equation}
\gamma^{-1}(m_1 \otimes v_1) = (1 \otimes v_1)_{m_1^{-1}} = (1 \otimes h_1 v_2)_{m_1^{-1}}
\end{equation}
and
\begin{equation}
\gamma^{-1}(m_2 \otimes v_2) = (1 \otimes v_2)_{m_2^{-1}} = (1 \otimes v_2)_{p_2(h_1)^{-1} m_1^{-1}}
\end{equation}
but these are in the same component, since $g \sim g'$ when $g'
{p'}_1(k) = p_2(h) g$ for some $h \in G_A$ and $k \in G_{A'}$.  But
then, taking $k=1$ and $h = h_1^{-1}$, we get that $m_1^{-1} \sim
m_2^{-1}$, and hence the component of $\gamma(m \otimes v)$ is well
defined.

But then, consider $m \otimes v = \gamma((k \otimes v)_g) = {p'}_1(k)
g^{-1} \otimes v$.  Applying $\gamma^{-1}$ we get:
\begin{equation}
\gamma^{-1} \circ \gamma (k \otimes v)_g = (1 \otimes v)_{g {p'}_1(k)^{-1}}
\end{equation}
so we hope that these determine the same element.  But in fact, notice
that the morphism in the weak pullback which gives that $g^{-1}$ and
${p'}_1(k) g^{-1}$ are isomorphic is just labelled by $(h,k) = (1,k)$,
which indeed takes $k$ to $1$ and leaves $v$ intact.  So these are the
corresponding elements under this isomorphism.

So $\gamma$ is invertible, hence an isomorphism.  Thus we define
\begin{equation}
\beta_{S,S'} = 1 \otimes \gamma \otimes 1
\end{equation}
This is the isomorphism we wanted.
\end{proof}

\begin{remark} The weak pullback square gave a natural isomorphism:
\begin{equation}
\alpha^{\ast} :  P_{S'}^{\ast} \circ  (p'_1)^{\ast}   \ra  P_{S}^{\ast} \circ p_2^{\ast}
\end{equation}
Given a connection on a composite of cobordisms $S' \circ S$, $\alpha$ gives
the gauge transformation of the restriction, on their common boundary
$B_2$, needed so the gluing of connections on $S$ and $S'$ is
compatible.

We proved that the other square---the ``mate'' under the adjunctions,
also has a natural isomorphism (``vertically'' across the square),
namely that there exists:
\begin{equation}
\beta_{S,S'} : (P_{S'})_{\ast} \circ (P_S)^{\ast} \ra ({p'}_1)^{\ast} \circ (p_2)_{\ast}
\end{equation}

In fact, these are related by the units for both pairs of adjoint functors:
\begin{equation}
\eta_{S'} : 1_{Z_G(S'\circ S)} \ra  (P_{S'})_{\ast} \circ (P_{S'})^{\ast}
\end{equation}
and
\begin{equation}
\eta_2 : 1_{Z_G(S)} \ra  (p_2)_{\ast} \circ (p_2)^{\ast}
\end{equation}

So the desired ``vertical'' natural transformation across the square
\ref{eq:adjunctionmatesquare} is determined  by the condition that it complete the following
square of natural transformations to make it commute:
\begin{equation}\label{eq:nearlybeck2}
\xymatrix{
(P_{S'})_{\ast} \circ (P_S)^{\ast} \ar@{=>}[rr]^(.4){1 \otimes \eta_2} \ar@{==>}[d]^{\beta_{S,S'}} & & (P_{S'})_{\ast} \circ (P_S)^{\ast} \circ p_2^{\ast} \circ (p_2)_{\ast} \ar@{=>}[d]^{1 \otimes (\alpha^{\ast})^{-1} \otimes 1} \\ 
(p'_1)^{\ast} \circ (p_2)_{\ast} \ar@{=>}[rr]^(.4){1 \otimes \eta_{S'}} & & (P_{S'})_{\ast} \circ (P_{S'})^{\ast} \circ (p'_1)^{\ast} \circ (p_2)_{\ast} \\ 
}
\end{equation}

The crucial element of this is the fact that the (weak) pullback
square for the groupoids of connections in the middle of the
composition diagram gives rise to a square of $\V$-presheaf
categories.  To get this we used that the adjunction between the
pullback and pushforward along the $\pi$ maps had unit and counit
2-morphisms which turn a natural transformation vertically across the
first square to one horizontally across the second.  Note, however,
that we do not expect this to be invertible.  When it is, the square
is said to satisfy the \textit{Beck-Chevalley} (BC) condition.  This
is discussed by B\'enabou and Streicher \cite{BCES}, MacLane and
Moerdijk \cite{macmoer}, and by Dawson, Par\'e and Pronk
\cite{unispan}.
\end{remark}

\begin{remark}It is useful to consider a description of the two
functors between which we have found this natural isomorphism
$\beta_{S,S'}$---namely, the two 2-linear maps across the central
square in (\ref{xy:conncompos}).  See Remark \ref{rk:pullpushunit} for
the general case.  In this situation, these behave as follows:

First, the ``push-pull'': given a functor $f : \fc{S} \ra \V$ (i.e. in
$Z_G(S)$), in the first stage, push forward to a functor in
$Z_G(B_2)$.  This gives, for each connection $C$ on $B_2$, a vector
space which is the colimit of a diagram of the vector spaces $f(C_i)$
for all connections $C_i$ on $S$ which restrict to $C$ on $B_2$.  In
the second stage, pull back to $\fc{S'}$: for each connection $C'$ on
$S'$, find the connection $C$ it restricts to on $B_2$, and assign
$C'$ the vector space obtained for $C$ above.  Namely, the colimit of
the diagram of vector spaces $f(C_i)$ for connections $C_i$ which also
restrict to $C$.

Next, the ``pull-push'' given a functor $f \fc{S} \ra \V$, in the first
stage, pull back to a functor on $\fc{ST}$.  This gives, at each
connection $C$ on $S' \circ S$, a vector space which is just
$f(C|_S)$, the one assigned to the connection given by $C$ restricted
to $S$.  At the second stage, push this forward to a functor in
$\fc{S'}$.  This gives, at each connection $C'$ on $S'$, the colimit
of a diagram whose objects are all the $f(C_i|_S)$ obtained in the
first stage, for any $C_i$ which restricts to $C'$ on $T$.

In both cases there is a colimit over a diagram including all possible
connections on $S$ which match some specified one on $S'$.  This
``matching'' can occur either by inclusion in a bigger entity (the
composite being the minimal cobordism $S' \circ S$ containing both $S$
and $S'$).  Or it can occur just by matching along the shared boundary
$B_2$.  However, since the composition of $S$ and $S'$ is weak, the
groupoid of connections on $S' \circ S$ only needs to have inclusions
of the groupoids $\fc{S'}$ and $\fc{S}$ which agree on $B_2$ up to
gauge equivalence.  This gauge equivalence is part of the
specification of an object in the weak pullback of the groupoids of
connections.
\end{remark}

\begin{remark} \label{rk:composbeta}We can describe more explicitly
the effect of $\beta$.  Suppose we have a composite of cobordisms, $S'
\circ S$.  ,By Lemma \ref{lemma:kv2linmatrix}, we have that the
functors $(P_{S'})_{\ast} \circ (P_S)^{\ast}$ and $(p'_1)^{\ast} \circ
(p_2)_{\ast}$ can be written in the form of a matrix of vector spaces
as in (\ref{eq:kv2linmatrix}).  The matrix components for each
2-linear map are given by colimits of diagrams of vector spaces in
groupoids of connections on $S'$ matching a specified one on $S$.

However, the criterion for ``matching'' is different: when we push
first, then pull, the connections must match exactly on $B_2$; when
pulling first, then pushing, the connections must both be restrictions
of one on $S' \circ S$, but are only required to match up to gauge
equivalence $\alpha$ on $B_2$.  The isomorphism $\beta_{A,A'}$ is just
the isomorphism between the colimits which induced by the permutation
of vector spaces associated to these gauge transformations.

Recall that the source, $Z_G(S' \circ S)$, is given by a matrix
indexed by gauge equivalence classes of connections $[A_1]$ on $B_1$
and $[A_3]$ on $B_3$.  The entries are isomorphic to $\mathbbm{C}^n$
where $n$ is the number of classes of connections on $S_2 \circ S_1$
restricting to representatives of $[A_1]$ and $[A_3]$.

On the other hand, this can be seen (by the isomorphisms $\beta$) as a
matrix product of $Z_G(S')$ with $Z_G(S)$, which has components given
by a direct sum over equivalence classes $[A_2]$ of connections on
$B_2$:
\begin{equation}\label{eq:cobmatmultcomponent}
[Z_G(S_2) \circ Z_G(S_1)]_{[A_1],[A_3]} \ralim^{\beta_{S,S'}} \bigoplus_{[A_2]} [Z_G(S)]_{[A_1],[A_2]} \otimes [Z_G(S)]_{[A_2],[A_3]}
\end{equation}
Recall that $[Z_G(S)]_{[A_1],[A_2]} \cong \mathbbm{C}^m$, where $m$ is
the number of gauge equivalence classes of connections on $S$ which
restrict to $[A_1]$ and $[A_2]$.  Similarly $[Z_G(S')]_{[A_2],[A_3]}
\cong \mathbbm{C}^{m'}$, where $m'$ is the number of classes of
connections on $S'$ which restrict to $[A_2]$ and $A_3$.  Indeed, the
components are just the vector spaces whose bases are these
equivalence classes.

The isomorphism $\beta$ identifies the composite, whose components
count connections on $S' \circ S$, with this product.  This consists
of identification maps in each component.  A component indexed by
$[A_1]$ and $[A_3]$ comes from the groupoid of all connections on $S'
\circ S$ which restrict to $[A_1]$ and $[A_3]$ on $B_1$ and $B_3$.
Each such connection restricts to connections on $S'$ and $S$ by the
maps $\pi_{S'}$ and $\pi_{S}$.  These in turn restrict by $\pi_2$ and
${\pi}'_2$ to $B_2$ to gauge-equivalent (by $\alpha$) connections -
and those restricting to different $[A_2]$ are in different components
of $\fc{S' \circ S}$.  Over each $[A_2]$, the we have the product
groupoid of the groupoids of all connections on $S$ and $S'$
restricting to this $[A_2]$ (and to $[A_1]$ and $[A_3]$ respectively).
So the groupoid of such connections on $S' \circ S$ is isomorphic to a
fibred product over $Z_G(B_2)$.

Then the $([A_1],[A_3])$ component of $Z_G(S' \circ S)$ is a vector
space whose basis is the set of components of this groupoid, and
$\beta$ is an isomorphism which takes which takes the vector spaces
over this to those in (\ref{eq:cobmatmultcomponent}).
\end{remark}

\begin{example} Consider the ``pair of pants'' cobordism (the
``multiplication'' cobordism from the generators of $\iiCob$):
\begin{figure}[h]
\begin{center}
\includegraphics{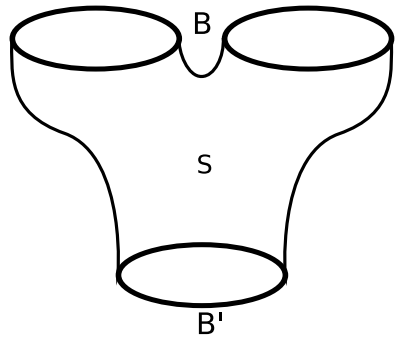}
\end{center}
\caption{The ``Pair of Pants''\label{fig:pants}}
\end{figure}
This can be seen as a morphism $S : B \ra B'$ in $\iiCob$, where
$B=S^1 \cup S^1$ and $B' = S^1$.  The 2-linear map corresponding to
it can be found by the above procedure.  To begin with, recall the
2-vector space on $S^1$ found in Example
\ref{ex:ZonS1}.  It is equivalent to $[G/\!\!/G,\V]$, the 2-vector
space of $\V$-presheaves on $G$ which are equivariant under
conjugation by elements of $G$.

The groupoid of connections on $S^1 \cup S^1$ can be found using
the fact that the path groupoid is just $\Pi_1(S^1) \cup
\Pi_1(S^1)$, a disjoint union of two copies of the groupoid
$\Pi_1(S^1) \cong \mathbbm{Z}$.  Notice that this is different from
the group $\mathbbm{Z}^2$, since a group is a one-object groupoid,
whereas here we have a two-object groupoid, each object having a group
of morphisms isomorphic to $\mathbbm{Z}$.  A functor from this into
$G$ amounts to two choices $g, g' \in G$, but a gauge
transformation amounts to a conjugation by some $h \in G$ at
\textit{each} of the two objects (one chosen base points in each
component), so:
\begin{eqnarray}
\fc{S^1 \cup S^1}  & \cong &  (G \times G) /\!\!/ (G \times G) \\
\nonumber  &\cong &  (G /\!\!/ G)^2
\end{eqnarray}
where $G \times G$ acts on itself by conjugation componentwise.  This
just says that a connection on the space consisting of two circles is
the same as a choice of connection on each one separately.  This is
illustrated in Figure \ref{fig:pantsconnection}, where we show the
pants as a disc with two holes, and label a connection on $S$ with its
restrictions to the boundary.  The connection on $S$ has holonomies
$g$ and $g'$ around the two holes.  On $S^1 \cup S^1$, this
restricts to a connection with holonomies $g$ and $g'$ respectively,
and on $S^1$ to the product (since the circle around the outside $S^1$
is homotopic to the composite of the two loops).

\begin{figure}[h]
\begin{center}
\includegraphics{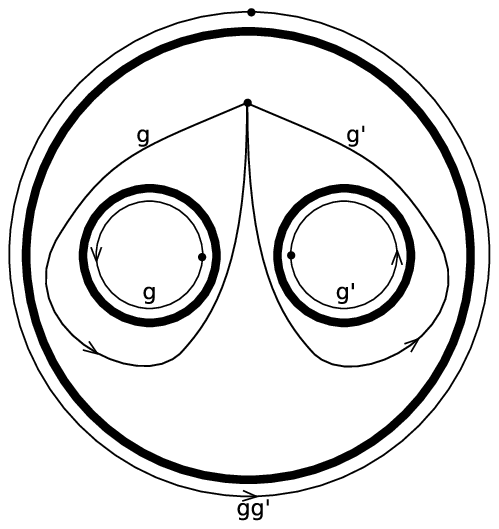}
\end{center}
\caption{Connection for Pants\label{fig:pantsconnection}}
\end{figure}

On the other hand, the manifold with boundary, $S$, is homeomorphic to
a two-punctured disc, whose path groupoid has a skeleton with
one-object, and group of morphisms $\pi_1(S) = F(\gamma_1,\gamma_2)$,
the free group on two generators.  Functors from this into $G$ amount
to homomorphisms $(g,g') : F(\gamma_1,\gamma_2) \ra G$.  That is, a
choice of two elements of $G$ (the images of the generators).  A gauge
transformation amounts to conjugation at the single object (a chosen
base point in $S$---indicated in Figure \ref{fig:pantsconnection} as a
dot on the loop).  So we have the span of connection groupoids:

\begin{equation}
\fc{S} \cong (G \times G) /\!\!/ G
\end{equation}
where $G$ acts on $G \times G$ by conjugation in both components at
once.  Then the span (\ref{eq:zstep2}):
\begin{equation}
\xymatrix{
  & (G \times G)/\!\!/ G \ar[dl]_{p_1} \ar[dr]^{p_2} & \\
 (G /\!\!/ G)^2 & & G /\!\!/G
}
\end{equation}
Both projections are restrictions of a connection on $S$ to the
corresponding connection on the components of the boundary.  It is
easily seen that $p_1$ leaves objects intact and takes the
morphism corresponding to conjugation by $h$ to that corresponding to
conjugation by $(h,h)$.  The projection $p_2$ maps object $(g,g')$
to $gg'$, and the morphism for conjugation by $h$ to, again,
conjugation by $h$.

The gauge-equivalent connections on $S$ have holonomies of the form
$(hgh^{-1},hg'h^{-1})$ for any $h \in G$, and those for $S^1$ are
compatible, since they have holonomies of the form $hgg'h^{-1}$ for $h
\in G$.  Those for $S^1 \cup S^1$ can be any connection with
holonomies $(hgh^{-1},h'g'(h')^{-1})$ for any choices of $(h,h') \in
G^2$, so that connections which are gauge equivalent on $S^1 \cup
S^1$ may be restrictions of inequivalent connections on $S$.

Finally, suppose we have a functor $f : \fc{S^1 \cup S^1} \ra \V$, and
transport it to $(p_2)_{\ast} \circ p_1^{\ast} (f) : \fc{S^1} \ra \V$.
To see what this does, note that since $Z_G(S^1 \cup S^1$ that any
such $f$ can be written as a sum of irreducible functors (since
$Z_G(S^1 \cup S^1)$ is a KV 2-vector space).  So we can consider one
of these, say $f$, which assigns a copy of $\mathbbm{C}$ to each
connection in some gauge-equivalence class, say $([g],[g'])$, and $0$
to all others.  This $f$ assigns an isomorphism, compatibly, to each
gauge transformation (i.e. pair of elements $(h,h')$).  Such an
isomorphism amounts to multiplication by a complex number---so we get
a representation $\rho : G \times G \ra \mathbbm{C}$.

Now pull $f$ back to $p_1^{\ast}(f) : \fc{S} \ra \V$, a functor
$f(p_1(A))$.  This assigns a copy of $\mathbbm{C}$ to any connection
on $S$ which restricts to any representative of $([g],[g'])$---note
that these are not all equivalent.  To any gauge transformation given
by conjugation by $h$, it assigns the isomorphism $\rho(h,h)$.  So we
get the representation $\rho \circ \Delta : G \ra \mathbbm{C}$ for each
equivalence class (where $\Delta : G \ra G \times G$ is the diagonal
map).

Then push $p_1^{\ast}(f)$ forward to $(p_2)_{\ast} \circ p_1^{\ast}(f)
: \fc{S^1} \ra \V$.  To each connection on $S^1$ (represented by $g_1
\in G$) the colimit over the diagram of all connections restricting to
$g_1$.  That is, over all $(g,g')$ such that $gg'=g_1$.  So then we
get a copy of $\mathbbm{C}$ for each pairs of representatives of $[g]$
and $[g']$ which give $g_1$ as a product: note that there may be more
than one such, which are not gauge equivalent in $\fc{S}$.  The
diagram of all these amounts (by taking its skeleton) to just a
disjoint union of gauge-equivalence classes in $\fc{S}$.

For each class (since all copies of $\mathbbm{C}$ over it are equipped
with compatible isomorphisms) we just get one copy of $\mathbbm{C}$.
The group $G$ thought of as the group of gauge transformations acts on
each copy of $\mathbbm{C}$. If it acts nontrivially, then in the
colimit, at least two points in that $\mathbbm{C}$ will be identified
(since the isomorphisms given by the $G$-action must agree with the
maps into the colimit).  If this happens, that copy of $\mathbbm{C}$
collapses to zero.

So finally we have that 
\begin{equation}
(\pi_2)_{\ast} \circ \pi_1^{\ast}(g_1) \cong \bigoplus_{(g,g') \in
([g],[g']) | gg' = g_1} \mathbbm{C}[\opname{Aut}(g_1)]
\otimes_{\mathbbm{C}[\opname{Aut}(g,g')]} \mathbbm{C}
\end{equation}
where the direct sum is over all non-equivalent
$(g,g')$ representing $([g],[g'])$ and satisfying $gg'=g_1$, and the
action of $G$ on each component is as we have described.  On morphisms,
we get the direct sum of the isomorphisms between these copies of
$\mathbbm{C}$.

We can describe this as a categorified ``convolution of class
functions'' on $G$.  This is related to Example \ref{ex:grp2alg}, the
group 2-algebra on a group.  Note that this is almost the 2-vector
space of $\V$-presheaves on the groupoid of connections on $S^1$ -
except that here only ``equivariant'' functors, where there are
isomorphisms between spaces over conjugate elements of $G$, are
considered.  For such functors, the ``pants'' morphism amounts to
multiplication in the group 2-algebra.
\end{example}

An important special case of a higher cobordism for our
extended TQFT is the one where the objects in $\nCob$ are empty
manifolds $\emptyset$.  Then cobordisms between these are themselves manifolds
\textit{without} boundary, and cobordisms between these have boundary,
but no nontrivial corners.  So we have just a cobordism from one
manifold to another.  It is reasonable to expect that in this case,
the extended TQFT based on a group $G$ should give results equivalent
to those obtained from a TQFT based on the same group, suitably
reinterpreted.

\begin{example}\label{ex:ZSemptyB}
Consider a manifold $S$, thought of as a cobordism $S :
\emptyset \ra \emptyset$.  We expect that finding our $Z_G(S)$ for such a
cobordism should be like finding the vector space assigned to the
manifold $S$ by an ordinary TQFT. 

To see this, first note that $\Pi_1(\emptyset) = \emptyset$, the empty
category, and since this is the initial category, there is a single
functor from it to $G$, hence $\fc{\emptyset} = \catname{1}$, the
category with one object and one morphism.  Thus, $Z(\emptyset) \cong
\V$

Now, since every connection on $S$ ``restricts'' to the unique trivial
connection on $\emptyset$, the 2-linear map takes $\V$ to $\V$, and can
be represented as a $1 \times 1$ matrix of vector spaces.  In other
words, the operators both involve tensoring with a single vector
space.

Too see which vector spaces this is, begin with a 2-vector in
$Z(\emptyset) \cong \V$.  This amounts to a choice of a vector space,
say $V \in \V$.  Pulling back to $S$, we simply get the functor
assigning a copy of $V$ to every object of the groupoid $\fc{S}$.
Isomorphisms from $V$ to $V$ must be assigned to every arrow in this
groupoid.  But there is a unique isomorphism is $\fc{\emptyset}$, namely
the identity---so the pullback to $S$ must assign the identity to
every arrow.

So in fact, taking the pushforward gives a colimit of a diagram which
has a single copy of $V$ for each isomorphism class in $\fc{S}$, which
decomposes as a direct sum of these classes.  This is since the
colimit for just one class is just $V$, and for the whole groupoid is
the direct sum of one copy of $V$ from each isomorphism class.

So we have that:
\begin{equation}
(\pi_2)_{\ast} \circ \pi_1^{\ast}(-) \cong (- \otimes \mathbbm{C}^k)
\end{equation}
Where $k = |\underline{\fc{S}}|$ is the number of connected components
of $\fc{S}$.

If we reinterpret this as assigning $\mathbbm{C}^k$ to $S$, thought of
as a manifold, this does indeed recover the usual formula obtained
from a TQFT.  The TQFT based on the finite group $G$ will assign to a
manifold the Hilbert space of complex-valued functions on the space of
connections (strictly) modulo gauge transformations.  This is
equivalent to what we have just found.
\end{example}

The final element of our weak 2-functor is its effect on 2-morphisms,
so this is the subject of the next section.

\subsection{$Z_G$ on Cobordisms of Cobordisms}\label{sec:ZGonCobCob}

Now we consider the situation of a cobordism between cobordisms.  We
want to describe our extended TQFT as a weak 2-functor, so we want a
bicategory derived from our double bicategory $\nCob$.  By Theorem
\ref{thm:equiv}, this is possible, but we need to see just what a
2-morphism in this corresponding bicategory looks like.  Recall that
the source and target morphisms of the corresponding 2-morphism are
those obtained by composing horizontal and vertical morphisms which
form the edges of the square.

Given a square in $\nCob$, we have a diagram of the form
(\ref{xy:cspan2}).  When we turn this into a 2-cell, the source
morphism will be a cospan in the category of manifolds with corners.
It is found by taking the following pushout:
\begin{equation} \xymatrix{
   & & T_Y \circ S & & \\
   & S\ar^{\iota_S}[ur] & & T_Y \ar_{\iota_{T_Y}}[ul] & \\
   X\ar^{\iota_X}[ur] & & Y\ar_{\iota_Y}[ul]\ar^{\iota_Y}[ur] & & Y'\ar_{\iota_{Y'}}[ul] \\
}
\end{equation}
The pushout square is the central square here, where we get the object
$T_Y \circ S$ equipped with injections $\iota_{S}$ and $\iota_{T_Y}$
which make the square commute, and which is universal in the sense
that any other object with injections from $S$ and $T_Y$ factors
through $T_Y \circ S$.  So in particular, the maps into $M$ can be
factorized as the maps into $T_Y \circ S$ and the canonical injection
$\iota : T_Y \circ S \ra M$.  A similar argument applies to the target
morphism, so the situation we are interested in can be represented as
a cospan of cospans in the following way:
\begin{equation}\label{eq:zcobcob1}
\xymatrix{
  & \Pi_1(M) & \\
\Pi_1(S_1) \ar[ur]^{\iota} & & \Pi_1(S_2) \ar[ul]_{\iota'} \\
\Pi_1(X) \ar[urr]^(.7){\iota_2} \ar[u]^{\iota_1} & & \Pi_1(Y') \ar[u]_{\iota'_2} \ar[ull]_(.7){\iota'_1} \\
}
\end{equation}
with $S_1 = T_Y \circ S$ and $S_2 = S' \circ T_X'$.

Given this situation, which is a 2-morphism for the bicategory of
cobordisms, we want to get a 2-morphism in the bicategory $\iiV$.
That is to say, a natural transformation $\alpha_M$ between a pair of
2-linear maps.  The 2-linear maps in question are those we get by the
construction (\ref{eq:zstep3}).  So in particular,
\begin{equation}\label{eq:zcobcob2}
\xymatrix{
  & \fc{M} \ar[dl]^{\pi} \ar[dr]_{\pi'} & \\
\fc{S_1} \ar[d]^{\pi_1} \ar[drr]^(.3){\pi'_1} & & \fc{S_2} \ar[dll]_(.3){\pi_2} \ar[d]_{\pi'_2}  \\
\fc{X}  & & \fc{Y'}  \\
}
\end{equation}

And finally, quantizing these configuration groupoids by taking
functors into $\V$:

\begin{equation}\label{eq:zcobcob3}
\xymatrix{
  & \Z{M} & \\
\Z{S} \ar[ur]^{\pi^{\ast}} & & \Z{S'} \ar[ul]_{\pi'^{\ast}} \\
\Z{X} \ar[u]^{\pi_1^{\ast}} \ar[urr]^(.7){{\pi'}_1^{\ast}} & & \Z{Y'} \ar[ull]_(.7){\pi_2^{\ast}} \ar[u]^{{\pi'}_2^{\ast}}\\
}
\end{equation}

Now, recall that each of the pullback maps appearing here has an
adjoint, so we have functors $F_1 = ({\pi'}_1)_{\ast} \circ \pi_1^{\ast}$
and $F_2 = ({\pi'}_2)_{\ast} \circ \pi_2^{\ast}$ from $Z(X)=\Z{X}$ to
$Z(Y')=\Z{Y'}$.  A natural transformation will take an object $f \in
Z_G(X)$ and give a morphism $Z_G(M)(f) : F_1(f) \ra F_2(f)$ in $Z(Y')$
satisfying the usual naturality condition.  Now, an object in $Z_G(X)$,
namely a 2-vector, is a $\V$-presheaf on the groupoid of
$G$-connections on $X$ weakly mod gauge transformations.

The hoped-for morphism $Z_G(M)(f)$ in $Z(Y')$ is just a natural
transformation between two such functors $g,g' : \fc{Y'} \ra \V$.
That is, it assigns, for each connection $A \in \fc{Y}$, a linear map
between the two vector spaces: $(Z_G(M)(f))(A) : g(A) \ra g'(A)$.
We want to get $Z_G(M)$ from the cobordism with corners, $M$.
This we define by means of a ``pull-push'' process, similar to the one
used to define the 2-linear maps in the first place.

However, as remarked in Section \ref{sec:kv2vsgrpd}, any natural
transformation between a pair of 2-linear maps between KV 2-vector
spaces can be represented as a matrix of linear operators, as in
(\ref{eq:kvnattransmatrix}).  The matrix in question is indexed by
gauge equivalence classes of connections on $X$ and on $Y$.  Writing
$Z(S)$ in the matrix form means that given a pair $([A],[A'])$ of such
classes, there is a vector space $Z(S)_{[A],[A']}$.  Recall that we
found these vector spaces by the ``pull-push'' process for presheaves
along inclusion maps.

A natural transformation between such functors is a matrix of linear
maps, so we will have 
\begin{equation}
Z_G(M)_{[A],[A']} : Z_G(S)_{[A],[A']} \ra Z_G(S')_{[A],[A']}
\end{equation}
But now we can use the fact that the top level of the tower of spans
of groupoids in (\ref{eq:zcobcob3}) is of the same form as that for
cobordisms between manifolds given in (\ref{eq:zstep3}).  The
component linear maps arise by applying a similar ``pull-push''
process to that used in Section \ref{sec:ZGonCob} to define $Z_G$ on
cobordisms.  

Since there are canonical bases ${[A] \in \mathcal{A}_0(S)}$ and $[A']
\in \mathcal{A}_0(S')$ for the vector spaces $Z(S) \cong
\mathbbm{C}^k$ and $Z(S') \cong \mathbbm{C}^{k'}$, so we can represent
$T$ as a $k \times k'$ matrix.  We then need to describe the effect of
$T$ on a vector in $\mathbbm{C}^k$.  Such a vector amounts to an
assignment of a scalar to each gauge equivalence class of connections
in $\fc{S}$.  In particular, to find the component $T_{[A'],[A]}$
indexed by the class $[A']$ of connections on $S'$, and the class
$[A]$ on $S$, take the vector corresponding to the function equal to
$1$ on $[A]$ and $0$ elsewhere.

The linear map $T$ acts by the ``pull-push'' operation.  The first
stage---pullback gives a function on $\fc{M}$ which is $1$ on any
gauge-equivalence class of connections $[B]$ on $M$ restricting to
$[A]$ on $S$.  Pushing this forward involves taking a sum over all
classes of connection restricting to $[A']$ on $S'$.  Clearly, the
only nonzero contributions are from those connections which restrict
to $[A]$ on $S$.  The action of $T$ extends linearly to all of $V$, so
it is represented by a $k \times k'$ matrix whose entries are indexed
by classes of connections.

So indeed, all discussion of the construction of the natural
transformation will parallel the construction of the 2-linear maps,
but at a lower categorical level, since we get a matrix of scalars
rather than vector spaces---this time in each component $([A],[A'])$.
The resulting linear map (and its matrix representation) can then be
``lifted'' to a natural transformation between 2-linear maps.

A more tricky question is what contribution to expect from those which
do restrict to $[A]$.  Naively, one might expect to simply take a sum
of the function values (all equal to $1$ at the moment) over all such
connections.  Since this ignores the morphisms in $\fc{M}$, one might
perhaps imagine the sum should be over only equivalence classes of
connections.  However, one should suspect that this is also incorrect,
since when we found a pushforward for $\V$-presheaves, we took not a
direct sum over equivalence classes, but a colimit.

In fact, the correct prescription involves the \textit{groupoid
cardinality} of the groupoid of those connections which contribute to
the sum.  This concept is described by Baez and Dolan \cite{finfeyn},
and related to Leinster's \cite{leinstereuler} concept of the Euler
characteristic of a category.  For a more in-depth discussion of
groupoid cardinality, and also of its role (closely related to the
role it plays here) in a simple model in quantum mechanics, see the
author's paper \cite{morton} on the categorified harmonic oscillator.

The cardinality of a groupoid $\catname{G}$ is:
\begin{equation}
 |\catname{G}| = \sum_{[ x ] \in \underline{\catname{G}}} \frac{1}{| \opname{Aut} ( x ) |}
\end{equation} the sum is over isomorphism classes in $\catname{G}$.
This quantity is invariant under equivalence of categories, and should
be the pushforward of the constant function $1$.  So we define:

\begin{definition}\label{def:ZGonM}Given cobordism between cobordisms,
$M : S \ra S'$, for $S,S' : B \ra B'$, then
\begin{equation}
Z_G(M) : Z_G(S) \ra Z_G(S')
\end{equation} is a natural transformation given by a matrix of linear
operators:
\begin{equation}
Z_G(M)_{j,k} : Z_G(S)_{j,k} \ra Z_G(S')_{j,k}
\end{equation}
where the vector space $Z_G(S)_{j,k}$ is the $(j,k)$ component of the
matrix for the 2-linear map $Z(S)$.  This is indexed by choices
$(j,k)$, where $j$ identifies an equivalence class $[A]$ of
connections on $B$.  

The linear map $Z_G(M)_{j,k} = T$ is described by the matrix:
\begin{equation}\label{eq:ZGonM}
T_{[A],[A']} = |(j \times j')^{-1}(A,A')|
\end{equation}
the groupoid cardinality of the \textit{essential preimage} of
$(A,A')$, where $A$ is a connection on $S$ and $A'$ a connection on $S'$.
\end{definition}
(That is, of the groupoid of all connections on $M$ simultaneously
restricting to a connection gauge equivalent to $A$ on $S$ and $A'$ on
$S'$.)

Since this is a matrix of linear transformations between the correct
vector spaces, it defines a natural transformation.  This is the last
element of the extended TQFT $Z_G$ which needs to be defined---Theorem
\ref{thm:XTQFTmainthm} will show that its behaviour on manifolds,
cobordisms, and cobordisms between cobordisms satisfy the axioms of a
weak 2-functor.  Two parts of this we prove here separately.  The
first is strict preservation of vertical composition; the second is
preservation of horizontal composition as strictly as possible
(i.e. up to the isomorphisms $\beta$ which make comparison possible -
as we will see).

\begin{theorem}The assignment $Z_G(M)$ to cobordisms with corners
given by (\ref{eq:ZGonM}) preserves vertical composition strictly:
$Z_G(M'M) = Z_G(M') \circ Z_G(M)$.
\end{theorem}
\begin{proof}
Vertical composition is just component-wise composition of linear
operators.  So it suffices to show that given any component,
composition is preserved.  That is, given a vertical composite of two
cobordisms between cobordisms:
\begin{equation}
\xymatrix{
B \ar@/^3pc/[rrr]^{S_1}="1" \ar[rrr]^{S_2}="2" \ar@/_3pc/[rrr]_{S_3}="3" & & & B' \\
 {\ar@{=>}"1"+<0ex,-2.5ex> ;"2"+<0ex,+2.5ex>^{M}}
 {\ar@{=>}"2"+<0ex,-2.5ex> ;"3"+<0ex,+2.5ex>^{M'}}
}
\end{equation}
we get matrices $Z(S_1)_{(j,k)}$, $Z(S_2)_{(j,k)}$, and
$Z(S_3)_{j,k}$, of vector spaces indexed by
connections-and-representations on $B$ and $B'$ as in Definition
\ref{def:ZGonM}.

For the following, fix a component---i.e. a gauge equivalence class of
connections $[A]$ on $B$ and representation of $\opname{Aut}([A])$,
and similarly for $B'$.

Then we have two linear operators.  The first is
\begin{equation}
Z_G(M)_{j,k}=T :  Z(S_1)_{(j,k)} \ra Z(S_2)_{(j,k)}
\end{equation}
and is given as a matrix, indexed by equivalence classes of
connections $[A_1]$ on $S_1$ and $[A_2]$ on $S_2$, as follows.  The
component $T_{[A_1],[A_2]}$ is the groupoid cardinality of the
groupoid of all connections on $M$ which are gauge equivalent to ones
restricting to both $A_1$ and $A_2$---that is, the \textit{essential
preimage} of $(A_1,A_2)$. Denote this by $|\widehat{(A_1,A_2)}|$.

The second operator
\begin{equation}
Z_G(M')_{j,k}=T' :  Z(S_2)_{(j,k)} \ra Z(S_3)_{(j,k)}
\end{equation}
is likewise a matrix, indexed by equivalence classes of connections
$[A_2]$ on $S_2$ and $[A_3]$ on $S_3$, where $T'_{[A_2],[A_3]} =
|\widehat{(A_2,A_3)}|$, the groupoid cardinality of the essential preimage
of $(A_2,A_3)$ (a groupoid of connections on $M'$).

The product of these is then just given by matrix multplication, so
that
\begin{equation}
(T'T)_{[A_1],[A_3]} = \sum_{[A_2]} |\widehat{(A_1,A_2)}| \times |\widehat{(A_2,A_3)}|
\end{equation}
That is, to get the component of the image of a delta functon on
$[A_1]$ in the connection $[A_2]$, one takes a sum over equivalence
classes of connections $[A_2]$ on $B_2$.  The sum is of of the
products of the groupoid cardinalities of connections on $M$ and $M'$
restricting to this $[A_2]$.

We need to show this is the same as the linear operator obtained from
the same $(j,k)$ component for the 2-morphism $Z_G(M'M)$.  But we know that
\begin{equation}
Z_G(M'M)_{(j,k)} = R :  Z(S_1)_{(j,k)} \ra Z(S_3)_{(j,k)}
\end{equation}
has component
\begin{equation}
|\widehat{(A_1,A_3)}|
\end{equation}
the groupoid cardinality of the essential preimage of $(A_1,A_3)$,
which is a groupoid of connections on $M'M$.  So we really just need
the fact that groupoid cardinalities behaves well with respect to sum
and product.

In particular, $\widehat{(A_1,A_3)}$ is a groupoid of connections on $M'M$,
but each of these has a restriction to $S_2$, and if two connections
on $M'M$ have gauge-inequivalent restrictions to $S_2$, they must be
gauge inequivalent.  So $\widehat{(A_1,A_3)}$ is a direct sum over the
possible gauge-equivalence classes of restrictions $[A_2]$ to $S_2$.
Since the groupoid cardinality of a direct sum of groupoids is the sum
of their cardinalities, we thus have
\begin{equation}
|\widehat{(A_1,A_2)}| = \sum_{[A_2]} |\widehat{(A_1,A_2,A_3)}|
\end{equation}
where $\widehat{(A_1,A_2,A_3)}$ is the groupoid of connections on $M'M$
which restrict to all the $A_i$ simultaneously.

However, we claim this is just the cartesian product of groupoids.
This is since $M'M$ is an equivalence class of manifolds with corners,
where a standard representative for $M'M$ is a representative for $M'$
and for $M$, identified at the images of the common inclusions of
$S_2$.  By a generalization of the Meyer-Vietoris theorem (see, for
instance, Brown \cite{brown}) we have $\Pi_1(M'M)$ likewise is a
disjoint union of $\Pi_1(M')$ and $\Pi_1(M)$, modulo the equivalence
of the images of $\Pi_1(S_2)$.  But then, taking functors into $G$, we
have $\fc{M'M}$ is a subgroupoid of the product $\fc{M'} \times
\fc{M}$, containing only the objects (connections) such that the
connections in the two components agree on $S_2$.  Since we have fixed
a particular connection $A_2$ on $S_2$, we just get the cartesian
product of groupoids of connections on $M'$ and $M$ respectively which
restrict to $A_2$.

Now, since the groupoid cardinality of a cartesian product of groupoids
is the product of their groupoid cardinalities, we have
\begin{equation}
R_{[A_1],[A_3]} = (T'T)_{[A_1],[A_3]}
\end{equation}
so $Z_G$ preserves vertical composition of 2-morphisms strictly.
\end{proof}

A similar result holds for vertical composition.


\begin{theorem}The assignment $Z_G(M)$ to cobordisms with corners
given by (\ref{eq:ZGonM}) preserves horizontal composition strictly,
up to the isomorphism weakly preserving composition of the source and
target morphisms:
\begin{equation}
\xymatrix{
Z_G(B) \ar@/^3pc/[rrrr]^{Z_G({S'}_1 \circ S_1)}="a" \ar@{}[rrrr]^{}="c" \ar@/_3pc/[rrrr]_{Z_G({S'}_2 \circ S_2)}="b" \ar@/^1pc/[rr]^{Z_G(S_1)}="1" \ar@/_1pc/[rr]_{Z_G({S'}_1)}="2"  & & Z_G(B') \ar@/^1pc/[rr]^{Z_G(S_2)}="3" \ar@/_1pc/[rr]_{Z_G({S'}_2)}="4" & & Z_G(B'') \\
 {\ar@{=>}"1"+<0ex,-2.5ex> ;"2"+<0ex,+2.5ex>^{M}}
 {\ar@{=>}"3"+<0ex,-2.5ex> ;"4"+<0ex,+2.5ex>^{M'}}
 {\ar@{=>}"a"+<0ex,-2.5ex> ;"c"+<0ex,+2.5ex>^{\beta_{S_1,{S'}_1}}}
 {\ar@{=>}"c"+<0ex,-2.5ex> ;"b"+<0ex,+2.5ex>^{\beta_{S_2,{S'}_2}^{-1}}}
}
=
\xymatrix{
Z_G(B) \ar@/^2pc/[rrrr]^{Z_G({S'}_1 \circ S_1)}="1" \ar@/_2pc/[rrrr]_{Z_G({S'}_2 \circ S_2)}="2" & & & & Z_G(B'') \\
 {\ar@{=>}"1"+<0ex,-2.5ex> ;"2"+<0ex,+2.5ex>^{Z_G(M' \otimes_H M)}}
}
\end{equation}
\end{theorem}
\begin{proof}
The horizontal composition involves ``matrix multiplication'' at the level of composition of 2-linear maps.  Given a
horizontal composite
\begin{equation}
\xymatrix{
B \ar@/^2pc/[rr]^{S_1}="1" \ar@/_2pc/[rr]_{S_2}="2"  & & B' \ar@/^2pc/[rr]^{{S'}_1}="3" \ar@/_2pc/[rr]_{{S'}_2}="4" & & B'' \\
 {\ar@{=>}"1"+<0ex,-2.5ex> ;"2"+<0ex,+2.5ex>^{M}}
 {\ar@{=>}"3"+<0ex,-2.5ex> ;"4"+<0ex,+2.5ex>^{M'}}
}
\end{equation}
the functor $Z_G$ assigns 2-linear maps to the cobordisms $S_1$,
$S_2$, ${S'}_1$, and ${S'}_2$, and natural transformations to $M$ and
$M'$.  Then the horizontal composite is a natural transformation
\begin{equation}
Z_G(M' \otimes_H M) : Z_G({S'}_1) \circ Z_G(S_1) \ra Z_G({S'}_2) \circ Z_G(S_2)
\end{equation}

As discussed in Remark \ref{rk:composbeta}, the isomorphisms $\beta$
allow comparison of the horizontal composite of natural
transformations $Z_G(M') \otimes Z_G(M)$ with the natural
transformation $Z_G(M' \otimes_H M)$.  The presence of the $\beta$
isomorphisms only allows us to disregard the distinction between
$Z_G(S_2 \circ S_1)$ and $Z_G(S_2) \circ Z_G(S_1)$ (and likewise for
the $S'$).

So first consider $Z_G(M') \otimes Z_G(M)$, the horizontal composite
of the natural transformations in $\iiV$ corresponding to the
cobordisms with corners.  Each of these natural transformations can be
represented as a matrix of linear maps:
\begin{equation}
Z_G(M)_{[A_1],[A_2]} : V_{[A_1],[A_2]} \ra W_{[A_1],[A_2]}
\end{equation}
where the $V$'s are the coefficients of $Z_G(S_1)$ and $W$'s are those
of $Z_G(S_2)$.  The coefficients of $Z_G(M')$ are similarly
\begin{equation}
Z_G(M')_{[A_2],[A_3]} : {V'}_{[A_2],[A_3]} \ra {W'}_{[A_2],[A_3]}
\end{equation}
Then the horizontal product $Z_G(M') \otimes Z_G(M)$ will be given
by the matrix of linear maps:
\begin{equation}
\bigoplus_{[A_2]} \bigl{(}  Z_G(M)_{[A_1],[A_2]} \otimes Z_G(M')_{[A_2],[A_3]} \bigr{)} : \\
   \bigoplus_{[A_2]} \bigl{(}  V_{[A_1],[A_2]} \otimes {V'}_{[A_2],[A_3]} \bigr{)}
    \ra 
    \bigoplus_{[A_2]} \bigl{(} W_{[A_1],[A_2]} \otimes {W'}_{[A_2],[A_3]} \bigr{)}
\end{equation}

The $([A_1],[A_3])$ component of this product is a linear map given as
a block matrix, with one block for each gauge equivalence class of
connections $[A_2]$ on $B_2$, and whose blocks consist of the tensor
product of the matrices from the components of $Z_G(M)$ and $Z_G(M')$.
So suppose these are, respectively, $n \times m$ and $n' \times m'$
dimensional matrices.  Then the result is an $(n\times n') \times (m
\times m')$ matrix, where the $((i,i'),(j,j'))$ component is the
product of the $(i,j)$ component of $Z_G(M)$ and the $(i',j')$
component of $Z_G(M')$.

Recall that these indexes mark connections on the cobordisms: the
$(i,j)$ component of $Z_G(M)$ is the groupoid cardinality of the
groupoid of connections on $M$ which match the $i^{th}$ on $S_1$ and
the $j^{th}$ on $S_2$; and the $(i',j')$ component of $Z_G(M')$ is the
groupoid cardinality of the groupoid of connections on $M'$ which
match the ${i'}^{th}$ on ${S'}_1$ and the ${j'}^{th}$ on ${S'}_2$.
But this is the groupoid cardinality of the product groupoid whose
objects are just these pairs, since groupoid cardinality respects
products.

Next consider $Z_G(M' \otimes_H M)$, the natural transformation in
$\iiV$ corresponding to the horizontal composite of the cobordisms
with corners.  Again, this can be represented as a matrix of linear
maps indexed by pairs $([A_1],[A_3])$ just as above:
\begin{equation}
Z_G(M' \otimes_H M)_{[A_1],[A_3]}: U_{[A_1],[A_3]} \ra X_{[A_1],[A_3]}
\end{equation}
where the $U$ have a basis of equivalence classes connections on
${S'}_1 \circ S_1$, and the $X$ on ${S'}_2 \circ S_2$, which restrict
to $[A_1]$ and $[A_3]$.

But on the other hand, using the $\beta$ isomorphisms to identify the
source and targets 
allows us to compare this directly to the other side.

But the groupoid of connections on ${S'}_1 \circ S_1$ has the
restriction maps $p_S$ and $p_{S'}$ to give connections on $S$ and
$S'$.  Moreover, the connections obtained this way agree up to gauge
equivalence on $B_2$ (since composition of cobordisms is given by a
weak pushout).  The gauge equivalence up to which these agree is given
by the natural isomorphism $\alpha$ from the weak pullback of
connection groupoids.  So the components $U_{[A_1],[A_3]}$ decompose
as a direct sum over $[A_2]$ on $B_2$ of pairs of connections, one on
$S_1$, and one on ${S'}_1$, each of which restricts to $[A_2]$ and
either $[A_1]$ or $[A_3]$.  Similarly for the vector spaces
$X_{[A_1],[A_3]}$.

Now, the groupoid of all connections on $M' \otimes_H M$ is a fibred
product over $\fc{B_2}$, since each such connection restricts to just
one gauge equivalence class on $B_2$.  Then for each such $[A_2]$, the
groupoid of connections decomposes as a product over choices of
restrictions to the $S$ on each side.  So it is just a product of the
groupoids of connections on $M'$ and $M$, separately, which restrict
$[A_2]$.  Restrictions to ${S'}_1 \circ S_1$ and ${S'}_2 \circ S_2$
each give separate restrictions to the two halves. Then the
cardinality of this groupoid in any component (i.e. with any
particular restrictions to source and target) is just the product of
the groupoid cardinalities for the corresponding restrictions on $M'$ and $M'$.

But this is exactly what we found for $Z_G(M') \otimes Z_G(M)$.  So
the two sides are equal as required.
\end{proof}

Again, a special instance of an extended TQFT is when it ``restricts''
to a TQFT.

\begin{example}\label{ex:ZMemptyB} Returning to the example of
cobordisms between empty manifolds, suppose we have two such
cobordisms $S$ and $S'$, and a cobordism with (trivial!) corners $M :
S \ra S'$.  In fact, the effect should be similar to that of
evaluating a TQFT on $M$ thought of as a cobordism between manifolds,
in a precisely analogous way that $Z_G(S)$ can be thought of as a TQFT
giving a vector space for the manifold $S$.

In particular, we have, by the argument in Example \ref{ex:ZSemptyB}, that:
\begin{equation}
Z(S) \cong (- \otimes \mathbbm{C}^k)
\end{equation}
and
\begin{equation}
Z(S') \cong (- \otimes \mathbbm{C}^{k'})
\end{equation}
where $k$ and $k'$ are the number of isomorphism classes of
connections on $S$ and $S'$ respectively.  If we think of these as
being vector spaces $\mathbbm{C}^k$ and $\mathbbm{C}^{k'}$ assigned by
a TQFT, then a cobordism should assign a linear map $T : \mathbbm{C}^k
\ra \mathbbm{C}^{k'}$.  Indeed, such a linear map will give rise to a
natural transformation from $Z(S)$ to $Z(S')$ by giving, for any
objects $V \in \V$ on the left side of the diagram, the map $1_V
\otimes T$ on the right side.  Moreover, all such natural
transformations arise this way.

Now, the diagram from (\ref{eq:zcobcob3}) gives rise to a 2-morphism
in $\Cat$:
\begin{equation}
\xymatrix{
  \V \ar@/^2pc/[rr]^{(\pi_2)_{\ast} \circ \pi_1^{\ast}}="0" \ar@/_2pc/[rr]_{(\pi'_2)_{\ast} \circ (\pi'_1)^{\ast}}="1" & & \V \\
   {\ar@{=>}"0"+<0ex,-2.5ex> ;"1"+<0ex,+2.5ex>^{Z(M)}}
}
\end{equation}

Here, $Z(M)$ arises from the 2-linear map
\begin{equation}
\pi'_{\ast} \circ \pi^{\ast} : \Z{S} \ra \Z{S'}
\end{equation}
as described in Definition \ref{def:ZGonM}.
\end{example}

Having now described the effect of the extended TQFT at each level -
manifolds, cobordisms, and cobordisms with corners---it remains to
check that these really define a 2-functor of the right kind.  This is
the task of Section \ref{sec:maintheorem}.

\subsection{Main Theorem}\label{sec:maintheorem}

Now let us recap the discussion so far.  For each finite group $G$, we
want to get a weak 2-functor from the bicategory associated to the
double bicategory of cobordisms with corners into 2-vector spaces,
$Z_G: \nCob \ra \iiV$.  This has three aspects, for which we then must
verify some properties.

To a compact $(n-2)$-manifold, $Z_G$ assigns a 2-vector space.  This consists of
$\V$-presheaves on the groupoid of $G$-connections on $B$ weakly
modulo gauge transformations.

To a cobordism between $(n-2)$-manifolds, $S: B \ra B'$ in $\nCob$,
$Z_G$ assigns a span of the groupoids of $G$-connections, as in
(\ref{eq:zstep2}).  Then a $\V$-presheaf $F$ on $\fc{B}$ can be
transported along the span by first pulling back onto $\fc{S}$ along
the restriction $\pi$ of connections on $S$ to connections on $B$.  We
then push forward this pullback $\pi^{\ast} F$ along the restriction
$\pi'$ of $\fc{S}$ to $\fc{B'}$ to give a $\V$-presheaf $\pi'_{\ast}
\circ \pi^{\ast} F$ on $\fc{B'}$.

To a cobordism between cobordisms, $Z_G$ assigns a natural
transformation in a similar fashion.  Given two functors corresponding
to cobordisms, as above, if there is a cobordism between them, it
defines a way to push forward a vector in any of the component vector
spaces of the functor, written as a matrix.  This is done by pulling
back the function on the basis defined by the vector, and then pushing
forward using a weight given by the groupoid cardinality.

This construction is to give a weak 2-functor.  This must be equipped
with natural isomorphisms $\beta_{S,S'} : Z_G(S' \circ S) \ra Z_G(S')
\circ Z_G(S)$ giving weak preservation of composition, as described in
Theorem \ref{thm:zcomposfunc}.  It also must have a natural
transformation $U_B : 1_{Z_G(B)} \tilde{\ra} Z_G(1_B)$ giving weak
preservation of units.  Note that for any $(n-2)$-manifold $B$, the
idenity $1_B$ is a cobordism $I \times B$, which has the manifold $B$
embedded as $\{ ( 0, b) | b \in B \}$ and $\{ ( 0, b) | b \in B \}$
(and this cobordism is exactly the collar on both source and target).
Then we note that there is an equivalence of categories between
$\fc{B}$ and $\fc{1_B}$ since $B$ and $1_B$ have the same homotopy
type.  So $Z_G(1_B)$, which uses a ``pull-push'' through the groupoid
of connections on $I \times B$, is equivalent to the identity on
$1_{Z_G(B)}$.

\begin{definition}\label{def:ZG}
Given a finite group $G$, the extended TQFT $Z_G$ is a 2-functor defined as follows:
\begin{itemize}
\item For a closed compact manifold $B$, the weak 2-functor assigns a
      2-vector space:
\begin{equation}\label{eq:ZGonB}
  Z_G(B) = \Z{B}
\end{equation}
\item For a cobordism between manifolds:
\begin{equation}
B \ralim^{i} S \lalim^{i'} B'
\end{equation}
the weak 2-functor assigns a 2-linear map:
\begin{equation}\label{eq:ZGonS}
  Z_G(S) = (p')_{\ast} \circ p^{\ast}
\end{equation} where $p$ and $p'$ are the associated projections for
the underlying groupoids of connections weakly modulo gauge
transformations.
\item For a cobordism with corners between two cobordisms with the
      same source and target:
\begin{equation}
\xymatrix{
 & S_1 \ar[d]_{i} & \\
B \ar[ur]^{i_1} \ar[dr]_{i_2} & M & B' \ar[ul]_{i'_1} \ar[dl]^{i'_2} \\
 & S_2 \ar[u]^{i'} & \\
}
\end{equation}
the weak 2-functor assigns a natural transformation $\alpha$, whose
components (in the matrix representation) are as in (\ref{eq:ZGonM}).
\end{itemize}
The 2-functor $Z_G$ also includes the following:
\begin{itemize}
\item For each composable pair of cobordisms $S : B_1 \ra B_2$ and $S'
      : B_2 \ra B_3$, a natural isomorphism
\begin{equation}
 \beta : Z_G ( S' \circ S ) \ra Z_G(S') \ra Z_G(S)
\end{equation}, as described in Theorem \ref{thm:zcomposfunc}.
\item For each object $B \in \nCob$, the natural transformation
\begin{equation} U_B : 1_{Z_G(B)} \tilde{\ra} Z_G(1_B) \end{equation}
is the natural transformation induced by the equivalence between
$\fc{B}$ and $\fc{1_B}$.  \end{itemize} \end{definition}

Then we have the following:

\begin{theorem}\label{thm:XTQFTmainthm} For any finite group $G$,
there is a weak 2-functor $Z_G : \nCob \ra \iiV$ given by the
construction in Definition \ref{def:ZG}.
\end{theorem}
\begin{proof}
First, we note that by the result of Theorem \ref{thm:ZB2VS}, we know
that $Z_G$ assigns a 2-vector space to each object of $\nCob$.

If $S : B \ra B'$ is a cobordism between compact manifolds---i.e. a
morphism in $\nCob$, the map $Z_G(S)$ defined in Definition
\ref{def:ZGonS} is a linear functor by the result of Theorem
\ref{thm:2mapadjoints}, since it is a composite of two linear maps.
This respects composition, as shown in Theorem \ref{thm:zcomposfunc}.

Next we need to check that our $Z_G$ satisfies the properties of a
weak 2-functor: that the isomorphisms from the weak preservation of
composition and units satisfy the requisite coherence conditions; and
that $Z_G$ strictly preserves horizontal and vertical composition of
natural transformations.

The coherence conditions for the compositor morphisms
\begin{equation}
\beta_{S,T} :  Z_G(T) \circ Z_G(S) \ra Z_G(T \circ S)
\end{equation}
and the associator say that these must make the following diagram
commute for all composable triples of cobordisms $(S_1,S_2,S_3)$:
\begin{equation}\label{eq:wk2fncassoc}
\xy
 (0,20)*+{Z_G(S_3) \circ Z_G(S_2) \circ Z_G(S_1)}="top";
 (35,4)*+{Z_G(S_3 \circ S_2) \circ Z_G(S_1)}="rttop";
 (35,-20)*+{Z_G( (S_3 \circ S_2) \circ S_1 )}="rtbot";
 (-35,-20)*+{Z_G(S_3 \circ (S_2 \circ S_1) )}="leftbot";
 (-35,4)*+{Z_G(S_3) \circ Z_G(S_2 \circ S_1)}="leftop";
     {\ar^{1 \otimes \beta_{2,1}} "leftop";"top"}
     {\ar_{\beta_{3,2} \otimes 1} "rttop";"top"}
     {\ar^{\beta_{3,21}} "leftbot";"leftop"}
     {\ar_{\beta_{32,1}} "rtbot";"rttop"}
     {\ar^{Z_G(\alpha_{3,2,1})} "leftbot";"rtbot"}
\endxy
\end{equation}

We implicitly assume here a trivial associator for the 2-linear maps
in the expression $Z_G(S_3) \circ Z_G(S_2) \circ Z_G(S_1)$.  This is
because each 2-linear map is just a composite of functors, so this
composition is associative.  But note that we can similarly assume,
without loss of generality, that the associator $\alpha$ for
composition of cobordisms is trivial.  The composite $S_2 \circ S_1$
is a pushout of two spans of manifolds with boundary:
\begin{equation}
  \xymatrix{
      &   & S_2 \circ S_1  &  & \\
      & S_1  \ar[ur]_{I_1}  &   & S_2 \ar[ul]^{I_2} & \\
    B_1 \ar[ur]_{i_1} &  & B_2 \ar[ul]^{i_2} \ar[ur]_{i'_1} &   & B_3 \ar[ul]^{i'_2} \\
  }
\end{equation} This pushout is only defined up to diffeomorphism, but
one candidate is $S_1 \coprod S_2 / \sim$, where $i_1(x) \sim i_2(x)$
for any $x \in B_2$.  Any other condidate is diffeomorphic to this
one.  But then, the associator
\begin{equation}
\alpha_{3,2,1} : Z_G( S_3 \circ (S_2 \circ S_1) ) \ra Z_G( (S_3 \circ S_2) \circ S_1 )
\end{equation}
is just the identity.  (Choosing different candidates for the pushouts
involved would give a non-identity associator).

So it suffices to show that, with this identification,
\begin{equation}
(1 \otimes \beta_{2,1}) \circ \beta_{3,21} = (\beta_{3,2} \otimes 1) =\circ \beta{32,1}
\end{equation}
This is verified by a computation we leave to the reader.



In general, the coherence conditions for the ``unit'' isomorphism
\begin{equation}
U_B : 1_{Z_G(B)} \tilde{\ra} Z_G(1_B)
\end{equation}
which weakly preserves identities say that it must make the following
commute for any cobordism $S: B \ra B'$:
\begin{equation}\label{eq:wk2fncunit}
\xymatrix{
Z_G(S)  & \\
Z_G(S) \circ Z_G(1_B) \ar[u]_{1 \otimes U_B}  & Z_G(S \circ 1_B)\ar[ul]_{Z_G(r_S)} \ar[l]_{\beta_{S,1_B}}
}
\end{equation}
where $r_B$ is the right unitor for $B$.  There is also the symmetric
condition for the left unitor.

We notice that, as with $Z_G(1_B)$, $Z_G(r_B)$ is equivalent to the
identity since we can think of the unitor $r_S : S \circ 1_B \ra S$ as
a mapping cylinder diffeomorphic to $I \times S$.  Since $S \circ 1_B$
and $S$ are diffeomorphic, these are embedded as the ends of the
cylinder.

So the condition amounts to the fact that $\beta_{S,1_B} : Z_G(S \circ
1_B) \ra Z_G(S) \circ Z_G(1_B) = Z_G(S) $ is equivalent to the
identity in such a way that (\ref{eq:wk2fncunit}) commutes.  We again
leave this to the reader.
\end{proof}

This weak 2-functor is our extended TQFT.

\section{Extended TQFT's and Quantum Gravity}\label{chap:QG}

The title of this paper is ``Extended TQFT's and Quantum Gravity'',
but so far we have said much about the former and nearly nothing about
the latter.  Yet, despite the intrinsic interest extneded TQFT's in
themselves, the prospect of applying these results, or very similar
ones, to quantum gravity has been one of the major motivations behind
this work.  The prospects for doing this are good, at least in a
low-dimensional toy model.  In (2+1) dimensions (two dimensions of
space, and one of time), or 3 dimensions (with no time dimension),
Einsteinian gravity is a topological theory, whereas in higher
dimensions it is not.

So more specifically, the immediate result of extending our results
here will be not, in general, quantum gravity, but a topological gauge
theory called BF theory.  The connection to gravity is that this is
the same as Einsteinian gravity in 3 dimensions, and in 4-dimensions
it is a limit of Einsteinian gravity as $G \ra 0$ (where $G$ is
Newton's constant).  This is a limitation of our approach, but quantum
gravity is a large and mostly open field (see for instance Rovelli's
survey \cite{rovelli} of some of the work to date); so finding a clear
framework for certain, fairly simple, cases is a useful project.

In this final chapter, we sketch what kind of extension is needed, and
the implications of this work for quantum gravity in the case.  This
chapter is not intended to be mathematically rigorous.  Its role is to
describe in an impressionistic way some of the mathematical and
physical context for this work, as well as to suggest the directions
for its future development.

\subsection{Extension to Lie Groups}\label{sec:liegroups}

The first thing to consider is the possibility of extending the
analysis we have made for extended TQFT's corresponding to finite
groups.  In particular, we are interested in an analog of the
preceding when $G$ is a Lie group.  In particular, there is a special
case of interest, which is when $G = SU(2)$, and $n=3$: that is,
considering $Z_{SU(2)} : \catname{3Cob_2} \ra \iiV$.  We will describe
here how such a theory, if it is possible to construct it, would be
related to a well-studied theory of quantum gravity in three
dimensions: the Ponzano-Regge model.

The theorems so far apply only when $G$ is a finite group.  However,
we have seen in Section \ref{sec:CY} that there is a notion of an
infinite-dimensional 2-vector space $2L^2{X}$ for a measure space
$(X,\mu)$, consisting of maps from $X$ into $\V$.  This is an infinite
dimensional analog of the functor category $[ \X , \V ]$ which was used in
constructing an extended TQFT from a finite group (though we must
restrict to only ``measurable'' functors).  In particular, it should
still make sense to define a 2-vector space $\Z{B}$ for a manifold
$B$.  This involves both a generalization and a specialization from
the Crane-Yetter 2-vector space $\catname{Meas(X)}$, since in that
case $X$ was a measurable space, wheareas in the case of a Lie group
it comes equipped with a standard measure (Haar measure), but we also
consider its path groupoid, rather than merely the set.  So one would
need to extend the theory of categories of measurable fields of
Hilbert spaces to a theory of categories of measurable functors into
$\V$ from such a measurable groupoid.

Now, the construction used for a finite group used several facts we
showed for finite groupoids.  For example, Theorem
\ref{thm:2mapadjoints} established that the 2-linear map given by
pushforward is the adjoint of that given by pullback.  However, we
only showed this for finite groupoids.  In general, if $G$ is not
finite, $\fc{B}$ is not an essentially finite groupoid.  So this and
other theorems would need to be extended to the case of Lie groups.
In particular, since 2-vector spaces need not contain arbitrary
infinite colimits, the pushfoward we described may not exist.  So we
need the infiinite-dimensional 2-vector spaces in Crane and Yetter's
$\catname{Meas}$, as discussed in Section \ref{sec:CY}.

So in particular, such an extension should take advantage of the Haar
measure on $G$ to define the pushforward of a functor on a space by
direct integration, rather than by simply taking a general colimit
(which need not exist).  This and other such constructions would need
to be justified in order to try to imagine constructing an extended
TQFT from a Lie group as we have described with a finite group.  It
seems most clear how this would work in the case where $G$ is compact,
since compact Lie groups have finite total Haar measure.  If the total
measure of the group were infinite, we would not expect the integrals
one would use in these definitions to converge, and there would be a
problem of well-definedness.

Then in cases where the direct integral exists, we would expect, by
analogy with the formula from Definition \ref{def:ZGonS} that the
component in some connection $A'$ on $B'$ of $Z(S)$ applied to a
``state'' 2-vector $\Psi \in Z(B)$ is:
\begin{equation}\label{eq:Zofcob}
(Z(S)\Psi)(A) = \int_{\fc{B}}^{\oplus} \Bigl{(}\int_{\overline{\fc{S}}}^{\oplus} \Psi(A)  \mathd \overline{A} \Bigr{)} \mathd A'
\end{equation}
Where $\overline{\fc{S}}$ is the set of connections $\overline{A}$ on
$S$ such that $\overline{A}|_B = A'$ and $\overline{A}|_B' = A'$.
Both integrals are ``direct integrals'' of Hilbert spaces.  The outer
integral, over $B$, uses the principle that $\Psi$ can be represented
as a direct integral (though not a finite linear combination) of
simple objects in $Z_G(B)$.  The direct integral over connections on
$S$ stands in for a general colimit.  This assumes that we can treat
the ``pushforward'' phase of $Z_G(S)$ as a direct integral (rather
than a direct sum) of quotient spaces.

Here we are integrating with respect to a measure on the space of
connections.  Since this consists of functors from a finitely
generated groupoid into $G$, the measure is derived from the Haar
measure on $G$.

Presuming that this is justified, it should be possible to extend the
main results (somewhat modified) from this discussion of extended
TQFT's to the case where $G$ is any compact Lie group (and possibly
any Lie group).  The groups of major interest to quantum gravity are
rotation groups of various signatures, and their double covers (which
are used in describing \textit{spin} connections) .  For example,
connections valued in Euclidean rotation groups $SO(3)$ and $SO(4)$,
and their double covers $SU(2)$ and $SU(2) \times SU(2)$, are relevant
to 3- and 4-dimensional Euclidean quantum gravity respectively.

More precisely, since what we have discussed are \textit{flat}
connections, this remark needs to be qualified.  Flat $SU(2)$
connections do indeed describe configurations for 3D quantum gravity,
since in that case, gravity is a purely topological theory.  (For more
background on 3D quantum gravity, particularly in the case of
signature $(2,1)$, see work by Steven Carlip \cite{carlip},
\cite{carlipsixways}).

However, in 4 dimensions, a theory of flat connections does not
describe gravity, but rather a limiting case of Einsteinian gravity as
Newton's constant $G \ra 0$.  The subject of this limit, and in
general the deformation of gauge theories, is considered extensively
by Wise \cite{wise}.  What is true in 4 dimensions is that the
purely topological theory corresponds to a theory of flat connections
on a manifold known as $BF$ theory.   and by Freidel, Krasnov and Puzio
\cite{BFgrav}).  To describe a theory of gravity
would need something more than what is discussed here.  In Section
\ref{sec:QGwards} we briefly consider some possible approaches to this
problem.

\subsection{Ponzano-Regge with Matter}\label{sec:QG3D}

If $G=SU(2)$, the objects of $\mathcal{A}_0/\!\!/G$, just as for a
finite group, are equivariant functors from $[\Pi_1,SU(2)]$ to $\V$,
and can be represented in terms of a basis of irreducible objects.
Assuming that the previous results hold when $G$ is a Lie group, an
irreducible object amounts to a choice of conjugacy class in $SU(2)$
and action of $SU(2)$ on the associated vector spaces coming from the
isomorphism assocated to conjugation by $g$.
Let us assume that when we replace finite $G$ by the Lie group
$SU(2)$, we retain the classification of Example \ref{ex:ZonS1}.  Then
irreducible 2-vector by pairs $([g],\rho)$ of a conjugacy class $[g]
\in SU(2)/\opname{Ad}(SU(2))$, and representation $\rho$ of $SU(2)$ on
some vector space $V$.  Now, a conjugacy class in $SU(2)$ amounts to
specifying an \textit{angle of rotation} in $[0,4\pi]$.  This is since
this is the double cover of the 3D rotation group, and all rotations
by the same angle are conjugate to all others by some rotation taking
one axis of rotation to the other.  This number in $[0,4\pi]$ represents a
\textit{mass} in 3D quantum gravity---which manifests as an angle
deficit when one traces a path around a massive particle, one finds,
geometrically, that one has rotated by a certain angle proportional to
its mass, which has a maximum total mass allowable of $4\pi$ in
Euclidean 3D gravity.

On the other hand, a representation of $SU(2)$ is classified by a
half-integer, which is called a \textit{spin} since these label
angular momenta for spinning quantum particles.  This is exactly the
other attribute a particle in the 3D Ponzano-Regge model may have.
Mass and spin are the characteristics which determine the effect of a
particle on the connection---that is, its gravitational effect.  In
the Ponzano-Regge model, mass and spin label the edges of a graph
describing space.  In the case that the mass on an edge is zero, this
describes a \textit{spin network}, as described first by Penrose
\cite{penrose}.  A spin network is a combinatorial representation of
the geometry of space.

Penrose's original idea was that a quantum theory of gravity should
describe space in intrinsically discrete terms.  The description as a
graph is intrinsically discrete.  Edges are labelled with spins since
these are representations of the symmetry group related to angular
momentum.  This was chosen because angular momentum is already
discrete in quantum mechanics, and is plainly related to the (local)
rotational symmetry of space.

Such spin networks are related to the Ponzano-Regge model for 3D quantum
gravity.  The interpretation in terms of gravity comes from the
observation that a conjugacy class in $SU(2)$ is an angle in
$[0,\pi]$, which is a \textit{mass} $m$; In the case $m=0$ the
isomorphism is just a spin---an irrep of $SU(2)$, labelled by an
integer (or, for physics purposes, a half-integer).  For other $m$, we
get a spin when we reduce to a skeletal version of the 2-vector
spaces.

\begin{figure}[h]
\begin{center}
\includegraphics{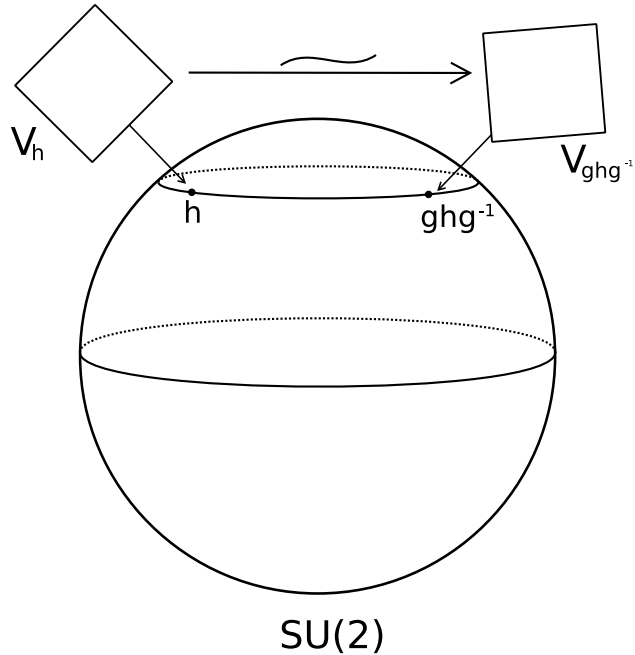}
\end{center}
\caption{Irreducible Object in $Z_{SU(2)}(S^1)$ \label{fig:catcentre}}
\end{figure}

The Ponzano-Regge model is a quantum theory which reproduces classical
General Relativity in a suitable limit.  Now, in General Relativity, gravity
can be thought of as the theory of a connection on a manifold, which
is the Levi-Civita connection associated to the metric in the usual
formalism.  So the Ponzano-Regge model can be seen as a quantum theory
for a connection on space of a given topology.

We can think of a cobordism with corners, as in Figure
\ref{fig:cobcorners} as having boundaries indicating the boundary of
the world lines of some system.  We can think of this as a Feynman
graph for some particles.  This interpretation makes the most sense if
our group $G$ is a Lorentz group, so that we think of the underlying
manifold with corners as ``spacetime''.  However, even if it is only
``space'', this cobordism can be though of as giving a graph, where
the circles represent the boundary in 2D ``space'' around some
system---the ``removed'' portions of space are the graph.  We can
think of the edges as particles---by which we only mean some bit of
matter.  A ``fundamental'' particle is then an irreducible state on
it.  This corresponds, as we remarked earlier, to a choice of a pair
$([g],\rho)$ consisting of a conjugacy class $[g]$ of $G$ and
representation $\rho$ of $G$ on some vector space.  Conjugacy classes
of rotation or Lorentz groups are ``mass shells'', corresponding to
the mass of the particle.  Representations of $G$, at least for
$SU(2)$ and similar groups, are labelled by ``spins''.  These
determine how a particle interacts with gravity.  This is precisely
what the Ponzano-Regge model describes: a network of edges labelled
with just this data, and with vertex amplitudes at the intersections.

So our extended TQFT gives ``particles''---boundaries in space -
labelled by a representation of a certain group.  Our example was
derived from a finite group, but if $G = SU(2)$ the label is a mass
and spin) moving on a background described by Ponzano-Regge quantum
gravity (see work by, for example, Freidel, Livine, and Louapr\'e \cite{PRI}
\cite{PRII}\cite{PRIII}, discussing the Ponzano-Regge model coupled to
matter, by Noui \cite{noui}, and Noui and Perez \cite{nouiperez} on 3D
quantum gravity with matter).

Baez, Crans, and Wise \cite{BCW} describe how conjugacy classes of
gauge groups can be construed as ``particle types'': an ``elementary''
particle corresponds with an irreducible 2-vector in $Z_G(B)$.  This
associates to a hole---whose boundary is diffeomorphic to the circle
$S^1$---a holonomy in a given conjugacy class $[g]$ of $G$.  This is
physically indistinguishable from any other corresponding to the same
class.  But they are distinguishable from particles giving holonomies
in some other conjugacy class.  So one says these represent different
``types'' of particle.  

Now, we have said that for a 3D extended TQFT, the 2-vector space of
states for a circle has a basis in which each object is given by a
conjugacy class of $G$ and representation of the stabilizer of that
class.  Wise \cite{wise} describes a way to interpret such conjugacy
classes as \textit{particle types} in a topological gauge theory.
More generally, in any dimension, given a space with a ``puncture'' of
codimension 2, there can be nontrivial holonomy for a connection
around that puncture.  In 3D, this is a 1-D puncture, which we think
of as the worldline of a point particle.  In the framework discussed
in this paper, we think of the particle as a puncture in 2D space,
surrounded by a 1D manifold, namely a circle.  This is the manifold
$B$ for our extended TQFT.  Then the ``space'' from which the particle
is removed is represented as the cobordism $S$ in our setup, and ``spacetime'' 

Just as a conjugacy class in $SU(2)$, as we have seen, can be
interpreted as a mass in Ponzano-Regge gravity, similarly, for other
gauge groups, conjugacy classes in the group classify ``types'' of
matter particles which may be coupled to the field.  A state for the
boundary around such a defect in our extended TQFT gives These
represent possible holonomies, up to gauge equivalence, around such a
defect.  These classify the physically distinguishable particles.

The interpretation described here so far is purely kinematical, though
in 4D, where these punctures are ``strings'' (i.e. the punctures in
space are 1-dimensional manifolds, namely circles, and in
``spacetime'' are 2-dimensional, namely ``world-sheets'') the dynamics
for such matter has been studied by Baez and Perez \cite{bfbrane}.  In
terms of our extended TQFT setting, the dynamics are described by the
action of $Z_G$ on cobordisms of cobordisms.

In particular, suppose we have a cobordism with corners $M: S \ra S'$
for cobordisms $S, S' : B \ra B'$, and are given specified ``particle
types'' for the punctures in the initial and final spaces.  This
amounts to choosing particular basis 2-vectors in $Z_G(B)$ and
$Z_G(B')$.

Then on each space---cobordisms having these punctures as boundary -
this gives a vector space as the component of $Z_G(S)$ which
corresponds to these basis 2-vectors, and similarly for $Z_G(S')$.
Then the corresponding component of $Z_G(M)$ is a linear operator
between these states.  The interpretation is that these components
describe the spaces of states for a field coupled to matter of the
specified type, and the linear operator which gives its time
evolution.  This is found, as we saw in (\ref{eq:ZGonM}), is given by
a certain ``sum over histories'', where each history is a connection
on the ``spacetime'' $M$.  The topology of the punctures in $M$ can be
thought of as a Feynman graph for interactions of the matter which is
the source of the field.

One should carefully note that to take this interpretation in terms of
``histories'' and ``spacetime'' literally requires a noncompact gauge
group $G$ such as Loretz groups $SO(2,1)$, $SO(3,1)$, or their double
covers $SL(2,\mathbbm{R})$ and $SL(2,\mathbbm{C})$ respectively.  We
expect that it would be more difficult to make these concepts precise
for noncompact gauge groups.

\subsection{Further Prospects}\label{sec:QGwards}

The relationship between the extended TQFT's discussed here and BF
theory leads one to ask about the relations between this approach and
other ways of looking at BF theory which have already been studied.
One of these which is particularly relevant involves so-called
\textit{spin foam} models.  A self-contained description of such
models for BF theory and quantum gravity by Baez \cite{BFfoam}.  Spin
foam models are a generalization of the \textit{spin networks} of
Penrose \cite{penrose}.

A spin network is a network in the sense of a \textit{graph}---a
collection of nodes, connected by edges.  In a spin network, the edges
are labelled by spins---representations of $SU(2)$, which are labelled
by half-integers.  The vertices by intertwining operators---that is,
morphisms in the category of representations of $SU(2)$ taking some
tensor product of irreducible representations to some other such
tensor product.  These are taken to be a representation of a
``combinatorial spacetime'' in which the nodes represent events, and
the edges give information about distance between events.  In
particular, the attitude is that this is the \textit{only} information
about distances within this combinatorial model of spacetime.

The idea behind spin foam models is to view spin networks as
describing configurations for the geometry of space.  Then a spin
\textit{foam} is a morphism between spin networks.  In fact, it is a
structure which contains spin networks as start and end states in much
the same way that an $n$-dimensional cobordism has $(n-1)$-manifolds
as source and targets.  A spin foam is a complex vertices, edges, and
faces, with group representations labelling faces, and intertwining
operators labelling edges.  So, in particular, a generic codimension-1
cross-section of a spin foam

The expected link to the present work is a generalization of the FHK
construction described in Section \ref{sec:fhk}.  In that case, one
develops a TQFT by using triangulations of the manifolds and
cobordisms on which the TQFT is to define Hilbert spaces and linear
maps.  We saw, as illustrated in Figure \ref{fig:FHK}, that there is a
network dual to this triangulation.  To the edges in this network one
assigns copies of a certain algebra, namely $\ZCG$, and to the nodes
one assigns a multiplication operator.  As described in Section
\ref{sec:pachner}, the coherence laws satisfied by these operators are
described by tetrahedrons.  These are the Pachner moves in 2-D:
attaching a tetrahedron to a triangulation along one, two, or three
triangular faces gives a move by replacing the attached faces with the
remaning faces of the tetrahedron.  The way of assigning an operator
to a vertex of the dual to a triangulation must have the property that
it is invariant under such moves.

We have categorified this picture in order to increase the codimension
of the theory - that is, the difference in dimension between the basic
manifolds and the highest-dimensional cobordisms.  So there should be
a categorified equivalent of the FHK construction, in which we begin
with triangulated manifolds and cobordisms.  In categorifying, we
replace the equations given by the Pachner moves with 2-morphisms.
Each move gives a 2-morphism between a pair of morphisms in $\iiV$,
corresponding to a tetrahedron thought of as a cobordism connecting
two parts of its boundary.  Any cobordism can be built of such units,
attached together in some triangulation:

\begin{figure}[h]
\begin{center}
\includegraphics{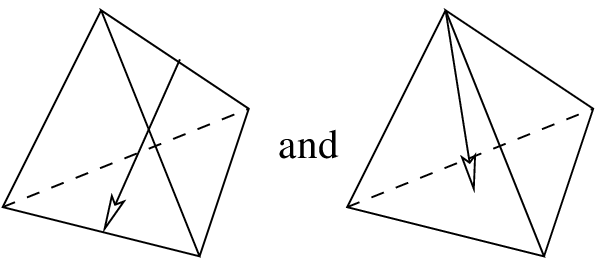}
\end{center}
\caption{Tetrahedra Assigned 2-Morphisms \label{fig:tetra2mor}}
\end{figure}

These obey coherence laws (equations) given by the 2-3 and 1-4 Pachner
moves:

\begin{figure}[h]
\begin{center}
\includegraphics{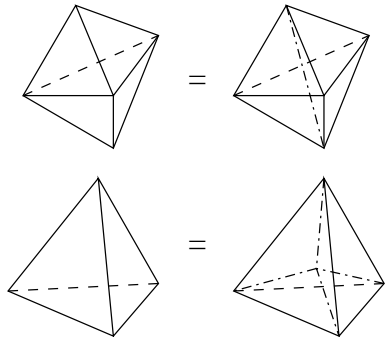}
\end{center}
\caption{Coherence Rules as Pachner Moves \label{fig:tetracoherence}}
\end{figure}

As in 2D, where the algebra assigned by the FHK construction to edges
is $\ZCG$, the categorified version should assigne $\ZVG$, which
corresponds to our assignment to a circle of equivariant
$\V$-presheaves on $G$.  Assigning these to edges reproduces the
Ponzano-Regge model when $G=SU(2)$, since the irreducible objects in
this category are, as we have seen, precisely labelled by mass and
spin.  Analogous results hold for other $G$, giving different field
theories.  But notice that this is different from the way we recovered
the Ponzano-Regge model above: now we are assigning this data to edges
of a triangulation, not a boundary of a ``worldline''.  The relations
between these two pictures are close, but more than we can go into in
detail here.

However, it is enough to observe that there is a close relation
between the extended TQFT we have developed and state-sum
(i.e. spin-network and spin-foam) models for BF theory, and 3D quantum
gravity.  So one avenue for further exploration is to see how the
framework described here can be extended to incorporate other theories
described by such state sum models.

Our basic result involved the construction of an extended TQFT as a
weak 2-functor for any finite group $G$.  In Section
\ref{sec:liegroups}, we discussed the possibility of extending the
construction to the case where $G$ is a Lie group, and in particular,
indicated that this is expected to be more natural when $G$ is a
\textit{compact} Lie group.  Of course, noncompact groups are also of
interest - for example, the Lorentz groups.  But there are other
directions in which to generalize this.  We briefly consider two
possibilities here: categorical groups (also known as 2-groups), and
quantum groups (by which we mean quasitriangular Hopf algebras).

An extension of the Dijkgraaf-Witten theory to categorical groups is
described by Martins and Porter \cite{martinsporter}.  A categorical
group, also known as a 2-group, is a category object in the category
$\catname{Grp}$ of groups.  That is, it is a structure having a group
of objects and a group of morphisms, satisfying the usual category
axioms expressed in terms of morphisms within $\catname{Grp}$.  Any
group $G$ is an example of a 2-group, where the group of objects is
trivial : this is in fact how we have been thinking of the gauge group
$G$ throughout this paper.  But there are many other examples of
2-groups, including, importantly for us, 2-groups which arise from
semidirect products of groups $H \rtimes G$.  In this case, the group
of objects is $G$ and whose group of morphisms is $H \rtimes G$: the
group of automorphisms of any given object is isomorphic to $H$.  Such
a 2-group is called an \textit{automorphic 2-group}.

The category of 2-groups can be shown to be equivalent to the category
of \textit{crossed modules}, a concept due to Brown and Spencer
\cite{brownspencer}.  A crossed module consists of a tuple
$(G,H,t,\alpha)$, where $G$ and $H$ are groups, $t : H \ra G$ is a
homomorphism, and $\alpha : G \ra Aut(H)$ is an action of $G$ on $H$,
such that $t$ and $\alpha$ satisfy some compatibility conditions,
which turn out to be equivalent to the category axioms in the 2-group
described above.  The \textit{Poincare 2-group}, introduced by Baez
\cite{baezpoinc}, is an example of an automorphic 2-group.  It has
been a subject of interest as a source of a new class of spin foam
models, first suggested by Crane and Sheppeard \cite{cranesheppeard}.
Such models are based on the representation theory of 2-groups, which
is 2-categorical in nature, since one must consider representations,
intertwiners between representations, and 2-intertwiners between
intertwiners, which form a 2-category.  A spin foam model based on a
2-group uses these to label faces, edges, and vertices respectively.

The most evident relation of 2-groups to the sort of extended TQFT's
we have been discussing is related to gauge theory.  The role of the
group $G$ in constructing the weak 2-functor $Z_G$ was through the
groupoid of \textit{connections on $G$-bundles} on a space $X$.  This
is $\fc{X}$, the category of functors from the fundamental groupoid of
$X$ into $G$ thought of as a category with one object.  One might
suppose that the natural extension would be to take $G$ to be a
2-group, with a group of objects, and take functors from $\Pi_1(X)$
into this.

This could be done, but perhaps a better approach is in the form of
\textit{higher gauge theory}.  Discussion of higher gauge theory can
be found in work by Baez and Schreiber \cite{HGA} and Bartels
\cite{bartels}.  The principle is that one should assign data from a
2-group to both paths and homotopies of paths, so what one uses is not
$\Pi_1(X)$, but $\Pi_2(X)$, the fundamental 2-groupoid of $X$.  This
is a 2-category whose objects are points in $X$, morphisms are paths,
and 2-morphisms are homotopy classes of homotopies of paths.  It
should be clear that this encodes information not only about the first
homotopy group of a space (as does the fundamental group), but also
the second homotopy group.  In higher gauge theory one studies, in
this case, flat ``2-connections'' (or more generally $n$-connections),
which are seen as 2-functors from $\Pi_2(X)$ to a 2-group.  In 3D, we
have discussed how an extended TQFT based on a Lie group could
possibly describe the evolution of point-particles along worldlines in
spacetime.  A categorified form of this based on 2-connections could
be of interest in 4D, where one could study the behaviour of
``strings'' as well as point particles (see, for instance, Baez and
Perez \cite{bfbrane}).

Having begin by categorifying the standard definition of a TQFT, one
could then hope to continue the process and find an infinite ``tower''
of theories, each having one more codimension than the last.

The last possible direction of generalization from our extended TQFT
based on a finite group would involve quantum groups.  Whereas moving
to 2-groups involves ``categorifying'' the concept of a group, moving
to quantum groups, as the name suggests, involves ``quantizing''.
Neither procedure is, in general, a well defined operation, but
particular examples are understood.  In particular, we could try to
generalize from finite groups to ``finite quantum groups'', by which
we mean finite-dimensional quasitriangular Hopf algebras.

The idea behind quantum groups is described by Shahn Majid
\cite{majid} and also notably by Ross Street \cite{streetqgrp}.  The
idea provides a way to speak of deforming topological groups, although
there is no way of smoothly deforming the group action of a
topological group to a family of other such groups.  Instead, one
works in a larger category, of ``quantum'' groups, of which usual
groups correspond to special cases.  This is done using
\textit{Gelfand duality}, which relates commutative algebra and
topological spaces.  Specifically, it gives an equivalence saying that
each commutative $C^{\star}$-algebra is the algebra $C(X)$ of
continuous complex functions on a compact Hausdorff space $X$.

Continuous functions $f : X \ra Y$ give algebra homomorphisms $C(f) :
C(Y) \ra C(X)$, so that if $X$ is a group as well as a space, the
$C^{\star}$-algebra $C(X)$ gets a comultiplication $C(\cdot) : C(X)
\ra C(X) \otimes C(X)$, counit $C(1) : C(X) \ra \mathbbm{C}$ and
involution $C(^{-1}): C(X) \ra C(X)$.  Since these come from
operations on a group, they, along with the (pointwise)
multiplication, unit, and inverse in $C(X)$, satisfy certain axioms,
and relations.  The axioms for a \textit{Hopf} algebra generalize
these.  In particular, they require that the multiplication be
associative, but not necessarily commutative.  A quasitriangular, or
``braided'' Hopf algebra $H$ has a distinguished element $\gamma$,
thought of as the image of $1 \otimes 1$ under a ``switch'' operation
$H \otimes H \ra H \otimes H$.  These Hopf algebras are what are
called ``quantum groups''.

We will not attempt a full explanation of quantum groups here, though
see the above references for full details).  For our purposes, the
interesting point is that the Hopf algebras coming from Lie groups $G$
as $C(G)$ can be deformed to noncommutative and non-cocommutative
quantum groups, with a parameter $q$ which is a unit complex number.
Given elements $x$ and $y$, the deformation replaces the operations
such as multiplication by new ones, given as power series in $q$.
When $q$ is a complex root of unity, this has particularly good properties.

In particular, we expected to recover the Ponzano-Regge model of 3D
quantum gravity, based on $SU(2)$, as an extended TQFT.  Now, the
Turaev-Viro model (see \cite{turaevviro} and \cite{foxon}) is based on
the $q$-deformed quantum groups $SU(2)_q$, and in some respects is
more convenient than the Ponzano-Regge model.  In particular, there
are infinitely many representations of $SU(2)$, but only finitely many
of $SU(2)_q$ when $q$ is an $n^{th}$ root of unity (specifically,
$n-1$ of them).  This gives better convergent properties when summing
over representations.  In general, spin foam models involving quantum
groups sometimes have such good finiteness properties.

As a first effort to generalize from our situation of an extended TQFT
based on a finite group, we may try to develop an extended TQFT from
the corresponding class of quantum groups - namely, finite dimensional
quasitriangular Hopf algebras.

Finally, it should be possible to combine our different directions of
generalization.  For example, Crane and Yetter \cite{CY} discuss
generally a similar family of algebraic and higher-algebraic
structures which give rise to TQFTs in various dimensions.  In 4D, the
relevant structure is a Hopf category - a categorified equivalent of a
Hopf algebra.  Marco Mackaay \cite{mackaay} shows very explicitly how
to construct invariants of 4-manifolds from certain kinds of
2-categories by means of the sort of state-sum model which we have
been discussing.  It would be useful to study how much of this can be
described in the ``geometric'' style which we have examined here in
the form of groupoids of connections.

All of these directions suggest ways in which our results could be
expanded further by future investigation.

\appendix
\section{Internal Bicategories in $\Bicat$}\label{sec:internalbicat}

We rely on the notion of a bicategory \textit{internal} to $\Bicat$ in
our discussion of {\vdbs} in Chapter \ref{chap:doublebicat}, and thus
in the development of the {\vdb} $\nCob$ in Chapter
\ref{chap:cobcorn}.  Here we present a more precise definition of this
concept, and in Lemmas \ref{lemma:span2bicat} and
\ref{lemma:doublebicat} we use it to show that examples having
properties like those of $\iiCCosp_0$ (definition \ref{def:cspan20})
give ``double bicategories'' in the sense of Verity.  These lemmas were
used in the proofs of Theorems \ref{thm:cspanthm} and
\ref{thm:maintheorem}.

To begin with, we remark that the theory of bicategories,
$\catname{Th(\Bicat)}$ is more complicated than that for categories.
However as with $\catname{Th(\Cat)}$, it will be a category with
objects $\Obj$, $\M$ and $\B$, and having all equalizers and
pullbacks.  To our knowledge, a model of $\catname{Th(\Bicat)}$ in
$\Bicat$ has not been explicitly described as such before.  We could
treat $\Obj$ as a horizontal bicategory, and the objects of $\Obj$,
$\M$ and $\B$ as forming a vertical bicategory, but we note that
diagrammatic representation of, for instance, 2-morphisms in $\B$
would require a 4-dimensional diagram element.  The comparison can be
seen by contrasting tables \ref{table:doublecatdata} and
\ref{table:doublebicatdata}.

The axioms satisfied by such a structure are rather more unwieldy than
either a bicategory or a double category, but they provide some order
to the axioms for a {\vdb}, as shown in Definition
\ref{def:doublebicat}.  We note that, although that definition is
fairly elaborate, it is simpler than would be a similarly elementary
description of a \db.

In particular, where there are compatibility conditions involving
equations in this definition, such a structure would have only
isomorphisms, themselves satisfying additional coherence laws.  In
particular, in {\vdbs}, the action of 2-morphisms on
squares is described by strict equations, rather than being given by a
definite isomorphism.

Similarly, it is possible (see \cite{verity} sec. 1.4) to define
categories $\catname{Cyl_H}$ (respectively, $\catname{Cyl_V}$) of
\textit{cylinders} whose objects are squares, and maps are pairs of
vertical (respectively, horizontal) 2-morphisms joining the vertical
(resp. horizontal) source and targets of pairs of squares which share
the other two sides (this is shown in Table
\ref{table:doublebicatdata}, in Section \ref{sec:decatfy}: the
cylinders are ``thin'' versions of higher morphisms appearing there).
These are plain categories, with strict associativity and unit laws.
These conditions would be weakened in a {\db} (in which maps would
include not just pairs of 2-morphisms, but also a 3-dimensional
interior of the cylinder, which is a morphism in $2\Mor$, or
2-morphism in $\Mor$, satisfying properties only up to a 4-dimensional
2-morphism in $2\Mor$).

We start to see all this by describing how to obtain a {\db}.

\subsection{The Theory of Bicategories}

We described in Section \ref{sec:doublecat} how a double category may
be seen as a category internal to $\Cat$.  To put it another way, it a
model of $\catname{Th(\Cat)}$, the theory of categories, in $\Cat$,
which is a limit-preserving functor from $\catname{Th(\Cat)}$ into
$\Cat$.  We did not make a special point of the fact, but this is a
\textit{strict} model.  A weak model would satisfy the category axioms
such as composition only up to a 2-morphism in $\Cat$, namely up to
natural transformation.  So, for instance, the pullback
(\ref{eq:thcatmodel}) would be a weak pullback, so that instead of
satisfying $t \circ c_1 = s \circ c_2$, there would only be a natural
transformation relating $t \circ c_1$ and $s \circ c_2$.  Such a weak
model is the most general kind of model available in $\Cat$, but
double categories arise as strict models.

So here we note that we are thinking of $\Bicat$ as a mere category,
and that we are speaking of \textit{strict} internal bicategories.
In particular, the most natural structure for $\Bicat$ is that of a
tricategory: it has objects which are bicategories, morphisms which
are weak 2-functors between bicategories, 2-morphisms which are natural
transformations between such weak 2-functors, and 3-morphisms which are
``modifications'' of such transformations.  Indeed, $\Bicat$ is the
standard example of a tricategory, just as $\Cat$ is the standard
example of a bicategory.  But we ignore the tricategorical structure
for our purposes.

So as with double categories, we only consider strict models of the
theory of bicategories, $\catname{Th(\Bicat)}$ in $\Bicat$.  That is,
functors from the category $\catname{Th(\Bicat)}$ into $\Bicat$ (seen
as a category).  Equations in a model are mapped to equations (not
isomorphisms) in $\Bicat$.  We call these models {\dbs}.

Before we can say explicitly what this means, we must describe
$\catname{Th(\Bicat)}$ as we did for $\catname{Th(\Cat)}$ in Section
\ref{sec:doublecat}.

\begin{definition} The \textit{theory of bicategories} is the category
(with finite limits) $\catname{Th(\Bicat)}$ given by the following
data:
\begin{itemize}
\item Objects $\opname{Ob}$,  $\opname{Mor}$,  $\opname{2Mor}$
\item Morphisms $s,t:\opname{Ob} \ra \opname{Mor}$ and $s,t:\opname{Mor} \ra \opname{2Mor}$
\item \textbf{composition} maps $\circ :
      \opname{MPairs} \ra \M$ and $\cdot : \opname{BPairs}
      \ra \B$, satisfying the interchange law
      (\ref{eq:bicatinterchange}), where $\opname{MPairs} =
      \opname{Mor} \times_{\opname{Ob}} \opname{Mor}$ and
      $\catname{BPairs} = \opname{2Mor} \times_{\opname{Mor}}
      \opname{2Mor}$ are equalizers of diagrams of the form:
  \begin{equation}
    \xymatrix{
        &  & \opname{Mor} \ar[dr]^{t} & \\
      \catname{MPairs} \ar[r]^{i} & \opname{Mor}^2 \ar[ru]^{\pi_1}
        \ar[rd]_{\pi_2} & & \opname{Ob} \\
        &  & \opname{Mor} \ar[ur]^{s} & \\
    }
  \end{equation}
      and similarly for $opname{BPairs}$.
\item the \textbf{associator} map $a : \opname{Triples} \ra
      \B$, where $\opname{Triples} = \opname \times_{\opname{Ob}}
      \opname{Mor} \times_{\opname{Ob}} \opname{Mor}$ is the equalizer
      of a similar diagram for involving $\opname{Mor}^3$, such that
      $a$ satisfies $s(a(f,g,h)) = (f \circ g) \circ h$ and
      $t(a(f,g,h)) = f \circ (g \circ h)$
\item \textbf{unitors} $l,r : \opname{Ob} \ra \opname{Mor}$
      with $s \circ l = t \circ l = \opname{id}_{\opname{Ob}}$ and $s
      \circ r = t \circ r = \opname{id}_{\opname{Ob}}$
\end{itemize} This data is subject to the conditions that the
associator is subject to the Pentagon identity, and the unitors obey
certain unitor laws.
\end{definition}

\begin{remark}The Pentagon identity is shown in (\ref{eq:pentagonid})
for a model of $\catname{Th(\Bicat)}$ in $\catname{Sets}$), where we
can specify elements of $\Mor$, but the formal relations---that the
composites on each side of the diagram are equal---hold in general.
These are built from composable quadruples of morphisms and
composition as indicated in the labels.  Similar remarks apply to the
unitor laws shown in (\ref{eq:unitorlaws}).
\end{remark}

So we have the following:

\begin{definition}
A {\db} consists of:
\begin{itemize}
\item bicategories $\Ob$ of \textbf{objects}, $\M$ of
      \textbf{morphisms}, $\B$ of \textbf{2-morphisms}
\item \textbf{source} and \textbf{target} maps $s,t : \M \ra
      \Ob$ and $s,t : \B \ra \M$
\item partially defined \textbf{composition} functors $\circ : \M^2
      \ra \M$ and $\cdot : \B^2 \ra \B$, satisfying
      the interchange law (\ref{eq:bicatinterchange})
\item partially defined \textbf{associator} $a : \M^3 \ra \B$
      with $s(a(f,g,h)) = (f \circ g) \circ h$ and $t(a(f,g,h)) = f
      \circ (g \circ h)$
\item partially defined \textbf{unitors} $l,r : \Ob \ra \M$
      with $s(l(x)) = t(l(x)) = x$ and $s(r(x)) = t(r(x)) = x$
\end{itemize}
All the partially defined functors are defined for \textit{composable}
pairs or triples, for which source and target maps coincide in the
obvious ways.  The associator should satisfy the pentagon identity
(\ref{eq:pentagonid}), and the unitors should satisfy the unitor laws
(\ref{eq:unitorlaws}).
\end{definition}

With this definition in mind, we recall B\'enabou's classic example of
a bicategory, that of spans, revieweed in Section \ref{sec:spanbicat}.
There is an analogous example here, namely double spans, or in our
case \textit{double cospans}.

\subsection{The Double Cospan Example}\label{sec:doublecospan}

In Section \ref{sec:dblspan}, we described a {\vdb} of ``double
cospans'', $\iiCCosp_0$.  This notation is intended to suggest it
derives from a larger structure, $\iiCCosp$, which is a {\db}, as we
shall show shortly.  It is analogous to the ``profunctor-based
examples'' of pseudo-double categories described by Grandis and Par\'e
\cite{GP2}.  The {\vdb} described above is derived from it.  To see
these facts, we first define $\iiCCosp$ explicitly:

\begin{definition} $\iiCCosp$ is a {\db} of \textbf{double cospans} in $\C$,
consisting of the following: \begin{itemize}
\item the bicategory of objects is $\Ob = \CCosp$
\item the bicategory of morphisms $\M$ has: as objects, cospans in $\C$;
      as morphisms, commuting diagrams of the form \ref{xy:cspan2}
(in subsequent diagrams we suppress the labels for clarity)

\item as 2-morphisms, cospans of cospan maps, namely commuting
      diagrams of the following shape:
\begin{equation}\label{xy:cospanmap1}
\xy
 (0,0)*{\bullet}="11";
 (0,15)*{\bullet}="12";
 (0,30)*{\bullet}="13";
 (15,0)*{\bullet}="21";
 (15,15)*{\bullet}="22";
 (15,30)*{\bullet}="23";
 (30,0)*{\bullet}="31";
 (30,15)*{\bullet}="32";
 (30,30)*{\bullet}="33";
 (-5,10)*{\bullet}="12a";
 (10,10)*{\bullet}="22a";
 (25,10)*{\bullet}="32a";
     {\ar "11";"12"};
     {\ar "13";"12"};
     {\ar "21";"22"};
     {\ar "23";"22"};
     {\ar "31";"32"};
     {\ar "33";"32"};
     {\ar "11";"21"};
     {\ar "31";"21"};
     {\ar "12";"22"};
     {\ar "32";"22"};
     {\ar "13";"23"};
     {\ar "33";"23"};
     {\ar "12a";"22a"};
     {\ar "32a";"22a"};
     {\ar "11";"12a"};
     {\ar "13";"12a"};
     {\ar "21";"22a"};
     {\ar "23";"22a"};
     {\ar "31";"32a"};
     {\ar "33";"32a"};
     {\ar "12";"12a"};
     {\ar "22";"22a"};
     {\ar "32";"32a"};
\endxy
\\
\end{equation}


\item the bicategory of 2-morphisms has:
  \begin{itemize}
  \item as objects, cospan maps in $\C$ as in (\ref{eq:spanmorph})
  \item as morphisms, cospan maps of cospans:
\begin{equation}\label{xy:cospanmap2}
\xy
 (0,0)*{\bullet}="11";
 (15,0)*{\bullet}="12";
 (30,0)*{\bullet}="13";
 (0,15)*{\bullet}="21";
 (15,15)*{\bullet}="22";
 (30,15)*{\bullet}="23";
 (0,30)*{\bullet}="31";
 (15,30)*{\bullet}="32";
 (30,30)*{\bullet}="33";
 (20,-5)*{\bullet}="12a";
 (20,10)*{\bullet}="22a";
 (20,25)*{\bullet}="32a";
     {\ar "11";"12"};
     {\ar "13";"12"};
     {\ar "21";"22"};
     {\ar "23";"22"};
     {\ar "31";"32"};
     {\ar "33";"32"};
     {\ar "11";"21"};
     {\ar "31";"21"};
     {\ar "12";"22"};
     {\ar "32";"22"};
     {\ar "13";"23"};
     {\ar "33";"23"};
     {\ar "12a";"22a"};
     {\ar "32a";"22a"};
     {\ar "11";"12a"};
     {\ar "13";"12a"};
     {\ar "21";"22a"};
     {\ar "23";"22a"};
     {\ar "31";"32a"};
     {\ar "33";"32a"};
     {\ar "12";"12a"};
     {\ar "22";"22a"};
     {\ar "32";"32a"};
\endxy
\\
\end{equation}

  \item as 2-morphisms, cospan maps of cospan maps:
\begin{equation}\label{eq:cospanmapmap}
\xy
 (0,0)*{\bullet}="11";
 (30,0)*{\bullet}="12";
 (60,0)*{\bullet}="13";
 (0,30)*{\bullet}="21";
 (30,30)*{\bullet}="22";
 (60,30)*{\bullet}="23";
 (0,60)*{\bullet}="31";
 (30,60)*{\bullet}="32";
 (60,60)*{\bullet}="33";
 (40,-5)*{\bullet}="12a";
 (40,25)*{\bullet}="22a";
 (40,55)*{\bullet}="32a";
 (-5,20)*{\bullet}="21b";
 (25,20)*{\bullet}="22b";
 (55,20)*{\bullet}="23b";
 (35,15)*{\bullet}="22ab";
     {\ar "11";"12"};
     {\ar "13";"12"};
     {\ar "21";"22"};
     {\ar "23";"22"};
     {\ar "31";"32"};
     {\ar "33";"32"};
     {\ar "11";"21"};
     {\ar "31";"21"};
     {\ar "12";"22"};
     {\ar "32";"22"};
     {\ar "13";"23"};
     {\ar "33";"23"};
     {\ar "12a";"22a"};
     {\ar "32a";"22a"};
     {\ar "11";"12a"};
     {\ar "13";"12a"};
     {\ar "21";"22a"};
     {\ar "23";"22a"};
     {\ar "31";"32a"};
     {\ar "33";"32a"};
     {\ar "12";"12a"};
     {\ar "22";"22a"};
     {\ar "32";"32a"};
     {\ar "11";"21b"};
     {\ar "31";"21b"};
     {\ar "12";"22b"};
     {\ar "32";"22b"};
     {\ar "13";"23b"};
     {\ar "33";"23b"};
     {\ar "21b";"22ab"};
     {\ar "23b";"22ab"};
     {\ar "22b";"22ab"};
     {\ar "22a";"22ab"};
     {\ar "21";"21b"};
     {\ar "22";"22b"};
     {\ar "23";"23b"};
     {\ar "21b";"22b"};
     {\ar "23b";"22b"};
     {\ar "12a";"22ab"};
     {\ar "32a";"22ab"};
\endxy
\\
\end{equation}
  \end{itemize}
\end{itemize} All composition operations are by pushout; source and
target operations are the same as those for cospans.
\end{definition}

Note that we could of course make the dual definition for spans, which
may be more natural (but is not what we need for the cobordism case).

\begin{remark} Just as 2-morphisms in $\M$ and morphisms in $\B$ can
be seen as diagrams which are ``products'' of a cospan with a map of
cospans, 2-morphisms in $\B$ are given by diagrams which are
``products'' of horizontal and vertical cospan maps.  These have, in
either direction, four maps of cospans, with objects joined by maps of
cospans.  Composition again is by pushout in composable pairs of
diagrams.
\end{remark}

The next lemma shows how this is really an example of a {\db}:

\begin{lemma}\label{lemma:span2bicat} For any category $\C$ with
pushouts, $\iiCCosp$ forms a {\db}.
\end{lemma}
\begin{proof}
$\M$ and $\B$ are bicategories since the composition functors act just
like composition in $\CCosp$ in each column, and therefore satisfies
the same axioms.

Since the horizontal and vertical directions are symmetric, we can
construct functors between $\Ob$, $\M$, and $\B$ with the properties
of a bicategory simply by using the same constructions that turn each
into a bicategory.  In particular, the source and target maps from
$\M$ to $\Ob$ and from $\B$ to $\M$ are the obvious maps giving the
ranges of the projection maps in (\ref{xy:cspan2}).  The partially
defined (horizontal) composition maps $\circ : \M^2 \ra \M$
and $\otimes_H : \B^2\ra \B$ are defined by taking pushouts
of diagrams in $\C$, which exist for any composable pairs of diagrams
because $\C$ has pushouts.  They are functorial since they are
independent of composition in the horizontal direction.  The
associator for composition of morphisms is given in the pushout
construction.

To see that this construction gives a \db, we note that
$\Ob$, $\M$, and $\B$ as defined above are indeed bicategories.
$\Ob$, because $\CCosp$ is a bicategory.  $\M$ and $\B$ because the
morphism and 2-morphism maps from the composition, associator, and
other functors required for an {\db} give these the structure of
bicategories as well.

Moreover, the composition functors satisfy the properties of a
bicategory for just the same reason that composition of cospans (and
spans) does, since each of the three maps involved are given by this
construction.  Thus, we have a {\db}.
\end{proof}


\subsection{Decategorification}\label{sec:decatfy}

Our motivation for showing Lemma \ref{lemma:span2bicat} is to get show
that cobordisms with corners form a special example of a {\vdb} of
double cospans in some suitable category $\C$.  We have described how to
get a {\db} of such structures, so to get what we want, we need to
show how a {\vdb} can be a special kind of {\db}.  In particular, we
need to define conditions which allow us to speak of the action of
2-cells upon squares.  It is helpful, in trying to understand what
these are, to consider a ``lower dimensional'' example of a similar
process.

In a double category, thought of as an internal category in $\Cat$, we
have data of four sorts, as shown in Table \ref{table:doublecatdata}.

\begin{table}[h]
\begin{tabular}{|l|l|l|}
\hline
& $\Ob$ & $\M$ \\
\hline
Objects
&
\begin{minipage}{1in}
  \begin{equation*}
  \xymatrix{
  \bullet^{x} \\
  }
  \end{equation*}
\end{minipage}
 & 
\begin{minipage}{1in}
  \begin{equation*}
  \xymatrix{
  \bullet \ar[r]^{f} & \bullet
  }
  \end{equation*}
\end{minipage}
\\

\hline
Morphisms
&
\begin{minipage}{1in}

  \begin{equation*}
  \xymatrix{
  \bullet \ar[d]^{g} \\
  \bullet
  }
  \end{equation*}
\end{minipage}

 & 
\begin{minipage}{1in}
  \begin{equation*}
  \xymatrix{
  \bullet \ar[r] \ar[d] & \bullet \ar[d] \\
  \bullet \ar[r] \uriicell{F} & \bullet
  }
  \end{equation*}
\end{minipage}
\\

  \hline

\end{tabular}
\caption{Data of a Double Category\label{table:doublecatdata}}
\end{table}

That is, a double category $\catname{DC}$ has categories $\Ob$ of
objects and $\M$ of morphisms.  The first column of the table shows
the data of $\Ob$: its objects are the objects of $\catname{DC}$; its
morphisms are the \textit{vertical} morphisms.  The second column
shows the data of $\M$: its objects are the \textit{horizontal}
morphisms of $\catname{DC}$; its morphisms are the squares of
$\catname{DC}$.

\begin{remark}The kind of ``decategorification'' we will want to do to
obtain {\vdbs} has an analog in the case of double categories.  Namely,
there is a condition we can impose which effectively turns the double
category into a category, where the horizontal and vertical morphisms
are composable, and the squares can be ignored.  The sort of condition
involved is similar to the \textit{horn-filling conditions} introduced
by Ross Street \cite{street} in his first introduction of the idea of
weak $\omega$-categories.  In that case, all morphisms correspond to
simplicial sets, and a horn filling condition is one which says that,
for a given hollow simplex with just one face (morphism) missing from
the boundary, there will be a morphism to fill that face, and a
``filler'' for the inside of the simplex, making the whole commute.  A
restricted horn-filling condition demands that this is possible for
some class of candidate simplices.
\end{remark}

For a double category, morphisms can be edges or squares, rather than
$n$-simplices, but we can define the following ``filler'' condition:
given any pair $(f,g)$ of a horizontal and vertical morphism where the
target object of $f$ is the source object of $g$, there will be a
unique pair $(h,\star)$ consisting of a unique horizontal morphism $h$ and
unique invertible square $\star$ making the following diagram commute:
\begin{equation}\label{eq:hornfiller}
\xymatrix{
  x \ar@{-->}[d]_{h} \ar[r]^{f} & y \ar[d]^g \\
  z \ar[r]_{1_z} \uriicell{\ast} & z
}
\end{equation}
and similarly when the source of $f$ is the target of $g$.  Notice
that taking $f$ or $g$ to be the identity in these cases implies $F$
is the identity.

If, furthermore, there are no other interesting squares, then this
double category can be seen as just a category.  In that case, the
unique $h$ can just be interpreted as the composite of $f$ and $g$ and
$\star$ as the process of composition.  So we will use the notation
$g\circ f$ instead of $h$ in this situation.

To see that this defines a composition operation, we need to observe
that composition defined using these fillers agrees with the usual
composition in the horizontal or vertical categories, is associative,
etc.  For example, given morphisms as in the diagram:
\begin{equation}
\xymatrix{
  w \ar[r]^{f} & x \ar[r]^{f'} & y \ar[d]^{g}\\
  z \ar[r]_{1_z} & z \ar[r]_{1_z} & z \\
}
\end{equation}
there are two ways to use the unique-filler principle to fill this
rectangle.  One way is to first compose the pairs of horizontal
morphisms on the top and bottom, then fill the resulting square.  The
square we get is unique, and the morphism is denoted $g \circ (f'
\circ f)$.  The second way is to first fill the right-hand square, and
then using the unique morphism we call $g \circ f'$, we get another
square on the left hand side, which our principle allows us to fill as
well.  The square is unique, and the resulting morphism is called $(g
\circ f') \circ f$.  Composing the two squares obtained this way must
give the square obtained the other way, since both make the diagram
commute, and both are unique.  So we have:
\begin{equation}
\xymatrix{
  w \ar@{-->}[d]_{(g \circ f') \circ f} \ar[r]^{f} & x \ar@{-->}[d]^{g \circ f'} \ar[r]^{f'} & y \ar[d]^{g} \\
  z \ar[r]_{1_z} \uriicell{\ast} & z \ar[r]_{1_z} \uriicell{\ast}  & z \\
} \qquad = \qquad
\xymatrix{
  w \ar@{-->}[d]_{g \circ (f'\circ f)} \ar[r]^{f' \circ f} & y \ar[d]^{g} \\
  z \ar[r]_{1_z} \uriicell{\ast} & z \\
}
\end{equation}

So in fact we can ``decategorify'' a double category satisfying the
unique filler condition, and treat it as if it were a mere category
with horizontal and vertical morphisms equivalent.  The composition
between horizontal and vertical morphisms is given by the filler:
given one of each, we can find a square of the required kind, by
taking the third side to be an identity.

\begin{remark} Note that our condition does not give a square for
every possible combination of morphisms which might form its sources
and targets.  In particular, there must be an identity morphism---on
the bottom in the example shown.  If that identity could be any
morphism $h$, then by choosing $f$ and $g$ to be identities, this
would imply that every morphism must be invertible (at least weakly),
since there must then be an $h^{-1}$ with $h^{-1} \circ h$ isomorphic
to the identity.  When a filler square does exist, and we consider
$\catname{DB}$ as a category $\C$, the filler square indicates there
is a commuting square in $\C$: we think of it as the identity between
the composites along the upper right and lower left.
\end{remark}

The decategorification of a {\db} to give a {\vdb} is similar, except that
whereas with a double category we were cutting down only the squares
(the lower-right quadrant of Table \ref{table:doublecatdata}.  We need
to do more with a \db, since there are more sorts of data, but they
fall into a similar arrangement, as shown in Table
\ref{table:doublebicatdata}.

\begin{table}[h]
\begin{tabular}{|l|l|l|l|}
\hline
& $\Ob$ & $\M$ & $\B$ \\
\hline
Objects
&
\begin{minipage}{1in}
  \begin{equation*}
  \xymatrix{
  \bullet^{x} \\
  }
  \end{equation*}
\end{minipage}
 & 
\begin{minipage}{1in}
  \begin{equation*}
  \xymatrix{
  \bullet \ar[r]^{f} & \bullet
  }
  \end{equation*}
\end{minipage}
 & 
\begin{minipage}{1.5in}
  \begin{equation*}
  \xymatrix{
  \bullet \ar@/^1pc/[r]^{ }="0" \ar@/_1pc/[r]^{ }="1" & \bullet \\
  \ar@{=>}"0" ;"1"^{\alpha}
  }
  \end{equation*}
\end{minipage}
\\

\hline
Morphisms
&
\begin{minipage}{1in}

  \begin{equation*}
  \xymatrix{
  \bullet \ar[d]^{g} \\
  \bullet
  }
  \end{equation*}
\end{minipage}

& 
\begin{minipage}{1in}
  \begin{equation*}
  \xymatrix{
  \bullet \ar[r] \ar[d] & \bullet \ar[d] \\
  \bullet \ar[r] \uriicell{F} & \bullet
  }
  \end{equation*}
\end{minipage}
&
\begin{minipage}{1.5in}
  \begin{equation*}
  \xymatrix{
  \bullet \driicell{P_1} \ar[d] \ar@/^1pc/[r]^{ }="0" \ar@/_1pc/[r]^{ }="1" & \bullet \ar[d] \\
  \ar@{=>}"0" ;"1"
  \bullet \ar@{-->}@/^1pc/[r]^{ }="2" \ar@/_1pc/[r]^{ }="3" & \bullet  \\
  \ar@{==>}"2" ;"3"
  }
  \end{equation*}
\end{minipage}

\\

  \hline
2-Cells
&
\begin{minipage}{1in}
  \begin{equation*}
  \xymatrix{
  \bullet \ar@/^1pc/[d]^{ }="0" \ar@/_1pc/[d]^{ }="1" \\
  \bullet \\
  \ar@{=>}"0" ;"1"^{\alpha}
  }
  \end{equation*}
\end{minipage}

&
\begin{minipage}{1in}
  \begin{equation*}
  \xymatrix{
  \bullet \driicell{P_2} \ar@/^1pc/[d]^{ }="0" \ar@/_1pc/[d]^{ }="1" \ar[r] & \bullet  \ar@/^1pc/[d]^{ }="2" \ar@{-->}@/_1pc/[d]^{ }="3" \\
  \bullet \ar[r] & \bullet \\
  \ar@{=>}"0" ;"1"
  \ar@{==>}"2" ;"3"
  }
  \end{equation*}
\end{minipage}
&
\begin{minipage}{1.5in}
  \begin{equation*}
  \xymatrix{
  \bullet \ar@/^1pc/[dd]^{ }="0" \ar@/_1pc/[dd]^{ }="1" \ar@/^1pc/[rr]^{ }="2" \ar@/_1pc/[rr]^{ }="3" & & \bullet \ar@/^1pc/[dd]^{ }="4" \ar@{-->}@/_1pc/[dd]^{ }="5" \\
   & \Lleftarrow^{T} & \\
  \bullet \ar@{-->}@/^1pc/[rr]^{ }="6" \ar@/_1pc/[rr]^{ }="7" & &
  \bullet \\ \ar@{=>}"0" ;"1" \ar@{=>}"2" ;"3" \ar@{==>}"4" ;"5"
  \ar@{==>}"6" ;"7" } \end{equation*}
\end{minipage}
  \\
\hline
\end{tabular}
\caption{The data of a \db\label{table:doublebicatdata}}
\end{table}

\begin{remark}
This shows the data of the bicategories $\Ob$, $\M$, and $\B$, each of
which has objects, morphisms, and 2-cells.  Note that the morphisms in
the three entries in the lower right hand corner---2-cells in $\M$,
and morphisms and 2-cells in $\B$---are not 2-dimensional.  The
2-cells in $\M$ and morphisms in $\B$ are the three-dimensional
``filling'' inside the illustrated cylinders, which each have two
square faces and two bigonal faces.

The 2-cells in $\B$ should be drawn 4-dimensionally.  The picture
illustrated can be thought of as taking both square faces of one
cylinder $P_1$ to those of another, $P_2$, by means of two other cylinders
($S_1$ and $S_2$, say), in such a way that $P_1$ and $P_2$ share their
bigonal faces.  This description works whether we consider the $P_i$
to be horizontal and the $S_j$ vertical, or vice versa.  These
describe the ``frame'' of this sort of morphism: the ``filling'' is the
4-dimensional track taking $P_1$ to $P_2$, or equivalently, $S_1$ to
$S_2$ (just as a square in a double category can be read horizontally
or vertically).  Not all relevant parts of the diagram have been
labelled here, for clarity.
\end{remark}

Next we want to describe a condition similar to the one we gave which
made it possible to think of a double category as a category.  In that
case, we got a condition which effectively allowed us to treat any
square as an identity, so that we only had objects and morphisms.
Here, we want a condition which lets us throw away the three entries
of table \ref{table:doublebicatdata} in the bottom right.  This
condition, when satisfied, should allow us to treat a {\db} as a \vdb.
It comes in two parts:

\begin{definition}\label{def:actionconds} We say that a {\db} satisfies
the \textbf{vertical action condition} if, for any morphism $M_1 \in
\M$ and 2-morphism $\alpha \in \Ob$ such that $s(M_1) = t(\alpha)$,
there is a morphism $M_2 \in \M$ and 2-morphism $P \in \M$ such that
$P$ fills the ``pillow diagram'':
\begin{equation}\label{eq:vertactcond}
  \xymatrix{
    x \ar[r]^{ }="1" \ar[d] \ar@/^2pc/[r]^{ }="0" & y \ar[d] \\
    x' \ar[r] \uriicell{M_1} & y' \\
    \ar@{=>}"0" ;"1"^{\alpha}
  } \qquad \ra_P \qquad
  \xymatrix{
      x \ar [r] \ar[d] & y \ar[d] \\
      x' \ar[r] \uriicell{M_2} & y'
  }
\end{equation} where $M_2$ is the back face of this diagram, and the
2-morphism in $\Ob$ at the bottom is the identity.

An {\db} satisfies the \textbf{horizontal action condition} if for any
morphism $M_1 \in \M$ and object $\alpha$ in $\B$ with $s(M_1) =
t(\alpha$ there is a morphism $M_2 \in \M$ and morphism $P \in \B$
such that $P$ fill the pillow diagram which is the same as
(\ref{eq:vertactcond}) turned sideways.
\end{definition}

Here, $M_2$ is the square which will eventually be named $M_1 \star_V
\alpha$ when we define an action of 2-cells on squares.

\begin{remark} One can easily this condition is analogous to our
filler condition (\ref{eq:hornfiller}) in a double category by turning
the diagram (\ref{eq:vertactcond}) on its side.  What the diagram says
is that when we have a square with two bigons---the top one arbitrary
and the bottom one the identity---there is another square $M_2$ (the
back face of a pillow diagram) and a filler 2-morphism $P \in \B$
which fills the diagram.  If one imagines turning this diagram on its
side and viewing it obliquely, one sees precisely
(\ref{eq:hornfiller}), as a dimension has been suppressed.  What is a
square in (\ref{eq:hornfiller}) is a cylinder (2-morphism in $\B$);
the roles of both squares and bigons in (\ref{eq:vertactcond}) are
played by arrows in (\ref{eq:hornfiller}); arrows in
(\ref{eq:vertactcond}) become pointlike objects in
(\ref{eq:hornfiller}).
\end{remark}

However, to get the compatibility between horizontal and vertical
actions, we need something more than this.  In particular, since these
involve both horizontal and vertical cylinders (3-dimensional
morphisms in the general sense), the compatibility condition must
correspond to the 4-dimensional 2-cells in $\B$, shown in the lower
right corner of Table \ref{table:doublebicatdata}.

To draw necessary condition is difficult, since the necessary diagram
is four-dimensional, but we can describe it as follows:

\begin{definition} We say a {\db} satisfies the
\textbf{action compatibility condition} if the following holds.
Suppose we are given \begin{itemize}
\item a morphism $F \in \M$
\item an object $\alpha \in \B$ whose target in $\M$ is a source of $F$
\item a 2-cell $\beta \in \Ob$ whose target morphism is a source of $F$
\item an invertible morphism $P_1 \in \B$ with $F$ as source, and the
      objects $\alpha$ and $\opname{id}$ in $\B$ as source and target
\item an invertible 2-cell $P_2 \in \M$ with $F$ as source, and the
      2-cells $\beta$ and $\opname{id}$ in $\M$ as source and target
\end{itemize} where $P_1$ and $P_2$ have, as targets, morphisms in
$\M$ we call $\alpha \star F$ and $\beta \star F$ respectively.  Then
there is a unique morphism $\hat{F}$ in $\M$ and 2-cell $T$ in $\B$
having all of the above as sources and targets.
\end{definition}

Geometrically, we can think of the unique 2-cell in $\B$ as resembling
the structure in the bottom right corner of Table
\ref{table:doublebicatdata}.  This can be seen as taking one
horizontal cylinder to another in a way that fixes the (vertical)
bigons on its sides, by means of a translation which acts on the front
and back faces with a pair of vertical cylinders (which share the top
and bottom bigonal faces).  Alternatively, it can be seen as taking
one vertical cylinder to another, acting on the faces with a pair of
horizontal cylinders.  In either case, the cylinders involved in the
translation act on the faces, but the four-dimensional interior, $T$,
acts on the original cylinder to give another.  The simplest
interpretation of this condition is that it is precisely the condition
needed to give the compatibility condition (\ref{eq:actionindep}).

\begin{remark}Notice that the two conditions given imply the existence
of unique data of three different sorts in our \db.  If these are the
only data of these kinds, we can effectively omit them (since it
suffices to know information about their sources and targets.  This
omission is part of a decategorification of the same kind we saw for
the double category $\catname{DC}$.
\end{remark}

In particular, we use the above conditions to show the following:

\begin{lemma}\label{lemma:doublebicat}Suppose $\catname{D}$ is a {\db}
which has at most a unique morphism or 2-morphisms in $\B$, and at
most a unique 2-morphism in $\M$, having any specified sources and
targets; and $\catname{D}$ satisfies the horizontal and vertical
action conditions and the action compatibility condition; then
$\catname{D}$ gives a {\vdb} in the sense of Verity.
\end{lemma}
\begin{proof}
$\catname{D}$ consists of bicategories $(\Ob,\M,\B)$ together with all
required maps (three kinds of source and target maps, two kinds of
identity, three partially-defined compositions, left and right
unitors, and the associator), satisfying the usual properties.  To
begin with, we describe how the elements of a {\vdb} $\catname{V}$
(definition \ref{def:doublebicat}) arise from this.

The horizontal bicategory $\Hor$ of $\catname{V}$ is simply $\Ob$.
The vertical bicategory $\Ver$ consists of the objects of each of
$\Ob$, $\M$, and $\B$, where the required source, target and
composition maps for $\Ver$ are just the object maps from those for
$\catname{D}$, which are all functors.  We next check that this is a
bicategory.

The source and target maps for $\Ver$ satisfy all the usual rules for a
bicategory since the corresponding functors in $\catname{D}$ do.
Similarly, the composition maps satisfy (\ref{eq:bicatcompos}),
(\ref{eq:bicatrunit}) and (\ref{eq:bicatlunit}) up to natural
isomorphisms: they are just object maps of functors which satisfy
corresponding conditions.  We next illustrate this for composition.

In $\catname{D}$, there is an associator 2-natural transformation.
That is, a partially defined weak 2-functor $\alpha : {\M}^3 \rightarrow
\B$ satisfying the pentagon identity (strictly, since we are
considering a \textit{strict model} of the theory of bicategories).
Among the data for $\alpha$ are the object maps, which give the maps
for the associator in $\Ver$.  Since the associator 2-natural
transformation satisfies the pentagon identity, so do these object
maps.  The other properties are shown similarly, so that $\Ver$ is a
bicategory.

Next, we declare that the squares of $\catname{V}$ are the morphisms
of $\M$.  Their vertical source and target maps are the morphism maps
from the source and target functors from $\M$ to $\Ob$.  Their
horizontal source and target maps are the internal ones in $\M$.
These satisfy equations (\ref{eq:squarestmaps}) because the source and
target maps of $\catname{D}$ are functors (in our special example of
cospans, this amounts to the fact that (\ref{xy:cspan2}) commutes).

The horizontal composition of squares (\ref{eq:squarehorizcomp}) is just
the composition of morphisms in $\M$.  Now, by assumption, $\M$ is a
bicategory with at most unique 2-morphisms having any given source and
target.  If we declare these are identities (that is, identify their
source and target morphisms), we get that horizontal composition is
exactly associative and has exact identities.

The vertical composition of squares (\ref{eq:squarevertcomp}) is given
by the morphism maps for the partially defined functor $\circ$ for
$\M$, and so composition here satisfies the axioms for a bicategory.
In particular, it has an associator and a unitor: but these must be
morphisms in $\B$ since we take the morphism maps from the associator
and unitor functors (and the theory of bicategories says that these give
2-morphisms).  But again, we can declare that there are only identity
morphisms in $\B$, and this composition is exactly associative.

The interchange rule (\ref{eq:squareinterchangelaw}) follows again from
functoriality of the composition functors.

The action of the 2-morphisms (bigons) on squares is guaranteed by the
horizontal and vertical action conditions.  In particular, by
composition of in $\M$ or $\B$, we guarantee the existence of the
categories of horizontal and vertical cylinders $\catname{Cyl_H}$ and
$\catname{Cyl_V}$, respectively.  These come from the 2-morphisms in
$\M$ or morphisms in $\B$ respectively which those conditions demand
must exist.  Taking these to be identities, the cylinders consist of
commuting cylindrical diagrams with two bigons and two squares.

In the case where one bigon is the identity, and the other is any
bigon $\alpha$, the conditions guarantee the existence of a cylinder,
which we have declared to be the identity.  This defines the effect of
the action of $\alpha$ on the square whose source is the target of
$\alpha$.  If this square is $F$, we denote the other square $\alpha
\star_H F$ or $\alpha \star_V F$ as appropriate.

The condition (\ref{eq:actioncompat}) guaranteeing independence of the
horizontal and vertical actions follows from the action compatibility
condition.  For suppose we have a square $F$ whose horizontal and
vertical source arrows are the targets of 2-cells $\alpha$ and
$\beta$, and attach to its opposite faces two identity 2-cells.  Then
the horizontal and vertical action conditions mean that there will be
a square $\alpha\star_H F$ and a square $\beta\star_V F$).  Then the
action compatibility condition applies (the $P_i$ are the identities
we get from the action condition), and there is a morphism in $\M$,
namely a square in $\catname{V}$ and a 2-cell $T \in
\B$.  Consider the remaining face, which the action condition suggests
we call $\alpha \star_H (\beta \star_V F)$ or $\beta \star_V (\alpha
\star_H F)$, depending on the order in which we apply them.  The
compatibility condition says that there is a unique square which fills
this spot so the two must be equal.

Now suppose we have three composable squares---that is, morphisms $F$,
$G$, and $H$ in $\M$, which are composable along shared source and
target objects in $\M$.  The associator functor has an object map,
giving objects in $\B$ at the ``top'' and ``bottom'' of the squares.
It also has a morphism map, giving morphisms in $\B$.  But by
assumption there is only a unique such map between , these associators
must be the unique morphism in $\B$ with source $(H \circ G) \circ F$
and target $H \circ (G \circ F)$.  Then by the vertical action
condition, we have a filler 2-morphism in $\M$ for the action on the
composite square by the top associator, and then, taking the result
and composing with the bottom associator, we get another filler.  This
must be the unique map between the two composites---which is the
identity, since they have the same sources and targets.  So we get a
commuting cylinder.  Composing squares along source and target
morphisms in $\Ob$ works the same way by a symmetric argument.

The condition (\ref{eq:unitaction}) is similar---the unitor functor
will give the unique morphism in $\B$, and the action compatibility
condition gives the commuting cylinder for unitors on the composite of
squares.

So from any such {\db} we get a \vdb.
\end{proof}

\begin{remark} It is interesting to note how these arguments apply to
the case when we are looking at constructions in $\iiCCosp$, as will
be the case in $nCob$.

In particular, the interchange rules hold because the middle objects
in the four squares being composed form the vertices of a new square.
The pushouts in the vertical and horizontal direction form the
middle objects of vertical and horizontal cospans over these.  The
interchange law means that the pushout (in the horizontal direction)
of the objects from the vertical cospans is in the same isomorphism
class as the pushout (in the vertical direction) of the objects from
the horizontal cospans.  This is true because of the universal property
of the pushout.

The horizontal and vertical 2-morphisms are maps of cospans, and act on
the squares by composition of morphisms in $\C$: given a square $M$
with four maps $P_i$ and $\Pi_i$ to the edges as in (\ref{xy:cspan2});
and a morphism of cospans on any edge (for definiteness, say the top),
where the $\C$-morphism in the middle is $S \ralim^{f}
\tilde{S}$.  Then the composite $f \circ P_1 : M \ra
\tilde{S}$ is a source (or target) map to the cospan $X
\ralim^{\iota_1} \tilde{S} \lalim^{\iota_2} Y$.  The result
is again a square.  In particular, composition of internal maps in
horizontal and vertical morphism of cospans with the projections in a
square are independent.
\end{remark}

\bibliography{TQFTpaper}
\bibliographystyle{acm}

\end{document}